\documentclass[a4paper, 11pt, twoside]{article}
\usepackage{lspaper}
\usepackage{xcolor}
\pdfoutput=1

\newcommand{\hlabel}{\phantomsection\label}

\newgeometry{margin=0.8in}

\DeclareMathOperator{\id}{id}

\DeclareMathOperator{\het}{ht}

\newcommand{\clapdots}{\; \mathclap{\dots} \;}

\renewcommand{\phi}{\varphi}

\newcommand{\s}[2]{s\hspace{-4pt}\underset{\scriptscriptstyle{#2}}{\overset{\scriptscriptstyle{#1}}{\downarrow}}}
\newcommand{\su}[2]{s\hspace{-4pt}\overset{\scriptscriptstyle{#2}}{\underset{\scriptscriptstyle{#1}}{\uparrow}}}

\begin{document}

\title{Decomposable Specht modules indexed by bihooks}
\author{\begin{tabular}{cc}Liron Speyer&Louise Sutton\\[3pt]
{\normalsize University of Virginia,}&{\normalsize National University of Singapore,}\\
{\normalsize Charlottesville,}&{\normalsize 10 Lower Kent Ridge Road,}\\
{\normalsize VA 22904,}&{\normalsize Singapore 119076}\\
{\normalsize USA}&{\normalsize }\\[3pt]
{\normalsize\texttt{\normalsize l.speyer@virginia.edu}}&{\normalsize\texttt{\normalsize matloui@nus.edu.sg}}
\end{tabular}}

\renewcommand\auth{Liron Speyer \& Louise Sutton}

\runninghead{Decomposable Specht modules indexed by bihooks}
\msc{20C30, 20C08, 05E10}

\toptitle

\begin{abstract}
We study the decomposability of Specht modules labelled by \emph{bihooks}, bipartitions with a hook in each component, for the Iwahori--Hecke algebra of type $B$.
In all characteristics, we determine a large family of decomposable Specht modules, and conjecture that these provide a complete list of decomposable Specht modules indexed by bihooks.
We prove the conjecture for small $n$.
\end{abstract}

\section{Introduction}

Specht modules are of fundamental importance in the study of reflection groups and their deformations.
We are particularly interested in the Iwahori--Hecke algebras of types $A$ and $B$.
In type $B$, these Hecke algebras have been studied from the point of view of their decomposition numbers~\cite{Fay06,aj10}, their Kazhdan--Lusztig theory~\cite{
gip08,Jac11}, and via applications from higher representation theory~\cite{bsIII}.
Nowadays, the subject often takes on a more diagrammatic and categorical flavour, for example with the inception of Elias and Williamson's diagrammatic Hecke category~\cite{EW16}, Webster's diagrammatic Cherednik algebra~\cite{Webster} and recent work of Elias--Losev~\cite{el17}.

It is known by \cite{dj91,rouq08,fs16} that the Specht modules are indecomposable if the \emph{quantum characteristic} $e$ is not $2$, and under the further assumption in type $B$ that the parameters $\kappa_1$ and $\kappa_2$ are distinct.
Rouquier's work in fact gives us that the Hecke algebras admit faithful quasi-hereditary covers, whence indecomposability follows easily by considering the trivial endomorphism spaces of standard modules.

In type $A$, Murphy~\cite{gm80} and the first author~\cite{ls14} completely determined the decomposability of Specht modules indexed by \emph{hook} partitions.
The general case is very difficult, owing to the complicated structure of the endomorphism rings of Specht modules. In the case of the symmetric group, Dodge and Fayers~\cite{df12} give the first new family of decomposable Specht modules in thirty years, which are indexed by partitions of the form $(a,3,1^b)$.
Parallel to this, the graded composition multiplicities for Specht modules indexed by hooks have been determined using Fock space machinery in~\cite{cmt04}.

Here, we take the natural first step in extending this study of decomposable Specht modules to Iwahori--Hecke algebras of type $B$.
We study Specht modules indexed by \emph{bihooks}, that is bipartitions for which both components are hook partitions.
In a certain subfamily of these, the second author has determined graded decomposition numbers~\cite{Sutton17I,Sutton17II}, drawing an analogue in type $B$ with the aforementioned work of \cite{cmt04}.
As in~\cite{ls14}, we study these Specht modules from the perspective of the cyclotomic Khovanov--Lauda--Rouquier algebras that were introduced by Khovanov--Lauda~\cite{kl09} and Rouquier~\cite{rouq}, an equivalent point of view by virtue of the isomorphism theorem of Brundan and Kleshchev~\cite{bkisom}.
In this framework, we investigate endomorphisms of Specht modules, and obtain non-trivial generalised eigenspace decompositions for several large families of Specht modules, which we conjecture are the only decomposable Specht modules indexed by bihooks if $e\neq 2$ and $\nchar \bbf \neq 2$ (see \cref{enot2conj,e2conj}).
In other words, our main results prove one direction of our conjectural classification in all of the cases where decomposable Specht modules may arise, and we prove our classification in full in a few cases.
For small $n$ or $e=2$, we have some extra decomposable Specht modules -- see \cref{smallbihooks,kejefore=2}, respectively.
We summarise the majority of our decomposable Specht modules as follows.

{\theoremstyle{plain}\newtheorem*{mainresult}{Theorem \ref{mainresult}}
\begin{mainresult}
Suppose that we take a Hecke algebra of type $B$ with parameters $\kappa_1 = \kappa_2$.
Let $\la = ((ke + a, 1^b),(je + a, 1^b))$ or $((b+1, 1^{je+a-1}), (b+1, 1^{ke+a-1}))$, for some $j,k\geq 1$, $0< a \leq e$ and $0 \leq b < e$ with $a+b \neq e$, or for $a=b=0$.
\begin{enumerate}[label=(\roman*)]
\item For $j,  k > 1$, if $j+k$ is even and $\nchar \bbf \neq 2$, or if $j+k$ is odd, then $\spe\la$ is decomposable.

\item If $j=1$ or $k=1$, then $\spe\la$ is decomposable if and only if $\nchar\bbf \nmid j+k$.
\end{enumerate}
\end{mainresult}

\newtheorem*{lev1tolev2}{Theorem \ref{lev1tolev2}}
\begin{lev1tolev2}
Let $e=2$, and suppose that $\mu$ is a hook partition of $n$ such that $\spe \mu$ is a decomposable Specht module over the Hecke algebra of type $A$ (cf.~\cref{lev1murph,lev1spey}).
Then, for any partition $\nu$ of $m$, the Specht modules $\spe{(\mu, \nu)}$ and $\spe{(\nu, \mu)}$ over the Hecke algebra of type $B$ are decomposable.
\end{lev1tolev2}}

We now outline the layout of this paper.
In \cref{sec:background}, we will collect all necessary definitions and background from the literature, before studying the case of `small bihooks' (when $n \leq 2e$) in \cref{sec:small}.
In this case, we are able to completely determine the decomposability of Specht modules: we prove the above results in this special case, and furthermore show the converse, that all other bihooks index \emph{indecomposable} Specht modules.
Our method for this converse is a case-by-case analysis examining the tableaux that can appear in endomorphisms of these Specht modules. We emphasise that this method will not readily extend to large $n$.
Next, we conduct the majority of our study of Specht modules labelled by bihooks in \cref{sec:generalbihooks}, finding the aforementioned families of decomposable Specht modules.
Our method here is to first use the divided power functors to reduce proving \cref{mainresult} to the case of bipartitions of the form $((ke),(je))$, and then determine certain endomorphisms for Specht modules in~\Cref{keje} indexed by these bipartitions.
We compute three eigenvalues for these endomorphisms, yielding at least two distinct eigenvalues in any characteristic (with the exception of characteristic 2 when $j+k$ is even), resulting in a generalised eigenspace decomposition for the Specht modules having at least two non-trivial summands.
\cref{sec:e=2} covers the $e=2$ situation, which makes use of previous work of the first author in~\cite{ls14} to yield quick results and prove \cref{lev1tolev2,kejefore=2}.
We leave some long technical calculations for \cref{sec:calc}, where the keen reader may find the grittier details of our work.

\begin{ack}
The authors are grateful to the organisers of the conference Representation Theory of Symmetric Groups and Related Algebras, National University of Singapore, which allowed them to collaborate closely on parts of this research.
The second author is supported by Singapore MOE Tier 2 AcRF MOE2015-T2-2-003, and thanks the Universities of Osaka and Virginia for hosting her visits, as well as the Japan Society for the Promotion of Science for financial support.
The authors would also like to thank both Chris Bowman and Matthew Fayers for their helpful comments, as well as for the use of Fayers's GAP package, which was used extensively for computations.
We thank the referee for their close reading of the paper, offering many helpful comments and corrections.
\end{ack}

\section{Background}\label{sec:background}

In this section we give an overview of KLR algebras, Specht modules labelled by bihooks, and the associated combinatorics.
Throughout, $\bbf$ will denote an arbitrary field.

\subsection{Lie theoretic notation}

Let $e\in \{2,3,\dots\} \cup \{\infty\}$, which we call the \emph{quantum characteristic}.
If $e<\infty$, then we set $I:=\bbz/{e\bbz}$, which we identify with the set $\{0, 1, \dots, e-1\}$, whereas if $e=\infty$, we set $I:=\bbz$.
We let $\Gamma$ be the quiver with vertex set $I$ and an arrow $i\rightarrow{i-1}$ for each $i\in{I}$.
If $e=\infty$, then $\Gamma$ is the quiver of type $A_{\infty}$, otherwise $\Gamma$ is of type $A_{e-1}^{(1)}$.

Following Kac's book~\cite{kac}, we recall standard notation for the Kac--Moody algebra associated to the generalised Cartan matrix $(a_{ij})_{i,j\in I}$.
We have simple roots $\{\alpha_i\mid{i\in{I}}\}$, fundamental dominant weights $\{\La_i \mid i \in I\}$, and the invariant symmetric bilinear form $(\:,\:)$ such that $(\alpha_i, \alpha_j) = a_{i,j}$ and $(\La_i, \alpha_j) = \delta_{ij}$, for all $i, j \in I$.
Let $Q_+:= \bigoplus_{i\in I} \bbz_{\geq 0} \alpha_i$ be the positive cone of the root lattice.
If $\alpha = \sum_{i\in I} c_i \alpha_i \in Q_+$, then we define the \emph{height of $\alpha$} to be $\het(\alpha) = \sum_{i \in I} c_i$.

An \emph{$e$-bicharge} is an ordered pair $\kappa = (\kappa_1,\kappa_2)\in{I^2}$.
We define its associated dominant weight $\La$ of level two to be $\La = \La_\kappa := \La_{\kappa_1} + \La_{\kappa_2}$.

\subsection{The symmetric group}

Let $\mathfrak{S}_n$ be the symmetric group on $n$ letters.
We let $s_1,\dots,s_{n-1}$ denote the standard Coxeter generators, where $s_i$ is the simple transposition $(i, i+1)$ for $1\leq i<n$.
We define a \emph{reduced expression} for a permutation $w\in\mathfrak{S}_n$ to be an expression $s_{i_1}\dots s_{i_m}$ such that $m$ is minimal, and call $m$ the length of $w$, denoted $\ell(w)$.

We define the \emph{Bruhat order} $\leq$ on $\mathfrak{S}_n$ as follows.
If $x, w \in \mathfrak{S}_n$, then we write $x\leq w$ if there is a reduced expression for $x$ which is a subexpression of a reduced expression for $w$.

For $1\leq i\leq j\leq n-1$, we define $\s{j}{i}:=s_js_{j-1}\dots s_i$ and $\su{i}{j}:=s_is_{i+1}\dots s_j$.

\subsection{Bipartitions}

A \emph{partition} $\la$ of $n$ is a weakly decreasing sequence of non-negative integers $\la = (\la_1, \la_2, \dots)$ such that $|\la| := \sum \la_i = n$.
We write $\varnothing$ for the \emph{empty partition} $(0, 0, \dots)$.
A \emph{bipartition} $\la$ of $n$ is a pair $\la = (\la^{(1)}, \la^{(2)})$ of partitions such that $|\la| = |\la^{(1)}| + |\la^{(2)}| = n$.
We refer to $\la^{(1)}$ and $\la^{(2)}$ as the \emph{$1$st  and $2nd$ component}, respectively, of $\la$.
We abuse notation and also write $\varnothing$ for the \emph{empty bipartition} $(\varnothing, \varnothing)$.
We denote the set of all bipartitions of $n$ by $\mptn 2 n$.

For $\la, \mu \in \mptn 2 n$, we say that $\la$ \emph{dominates} $\mu$, and write $\la \dom \mu$, if for all $k\geq 1$,
\[
\sum_{j=1}^{k} \la_j^{(1)} \geq \sum_{j=1}^{k} \mu_j^{(1)} \text{ and } |\la^{(1)}| + \sum_{j=1}^{k} \la_j^{(1)} \geq |\mu^{(1)}| + \sum_{j=1}^{k} \mu_j^{(1)}.
\]

The \emph{Young diagram} of $\la = (\la^{(1)}, \la^{(2)}) \in \mptn 2 n$ is defined to be
\[
[\la]:= \{ (i,j,m)\in \bbn \times \bbn \times \{1, 2\} \mid 1\leq j \leq \la_i^{(m)}\}.
\]
We refer to elements of $[\la]$ as \emph{nodes} of $\la$.
We draw the Young diagram of a bipartition as a column vector of Young diagrams $[\la^{(1)}], [\la^{(2)}]$.
We say that a node $A\in[\la]$ is \emph{removable} if $[\la]\setminus \{A\}$ is a Young diagram of a bipartition, while a node $A\not\in[\la]$ is \emph{addable} if $[\la]\cup\{A\}$ is a Young diagram of a bipartition.

If $\la$ is a partition, the \emph{conjugate partition}, denoted $\la'$, is defined by
\[
\la_i'=|\left\{j\geq 1\mid\la_j\geq i\right\}|.
\]
If $\la \in \mptn 2 n$, then we define the conjugate bipartition, also denoted $\la'$, to be $\la' = (\la^{(2)'},\la^{(1)'})$.

\subsection{Tableaux}

Let $\la \in \mptn 2 n$.
Then a \emph{$\la$-tableau} is a bijection $\ttt:[\la] \rightarrow \{1, \dots, n\}$.
We depict a $\la$-tableau $\ttt$ by inserting entries $1, \dots, n$ into the Young diagram $[\la]$ with no repeats; we let $\ttt(i,j,m)$ denote the entry lying in node $(i,j,m) \in [\la]$.
We say that $\ttt$ is \emph{standard} if its entries increase down each column and along each row, within each component, and denote the set of all standard $\la$-tableaux by $\std\la$.

The \emph{column-initial tableau} $\ttt_\la$ is the $\la$-tableau where the entries $1,\dots,n$ appear in order down consecutive columns, working from left-to-right, first in component $2$, then component $1$.

The symmetric group $\mathfrak{S}_n$ acts naturally on the left on the set of $\la$-tableaux.
For $\ttt$ a $\la$-tableau, we define the permutation $w_\ttt \in \mathfrak{S}_n$ by $w_\ttt \ttt_\la =\ttt$.

Suppose $\la \in \mptn 2 n$.
Let $\tts$ and $\ttt$ be $\la$-tableaux with corresponding reduced expressions $w_{\tts}$ and $w_{\ttt}$, respectively.
Then we say that \emph{$\ttt$ dominates $\tts$}, written as $\ttt \dom \tts$, if and only if $w_{\ttt}\geq w_{\tts}$.

\subsection{Residues and degrees}

Fix an $e$-bicharge $\kappa = (\kappa_1,\kappa_2)$.
The \emph{$e$-residue} of a node $A = (i,j,m) \in \mathbb{N}\times\mathbb{N}\times\{1,2\}$ is defined to be
\[
\res A := \kappa_m+j-i \pmod{e}.
\]
We call a node of residue $r$ an $r$-\emph{node}.

Let $\ttt$ be a $\la$-tableau.
If $\ttt(i,j,m) = r$, we set $\res_\ttt (r)=\res (i,j,m)$.
The \emph{residue sequence} of $\ttt$ is defined to be
\[
\bfi_\ttt = (\res_\ttt (1),\dots,\res_\ttt (n)).
\]
We denote the residue sequence of the column-initial tableau $\ttt_\la$ by $\bfi_\la:=\bfi_{\ttt_\la}$.

We now define the degree of a standard tableau, which is the \emph{codegree} as given in \cite[\S3.5]{bkw11}.
For $\la \in \mptn 2 n$ and an $i$-node $A$ of $\la$, we define
\begin{align*}
d^A(\la):&=
\#\left\{
\text{addable $i$-nodes of $\la$ strictly above $A$}
\right\}\\
&-\#\left\{
\text{removable $i$-nodes of $\la$ strictly above $A$}
\right\}.
\end{align*}

Let $\ttt \in \std\la$ with $\ttt^{-1}(n) = A$.
We define the \emph{degree} of $\ttt$, denoted $\deg(\ttt)$, recursively, by setting $\deg(\varnothing):=0$, and  
\[
\deg (\ttt):=d^A(\la)+\deg (\ttt_{<n}),
\]
where $\ttt_{<n}$ is the standard tableau obtained from $\ttt$ by removing the node $A$.

\subsection{Regular bipartitions}

Let $\la \in \mptn 2 n$.
We define the \emph{$i$-signature of $\la$} by reading the Young digram $[\la]$ from the top of the first component down to the bottom of the last component, writing a $+$ for each addable $i$-node and a $-$ for each removable $i$-node.
We obtain the \emph{reduced $i$-signature of $\la$} by successively deleting all adjacent pairs $+-$ from the $i$-signature of $\la$, always of the form $-\dots-+\dots+$.

The removable $i$-nodes corresponding to the $-$ signs in the reduced $i$-signature of $\la$ are called the \emph{normal} $i$-nodes of $\la$, while the addable $i$-nodes corresponding to the $+$ signs in the reduced $i$-signature of $\la$ are called the \emph{conormal} $i$-nodes of $\la$.
The lowest normal $i$-node of $[\la]$, if there is one, is called the \emph{good} $i$-node of $\la$, which corresponds to the last $-$ sign in the $i$-signature of $\la$.
Analogously, the highest conormal $i$-node of $[\la]$, if there is one, is called the \emph{cogood} $i$-node of $\la$, which corresponds to the first $+$ sign in the $i$-signature of $\la$.

We say that a bipartition $\la \in \mptn 2 n$ is \emph{regular}, or \emph{conjugate-Kleshchev}, if $[\la]$ can be obtained by successively adding cogood nodes to $\varnothing$.
That is, we have a sequence $\varnothing = \la(0), \la(1), \dots, \la(n) = \la$ such that $[\la(i)] \cup\{A\} = [\la(i+1)]$, where $A$ is a cogood node of $\la(i)$.
Equivalently, $\la$ is regular if and only if $\varnothing$ can be obtained by successively removing good nodes from $[\la]$.
Observe in level one that the set of all regular partitions coincides with the set of all $e$-regular partitions.

\subsection{Cyclotomic Khovanov--Lauda--Rouquier algebras}

Suppose $\alpha\in Q^+$ has height $n$, and set
\[
I^\alpha = \lset{\bfi = (i_1, i_2, \dots, i_n)\in I^n}{\alpha_{i_1}+\dots+\alpha_{i_n} = \alpha}.
\]
Recalling that $\La = \La_\kappa$, we define $\mathscr{R}_\alpha^\La$ to be the unital associative $\bbf$-algebra with generating set
\[
\lset{e(\bfi)}{\bfi \in I^\alpha}\cup\{y_1,\dots,y_n\}\cup\{\psi_1,\dots,\psi_{n-1}\}
\]
and relations
{\allowdisplaybreaks
\begin{align*}
e(\bfi)e(\bfj) &= \delta_{\bfi,\bfj} e(\bfi);\\
\sum_{i \in I^\alpha} e(\bfi) &= 1;\\
y_re(\bfi) &= e(\bfi)y_r;\\
\psi_r e(\bfi) &= e(s_r\bfi) \psi_r;\\
y_ry_s &= y_sy_r;\\
\psi_r y_s &= \mathrlap{y_s\psi_r}\hphantom{\smash{\begin{cases}(\psi_{r+1}\psi_r\psi_{r+1}+1)e(\bfi)\\\\\\\end{cases}}}\kern-\nulldelimiterspace\text{if } s\neq r,r+1;\\
\psi_r \psi_s &= \mathrlap{\psi_s\psi_r}\hphantom{\smash{\begin{cases}(\psi_{r+1}\psi_r\psi_{r+1}+1)e(\bfi)\\\\\\\end{cases}}}\kern-\nulldelimiterspace\text{if } |r-s|>1;\\
y_r \psi_r e(\bfi) &= (\psi_r y_{r+1} - \delta_{i_r,i_{r+1}})e(\bfi);\\
y_{r+1} \psi_r e(\bfi) &= (\psi_r y_r + \delta_{i_r,i_{r+1}})e(\bfi);\\
\psi_r^2 e(\bfi)&=\begin{cases}
\mathrlap0\phantom{(\psi_{r+1}\psi_r\psi_{r+1}+1)e(\bfi)}& \text{if }i_r=i_{r+1},\\
e(\bfi) & \text{if }i_{r+1}\neq i_r, i_r\pm1,\\
(y_{r+1} - y_r) e(\bfi) & \text{if }i_r = i_{r+1} + 1,\\
(y_r - y_{r+1}) e(\bfi) & \text{if }i_r = i_{r+1} - 1;
\end{cases}\\
\psi_r\psi_{r+1}\psi_re(\bfi) &= \begin{cases}
(\psi_{r+1}\psi_r\psi_{r+1}+1)e(\bfi)& \text{if }i_{r+2} = i_r = i_{r+1} +1,\\
(\psi_{r+1}\psi_r\psi_{r+1}-1)e(\bfi)& \text{if }i_{r+2} = i_r = i_{r+1} - 1,\\
(\psi_{r+1}\psi_r\psi_{r+1})e(\bfi)& \text{otherwise;}
\end{cases}\\
y_1^{(\La, \alpha_{i_1})} e(\bfi) &= 0;
\end{align*}
}for all admissible $r, s, \bfi, \bfj$.
When $e= 2$, we actually have slightly different `quadratic' and `braid' relations, which may be found, for example, in \cite[\S 3.1]{kmr}.
We omit them here, as we will not explicitly calculate with these relations when $e=2$.

\begin{lem}\cite[Corollary 1]{bk09}
There is a unique $\bbz$-grading on $\mathscr{R}_\alpha^{\Lambda}$ such that, for all admissible $r$ and $\bfi$,
\[
\deg(e(\bfi))=0,\quad\deg(y_r)=2,\quad\deg\psi_r(e(\bfi))=-a_{i_r,r_{r+1}}.
\]
\end{lem}

The \emph{cyclotomic Khovanov--Lauda--Rouquier (KLR) algebra} or \emph{cyclotomic quiver Hecke algebra} $\mathscr{R}_n^\La$ is defined to be the direct sum $\bigoplus_\alpha\mathscr{R}_\alpha^\La$, where the sum is taken over all $\alpha\in Q^+$ of height $n$.

Here we sum over all $\alpha \in Q^+$ of height $n$, though in fact only finitely many of the summands will be non-zero, so (even when $e=\infty$) $\mathscr{R}_n^\La$ is a unital algebra.

These $\bbz$-graded algebras are connected to the Hecke algebras of type $B$ via (a special case of) Brundan and Kleshchev's \emph{Graded Isomorphism Theorem}.

\begin{thm}\cite[Main Theorem]{bkisom}
If $e = \nchar(\bbf)$ or $\nchar(\bbf)\nmid{e}$, then $\mathscr{R}_n^\La$ is isomorphic to the integral Hecke algebra $\hhh(q,Q_1,Q_2)$ of type $B$ with parameters $q \in \bbf$ a primitive $e$th root of unity, $Q_1 = q^{\kappa_1}$, and $Q_2 = q^{\kappa_2}$.
That is, $\hhh(q,Q_1,Q_2)$ has generators $T_0,\dots, T_{n-1}$ satisfying type $B$ Coxeter relations, with the quadratic relations replaced with
\[
(T_0 - q^{\kappa_1}) (T_0 - q^{\kappa_2}) = 0 \quad \text{and} \quad (T_i-q) (T_i+1) = 0 \;\; \forall \; 1\leq i \leq n-1.
\]
\end{thm}

\subsection{Specht modules labelled by bihooks}

\begin{defn}
We call a bipartition $\la$ a \emph{bihook} if it is of the form $\la = ((a,1^b),((c,1^d))$ for some integers $a,c \geq1$ and $b,d \geq 0$.
\end{defn}

\begin{defn}\cite[Definition 7.11]{kmr}
Let $\la = ((a,1^b),((c,1^d)) \in \mptn 2 n$.
The (column) \emph{Specht module} $\spe\la$ is the cyclic $\mathscr{R}_n^\La$-module generated by $z_\la$ of degree $\deg(z_\la):= \deg(\ttt_\la)$ subject to the relations:
\begin{itemize}
\item $e(\bfi_\la)z_\la = z_\la$;

\item $y_r z_\la = 0$ for all $r \in \{1, \dots, n\}$;

\item $\psi_r z_\la = 0$ for all $r \in \{1, \dots,n-1\} \setminus \{d+1, c+d, b+c+d+1\}$;

\item $\psi_1\psi_2\dots\psi_{d+1} z_\la = 0 = \psi_{c+d+1} \psi_{c+d+2} \dots \psi_{b+c+d+1} z_\la$ (these are the Garnir relations arising from nodes $(1,1,2)$ and $(1,1,1)$, respectively).
\end{itemize}
\end{defn}

For each $w\in\mathfrak{S}_n$, we fix a reduced expression $w = s_{i_1} \dots s_{i_m}$ throughout.
We define the associated element of $\mathscr{R}_n^{\Lambda}$ to be $\psi_w:=\psi_{i_1}\dots{\psi_{i_m}}$, which, in general, depends on the choice of reduced expression for $w$.
For $\la \in \mptn 2 n$ and a $\la$-tableau $\ttt$, we define $v_{\ttt}:=\psi_{w_{\ttt}} z_\la$.

Whilst these vectors $v_{\ttt}$ of $\spe\la$ also depend on the choice of reduced expression in general, the following result does not.

\begin{thm}\cite[Corollary 4.6]{bkw11} and \cite[Proposition 7.14 and Corollary 7.20]{kmr}
For $\la \in \mptn 2 n$, the set of vectors $\{v_{\ttt}\mid{\ttt\in\std\la}\}$ is a homogeneous $\bbf$-basis of $\spe\la$, with $\deg(v_{\ttt})=\deg(\ttt)$.
Moreover, for any $\la$-tableau $\tts$, $v_{\tts}$ is a linear combination of basis elements $v_{\ttt}$ such that $\tts \dom \ttt$.
\end{thm}

We record the following useful lemma that we will use frequently.

\begin{lemc}{bkw11}{Lemma 4.4}\label{lem:tabresi}
Let $\la \in \mptn 2 n$, and $\ttt \in \std\la$.
Then $e(\bfi) v_\ttt = \delta_{\bfi,\bfi_\ttt} v_\ttt$.
\end{lemc}

Of particular importance to the present paper is the following result on the decomposability of Specht modules, which is a special case of a result for higher level cyclotomic KLR algebras.

\begin{propc}{fs16}{Corollary 3.12}\label{prop:indecomp}
If $e\neq 2$ and $\kappa_1 \neq \kappa_2$, then the Specht modules $\spe\la$ are indecomposable for all $\la \in \mptn 2 n$.
\end{propc}

The following useful result is obtained from~\cite[Theorems~7.25 and 8.5]{kmr}.

\begin{thm}\label{transdec}
$\spe\la$ is decomposable if and only if $\spe{\la'}$ is.
\end{thm}

We know from \cite[Theorem 5.10]{bk09} that Specht modules $\spe\la$ indexed by regular bipartitions have simple heads, yielding the following.

\begin{prop}\label{simphd}
If $\la \in \mptn 2 n$ is a regular bipartition, then the Specht module $\spe\la$ is indecomposable.
\end{prop}

\begin{lemc}{bkw11}{Lemma 4.9}\label{psiaction}
Let $\la \in \mptn 2 n$, $1 \leq r < n$, and $\ttt \in \std\la$.
If $r$ and $r+1$ lie in the same row or in the same column of $\ttt$, then
\[
\psi_r v_\ttt = \; \sum_{\mathclap{\substack{\tts \in \std\la\\\bfi_\tts = \bfi_{s_r\ttt}\\ \tts \domsby \ttt}}} \; a_\tts v_\tts \quad \text{for some } a_\tts \in \bbf.
\]
\end{lemc}

\begin{defn}
We define
\[
\psid x y = \psi_x \psi_{x-1} \dots \psi_{y} \quad \text{and} \quad \psiu y x = \psi_y \psi_{y+1} \dots \psi_x
\]
if $x \geq y$ and set both equal to $1_\bbf$ if $x<y$.
Furthermore, we use the shorthand
\begin{alignat*}{2}
\psid x y \chaind {x+1} {y+1} \clapdots \chaind {x+c} {y+c} &:= \psid x y \psid {x+1} {y+1} \dots \psid {x+c} {y+c}, \qquad
\psid x y \chaind {x-1} {y-1} \clapdots \chaind {x-c} {y-c} &&:= \psid x y \psid {x-1} {y-1} \dots \psid {x-c} {y-c},\\
\psiu x y \chainu {x+1} {y+1} \clapdots \chainu {x+c} {y+c} &:= \psiu x y \psiu {x+1} {y+1} \dots \psiu {x+c} {y+c},
\qquad
\psiu x y \chainu {x-1} {y-1} \clapdots \chainu {x-c} {y-c} &&:= \psiu x y \psiu {x-1} {y-1} \dots \psiu {x-c} {y-c}.
\end{alignat*}
\end{defn}

\section{Small bihooks}\label{sec:small}

In light of \Cref{prop:indecomp}, we suppose that $\kappa_1=\kappa_2$ throughout \cref{sec:small,sec:generalbihooks}.
In fact, we need only assume that $\kappa=(0,0)$ since residue shifts do not change the isomorphism type of $\mathscr{R}_n^{\Lambda}$.
Here, we begin our examination of Specht modules labelled by bihooks by completely determining which Specht modules are decomposable when $n\leq 2e$.
We first make the following easy observation.

\begin{lem}\label{smalltrivindecomp}
If $k<e$, then $\la = ((k),(k))$ is a regular bipartition.
In particular, $\spe\la$ is indecomposable.
\end{lem}

\begin{proof}
Starting from $(\varnothing,\varnothing)$, we may add two cogood $0$-nodes, then two cogood $1$-nodes, and continue in this fashion until we have added two cogood $(k{-}1)$-nodes.
The resulting bipartition is $\la$.
It follows from~\cref{simphd} that $\spe\la$ is indecomposable.
\end{proof}

\begin{rem}
In the above proof, we may go so far as adding two cogood $(e{-}2)$-nodes to reach the regular bipartition $((e-1),(e-1))$, but can go no further.
Adding a cogood $e$-node yields the bipartition $((e),(e-1))$, but adding a second cogood $e$-node results in the bipartition $((e,1),(e-1))$, \emph{not} $((e),(e))$.
One can check that $((k),(k))$ is not regular for any $k\geq e$.
\end{rem}

In fact, it is not difficult to see that we may generalise the previous lemma as follows.

\begin{lem}\label{smallrepeatedindecomp}
Suppose $a\geq1$ and $b\geq 0$ with $a+b < e$.
Then $\la = ((a,1^b),(a,1^b))$ is a regular bipartition, so $\spe\la$ is indecomposable.
\end{lem}

Conversely, we will next show that if we instead have $a+b = e$, then the Specht modules $\spe{((a,1^b),(a,1^b))}$ are all decomposable.

\begin{lem}\label{tableauswitch}
Suppose $a\geq1$ and $b\geq 0$ with $a+b \leq e$, and let $\la = ((a,1^b),(a,1^b))$.
Then there is a unique standard $\la$-tableau $\ttt \neq \ttt_\la$ with $\res \ttt = \bfi_\la$.
\end{lem}

\begin{proof}
This is an easy consequence of the fact that $\bfi_\la = (0, e-1, \dots, e-b, 1, 2, \dots, a-1, 0, e-1, \dots, e-b, 1, 2, \dots, a-1)$.
If we write $\ttt_\la = (\ttt^{(1)},\ttt^{(2)})$, then $\ttt = (\ttt^{(2)},\ttt^{(1)})$.
\end{proof}

\begin{lem}\label{smallendo}
Suppose $a \geq 1$ and $b \geq 0$ with $a+b = e$, and let $\la = ((a,1^b),(a,1^b))$.
There is an endomorphism $\phi$ of $\spe\la$ determined by $\phi(z_\la) = v_\ttt$, where $\ttt$ is the tableau in \cref{tableauswitch}.
\end{lem}

\begin{proof}
The proof proceeds by checking that the annihilator of $z_\la$ also annihilates $v_\ttt$.

We already know by Lemmas \ref{lem:tabresi} and \ref{tableauswitch} that $e(\bfi) v_\ttt = \delta_{\bfi, \bfi_\la} v_\ttt$.
It is easy to see that $\deg{\ttt} = 1 =\deg{\ttt_\la}$, and therefore that $\deg(y_r v_\ttt) = 3$.
However, $y_r v_\ttt \in e(\bfi_\la)\spe\la = \langle v_{\ttt_\la}, v_\ttt \rangle_\bbf$, so $\deg(y_r v_\ttt) = 1$.
This contradiction gives us that $y_r v_\ttt = 0$ for all $r \in \{1,\dots,n\}$.

Finally, we check the relations involving $\psi$ generators.
We know by \cref{psiaction} that $\psi_{w}v_\ttt=0$ if there exists no standard $\la$-tableau with residue sequence $w \bfi_\la$, for any $w\in\mathfrak{S}_n$.

Observe that for $1\leq{r}\leq{b}$,
\[
s_r \bfi_\la = (0, e-1, e-2, \dots, e-r+2, e-r, e-r+1, i_{r+2}, \dots, i_n).
\]
It is clear that there is no standard $\la$-tableau with residue sequence $s_r \bfi_\la$, and hence $\psi_r v_\ttt = 0$ for all $1\leq{r}\leq{b}$. Similarly, $\psi_r v_\ttt = 0$ for all $r\in \{b+2, \dots, e-1\}\cup\{e+1, \dots, e+b\}\cup\{e+b+2, \dots, 2e-1\}$.

If $a>1$, we must also check the longer Garnir relations arising from the Garnir nodes $(1,1,1)$ and $(1,1,2)$.
Applying $s_1s_2\dots s_{b+1}$ to $\bfi_\la$, we find that the residue sequence $s_1 s_2 \dots s_{b+1} \bfi_\la$ begins with $1$, and thus $\psi_1 \psi_2 \dots \psi_{b+1} z_\la = 0$.
Similarly, if we apply $s_{e+1}s_{e+2} \dots s_{e+b+1}$ to $\bfi_\la$, we find that the residue sequence $s_{e+1} s_{e+2} \dots s_{e+b+1} \bfi_\la$ contains $1$ in its $e$th and $(e{+}1)$th positions, and the $0$ residues occur in the $1$st and $(e{+}2)$nd positions.
Clearly there is no standard $\la$-tableau with this residue sequence, and thus $\psi_{e+1} \psi_{e+2} \dots \psi_{e+b+1} z_\la = 0$.
This completes the proof.
\end{proof}

\begin{prop}\label{smallphivt}
If $\phi(z_\la) = v_\ttt$ as above, then
$\phi(v_\ttt) = \begin{cases}
(-1)^{b+1} 2 v_\ttt &\text{if $a>1$;}\\
(-1)^{e-1} 2 v_\ttt &\text{if $\la=((1^e),(1^e))$.}
\end{cases}$
\end{prop}

We prove \cref{smallphivt} by a rather lengthy calculation, which we relegate to~\cref{sec:calc}.

\begin{thm}\label{2edecomps}
Suppose $a \geq 1$ and $b \geq 0$ with $a+b = e$, and let $\la = ((a,1^b),(a,1^b))$.
Then $\spe\la$ is decomposable if and only if $\nchar \bbf \neq 2$.
\end{thm}

\begin{proof}
If $\nchar \bbf \neq 2$, then there is an endomorphism $\phi'$ of $\spe \la$ determined by $\phi' = (-1)^{(b+1)} \frac{1}{2} \phi$, which is an idempotent.	
%
%
If $\nchar \bbf = 2$, then it is easy to see from \cref{tableauswitch,smallphivt} that $\spe\la$ has no non-trivial idempotent endomorphisms, and the result follows.
\end{proof}

Next, we will generalise \cref{smallrepeatedindecomp} and show that all Specht modules indexed by `small' bihooks not appearing in~\cref{2edecomps} are indecomposable.
More precisely:

\begin{thm}\label{smallindecomp}
Suppose $e\neq 2$, and let $\la = ((a,1^b),(c,1^d)) \in \mptn 2 n$ such that either $a+b+c+d<2e$ or $a+b+c+d=2e$ with $a\neq c$ or $b\neq d$.
Then $\spe\la$ is indecomposable.
\end{thm}

\begin{proof}
We determine all of the possible standard $\la$-tableaux with residue sequence
\[
\bfi_\la = (0,-1, \dots, -d, 1, 2, \dots, c-1, 0, -1, \dots, -b, 1, 2, \dots, a-1).
\]
If no standard $\la$-tableau other than $\ttt_\la$ has residue sequence $\bfi_\la$, then there exists no non-trivial endomorphism of $\spe\la$ and thus $\spe\la$ is indecomposable.
For any other standard $\la$-tableau $\ttt$ distinct from $\ttt_\la$ such that $\bfi_\ttt = \bfi_\la$, we show that there exists no non-trivial endomorphisms of $\spe\la$ in each of these cases.

We fill $[\la]$ with $1, \dots, n$ to find all standard $\la$-tableaux with residue sequence $\bfi_\la$.
\begin{enumerate}[label=(\roman*)]
\item\label{smallindecomp1}
Suppose $b < d < e$.
Since $n\leq 2e$, we must have that $b\leq e-2$.
First, we shall assume that $b<e-2$.
Since $b\neq e-1, e-2$, it follows that $-b-1 \not\equiv 0$ or $1 \pmod e$, so we cannot put $1, \dots, b+1$ down the first column of the first component, as $b+2$ would have nowhere to go.
Thus we must have $1,\dots, b+2$ down the first column of the second component, and in fact the only way to fill in the remaining entries yields $\ttt_\la$, so there are no non-trivial endomorphisms of $\spe\la$.

If instead we assume that $b=e-2$, then our conditions give $\la = ((a,1^{e-2}),(c,1^{e-1}))$, for which there are only 3 possible bihooks: $((1^{e-1}),(1^e))$, $((1^{e-1}),(2,1^{e-1}))$, and $((2,1^{e-2}),(1^e))$.
It is easy to check that the first has no standard tableaux of residue $\bfi_\la$ besides $\ttt_\la$, the second also has the standard tableau $\ttt = s_e \ttt_\la$, and the third has the standard tableau $\tts = \s{e}{1} \s{e+1}{2} \dots \s{2e-1}{e} \ttt_\la$.
Since $\psi_{e-1} v_\ttt = v_{s_{e-1} \ttt}$, and $\psi_{e-1} v_\tts = v_{s_{e-1} \tts}$ (plus possibly some lower order terms), while $\psi_{e-1} z_\la = 0$, these cases yield no non-trivial homomorphisms.

\item\label{smallindecomp2}
Suppose $b<d$ and $d\geq e$.
If $b= e-2$, then we must have $\la = ((1^{e-1}),(1^{e+1}))$, for which we may easily see that only $\ttt_\la$ and $s_{e+1} \ttt_\la$ have the correct residue sequence, and  that $\psi_e v_{s_{e+1} \ttt_\la} =  v_{s_e s_{e+1} \ttt_\la}$, and therefore there are no non-trivial homomorphisms.
So we may assume that $b<e-2$.
As in part~\ref{smallindecomp1}, we must place $1,\dots, b+2$ down the first column of the second component, and in fact must place $1,\dots, e-1$ in there.
If we place the entry $e$ in node $(e,1,2)$, we can either place $e+1$ in node $(e+1,1,2)$ or node $(1,1,1)$.
One can check that in the former case, we may only obtain $\ttt_\la$.
In the latter, we must then place $e+2,\dots, d+1$ down the first column of the first component, which is only possible if $b\geq d-e$.
If so, we may continue, placing $d+2,\dots, d+c$ in the first row of the second component, then $d+c+1,\dots, 2d+c-e+1$ down the remaining nodes in the second component, and $2d+c-e+2, \dots, n$ as in $\ttt_\la$.
Call this standard tableau $\ttr$.

If $a\geq c \geq 2$ we could have also placed $d+2,\dots, d+c$ in the first row of the first component, whence we are forced to place $d+c+1,\dots, 2d+c-e+1$ down the first column of the second component, then $2d+c-e+2,\dots, b+c+d+1$ down the first column of the first component, then $b+c+d+2, \dots, b+2c+d$ in the first row of the second component, then $b+2c+d+1,\dots, n$ as in $\ttt_\la$.
Call this standard tableau $\tts$.

Finally, suppose that $c\geq 2$, and we instead placed the entry $e$ in the node $(1,2,2)$.
Then we must put the entries $e+1, \dots, d+1$ down the first column of the first component, which is only possible if $b\geq d-e$.
We must then place $d+2$ in node $(e,1,2)$ and place $d+3, \dots, d+c$ along the first row of the second component, then $d+c+1,\dots, 2d+c-e+1$ down the first column of the second component, then $2d+c-e+2,\dots, n$ are filled as in $\ttt_\la$.
Call this standard tableau $\ttt$.

Then $\psi_e z_\la = 0$, while $\psi_e v_{\ttr} = v_{s_e \ttr}$, $\psi_e v_{\tts} = v_{s_e \tts}$, and $\psi_e v_{\ttt} = v_{s_e \ttt}$ (plus possibly some lower order terms in this final case).
It follows that there is no non-trivial endomorphism of $\spe\la$ in any of these cases.

\item\label{smallindecomp3}
Suppose $b\geq d$.
If $c<a$, the result follows from parts~\ref{smallindecomp1} and \ref{smallindecomp2} by applying \cref{transdec}.
So we may assume that $c\geq a$.
The only bihook for fitting these conditions with $d\geq e-1$ is $((1^e),(1^e))$, so we may further assume that $d\leq e-2$.
If we place $1,\dots, d+1$ down the first column of the second component, then $d+2,\dots, d+c$ must also be in the second component, if $c\leq e$.
If $b\neq e-1$, then the only option at this point is to fill in $d+c+1, \dots, n$ as in $\ttt_\la$.
If $b=e-1$ and $a \geq 2$, we may also obtain the standard tableau $\ttr = s_{b+c+d+1} \ttt_\la$.
If $c > e$ (in which case we also have $b<e-1$ and $c>a$), then we must fill $d+2,\dots, d+e$ in the second component, but can place $d+e+1,\dots, d+c$ along the first row of the first component, if $c \leq a+e$.
In this case, we may obtain a standard tableau $\ttu$ by filling in $d+c+1$ in node $(1,e+1,2)$, then $d+c+2,\dots, d+c+b+1$ down the first column of the first component, then $d+c+b+2, \dots, d+2c+b-e$ along the first row of the second component, with $d+2c+b-e+1, \dots, n$ being in the first component, as in $\ttt_\la$.
Since $\psi_{d+e} v_\ttu = v_{s_{d+e} \ttu}$ (plus possibly lower order terms), there is no homomorphism mapping $z_\la$ to $v_\ttu$.

If instead, we place $1,\dots, d+1$ down the first column of the first component, then $d+2, \dots, d+a$ must be placed in the first row of the first component.
If $c>a$, then we have nowhere left to put $d+a+1$, so there is no such standard tableau.
If $c=a$, then we may place $d+a+1, \dots, 2d+a+1$ down the first column of the second component, then $2d+a+2, \dots d+a+b+1$ down the first column of the first component, and the remaining entries along the first row of the second component, obtaining a standard tableau we shall call $\tts$.
If we further have that $b=e-1$, we also have the standard tableau $\ttt = s_{d+a+b+1} \tts$.
There are no other standard tableaux of the correct residue sequence.

We now show that there exists no homomorphism $\phi$ such that $\phi(z_\la)=\alpha v_{\ttr} +\beta v_{\tts} +\gamma v_{\ttt}$ for some $\alpha,\beta,\gamma\in\mathbb{F}$.
We observe that $\psi_{c+2d+1} v_\ttr = \psi_{c+2d+1} \psi_{c+d+e} z_\la = \psi_{c+d+e} \psi_{c+2d+1} z_\la = 0$, since $d+1 \leq e-1$, and we cannot have that $d=e-2$, $b=e-1$, $c\geq a \geq 2$, while $\psi_{c+2d+1} v_\tts = v_{s_{c+2d+1} \tts}$ (plus lower order terms), and $\psi_{c+2d+1} v_\ttt = v_{s_{c+2d+1} \ttt}$ (plus lower order terms).
Thus $\psi_{c+2d+1}(v_{\ttr}+v_{\tts}+v_{\ttt})
=v_{s_{c+2d+1} \tts}+v_{s_{c+2d+1} \ttt}$ (plus lower order terms), so that there exists no homomorphism mapping $z_\la$ to a linear combination of $v_{\tts}$ and $v_{\ttt}$.
Furthermore, $\psi_{c+d+e-1} v_\ttr = v_{s_{c+d+e-1} \ttr}$, whilst $\psi_{c+d+e-1} z_\la = \psi_{c+2d+1} z_\la = 0$. Hence there is no non-trivial endomorphism of $\spe\la$.
\qedhere
\end{enumerate}
\end{proof}

We summarise the results of this section in the following theorem.


\begin{thm}\label{smallbihooks}
Let $n\leq 2e$ and $\la \in \mptn 2 n$ be a bihook.
If $e\neq 2$, then $\spe\la$ is decomposable if and only if $\nchar \bbf \neq 2$, $n = 2e$ and $\la = ((a,1^b),(a,1^b))$ for some $a\geq 1$, $b\geq 0$.
If $e=2$, then $\spe\la$ is decomposable if and only if $\nchar \bbf \neq 2$ and $\la = ((2),(2))$, $((1^2),(1^2))$, $((2),(1^2))$ or $((1^2),(2))$.
\end{thm}

\begin{rem}
If $e=2$, the four decomposable Specht modules are pairwise isomorphic.
It is easy to check that no other small bihooks label decomposable Specht modules.
\end{rem}

\section{General bihooks}\label{sec:generalbihooks}

This section comprises our main body of work, where we determine several large families of decomposable Specht modules in any characteristic when $e$ is finite.
Throughout this section, we will assume that $\kappa=(0,0)$.
We begin by applying induction and restriction functors to reduce our proof to the case of Specht modules $S_{((ke),(je))}$.
We then show that these are decomposable by exploiting certain endomorphisms, for which we determine at least two distinct eigenvalues, so that the generalised eigenspace decompositions contain at least two summands.
The amalgamation of these results can be stated in our main theorem of this section as follows.

\begin{thm}\label{mainresult}
Let $\la = ((ke + a, 1^b),(je + a, 1^b))$ or $((b+1, 1^{je+a-1}), (b+1, 1^{ke+a-1}))$, for some $j,k\geq 1$, $0< a \leq e$ and $0 \leq b < e$ with $a+b \neq e$, or for $a=b=0$.
\begin{enumerate}[label=(\roman*)]
\item For $j,  k > 1$, if $j+k$ is even and $\nchar \bbf \neq 2$, or if $j+k$ is odd, then $\spe\la$ is decomposable.
		
\item If $j=1$ or $k=1$, then $\spe\la$ is decomposable if and only if $\nchar\bbf \nmid j+k$.
\end{enumerate}
\end{thm}

\begin{conj}\label{enot2conj}
When $e\neq 2$ and $\nchar \bbf \neq 2$, \cref{mainresult,smallbihooks} provides a complete list of decomposable Specht modules indexed by bihooks.
\end{conj}

\begin{rem}
We have checked \cref{enot2conj} in GAP for all $e\in\{3,4,5\}$, $\nchar\bbf\in\{0,3,5\}$, and $n\leq 22$.
In characteristic $2$, we have found an extra handful of decomposable Specht modules which our theorem and our methods do not detect -- namely those indexed by $((4e),(2e))$, $((2e),(4e))$, $((8e),(2e))$, $((2e),(8e))$, $((5e),(3e))$, $((3e),(5e))$, their conjugates, and those bihooks obtained from these by our induction functor arguments.
It is tempting to speculate that the extra decomposable Specht modules in characteristic 2 correspond to $\la$ as in \cref{mainresult} with $j\neq k$ and $j-k \equiv 2\mod 4$.
\end{rem}

We begin with a reduction result, greatly simplifying the work we must do to prove \cref{mainresult}.

\begin{prop}\label{inducetriv}
Let $k, j \geq 1$, and $0\leq a < e$.
The Specht module $\spe{((ke),(je))}$ is decomposable if and only if $\spe{((ke+a),(je+a))}$ is.
\end{prop}

\begin{proof}
Our argument is similar to that in \cite[Theorem 3.2]{ls14}, using the graded `cyclotomic divided power functors' of~\cite[\S4.6]{bk09}, which we denote here by $e_i^{(r)}$ and $f_i^{(r)}$.
Let $\la = ((ke), (je))$ and $\mu = ((ke+a), (je+a))$.
Then $e_0^{(2)} e_1^{(2)} \dots e_{a-1}^{(2)} \spe\mu = \spe\la$ and $f_{a-1}^{(2)} f_{a-2}^{(2)} \dots f_0^{(2)} \spe\la = \spe\mu$.
It follows that $\spe\la$ and $\spe\mu$ have the same composition length, and that $e_0^{(2)} e_1^{(2)} \dots e_{a-1}^{(2)} D \neq 0$ for any composition factor $D$ of $\spe\mu$.
Hence, by exactness, $e_0^{(2)} e_1^{(2)} \dots e_{a-1}^{(2)} M \neq 0$ for any submodule $M \subseteq \spe\mu$.
It follows that if $\spe\mu$ is decomposable, then so is $\spe\la$.
Repeating the argument the other way round completes the proof.
\end{proof}

In fact, we may extend the above \lcnamecref{inducetriv} as follows.

\begin{prop}\label{inducetrivfurther}
Let $k, j\geq 1$, $0 < a \leq e$, and $0 \leq b < e$ with $a+b \neq e$.
The Specht module $\spe{((ke),(je))}$ is decomposable if and only if $\spe{((ke+a,1^b), (je+a,1^b))}$ is.
\end{prop}

\begin{proof}
The argument is similar to the proof of \cref{inducetriv}.
By \cref{inducetriv}, we know that $\spe{((ke), (je))}$ is decomposable if and only if $\spe{((ke+1), (je+1))}$ is.

Let $\la = ((ke+1), (je+1))$ and $\mu = ((ke+a,1^b), (je+a,1^b))$.
If $a+b<e$, we find that $e_{1}^{(2)} e_{2}^{(2)} \dots e_{a-1}^{(2)} \cdot e_{e-1}^{(2)} e_{e-2}^{(2)} \dots e_{e-b}^{(2)} \spe\mu = \spe\la$ and $f_{e-b}^{(2)} f_{e-b+1}^{(2)} \dots f_{e-1}^{(2)} \cdot f_{a-1}^{(2)} f_{a-2}^{(2)} \dots f_1^{(2)} \spe\la = \spe\mu$.

If $a+b > e$, we find that 
\[
e_{1}^{(2)} e_{2}^{(2)} \dots e_{a-2}^{(2)} \cdot e_{e-1}^{(2)} e_{e-2}^{(2)} \dots e_{a}^{(2)} \cdot e_{a-1}^{(4)} \cdot e_{a-2}^{(2)} e_{a-3}^{(2)} \dots e_{e-b}^{(2)} \spe\mu = \spe\la
\]
and 
\[
f_{e-b}^{(2)} \dots f_{a-3}^{(2)} f_{a-2}^{(2)} \cdot f_{a-1}^{(4)} \cdot f_{a}^{(2)} f_{a+1}^{(2)} \dots f_{e-1}^{(2)} \cdot f_{a-2}^{(2)} f_{a-3}^{(2)} \dots f_1^{(2)} \spe\la = \spe\mu
\]
where we adopt the convention that $f_{a}^{(2)} f_{a+1}^{(2)} \dots f_{e-1}^{(2)} = \id = e_{e-1}^{(2)} e_{e-2}^{(2)} \dots e_{a}^{(2)}$ if $a=e$.

We may now complete the proof identically to \cref{inducetriv}.
\end{proof}

The following result handles some indecomposable Specht modules when $n$ is reasonably small, essentially extending \cref{smallindecomp}.

\begin{prop}\label{repeatedbihooks3}
Let $1 \leq a \leq e$, and $0 \leq b < e$ with $a+b \neq e$, and let $\la = ((a,1^b),(a,1^b))$.
Then $\spe\la$ is indecomposable.
\end{prop}

\begin{proof}
It is easy to see that $\la$ is regular, from which the result also follows: starting with $\varnothing$, we may add two cogood $0$-nodes, followed by two cogood $1$-nodes, and so on up to adding two cogood $(a-1)$-nodes, then two cogood $(-1)$-nodes, and so on up to adding two cogood $(-b)$-nodes.
If $n>2e$, then after adding the first $2e-2$ cogood nodes in the above, we must add four cogood nodes of the next residue before reverting to adding two at a time.
\end{proof}

In view of the above results, we may assume that $\la = ((ke),(je))$ as we prove our main result.
We fix this choice of $\la$ for the remainder of the section, and compute an endomorphism of $\spe\la$ which we will use to prove our result.

We now introduce notation for the basis vectors $v_{\ttt}$ of $\spe\la$, analogous to \cite[\S5.2]{Sutton17I}.
Observe that a standard $\la$-tableau $\ttt$ is determined by the entries $a_r:=\ttt(1,r,2)$ lying in its second component, for all $r \in \{1, \dots, je\}$.
We can thus write $\ttt = w_\ttt \ttt_\la$, where
\[
w_\ttt = \s{a_1-1}{1} \s{a_2-1}{2} \dots \s{a_{je}-1}{je}\in\mathfrak{S}_n.
\]
It follows that $v_{\ttt} = \psi_\ttt z_\la$ where
\[
\psi_\ttt = \psid{a_1-1}{1} \psid{a_2-1}{2} \dots \psid{a_{je}-1}{je} \in \mathscr{R}_n^\La.
\]
In order to distinguish our standard tableaux compactly, we will often write $v(a_1, a_2, \dots, a_{je})$ for the standard $\la$-tableau with entries $a_1, a_2, \dots, a_{je}$ in the second component.

\begin{defn}
Let $\ttt \in \std\la$.
An \emph{$e$-brick} is a sequence of $e$ adjacent nodes containing entries $je+1, je+2, \dots, (j+1)e$ for $j\geq 0$.
We say that $\ttt$ is an \emph{$e$-brick tableau} if all entries of $\ttt$ lie in $e$-bricks.
We denote the set of all standard $e$-brick $\la$-tableaux by $\calt_e$.
\end{defn}

\begin{eg}
If $e=3$ and $\la = ((6),(6))$ then $\calt_e$ consists of the following six tableaux, obtained by permuting the four $e$-bricks.

\[
\begin{tikzpicture}
\node[align=left] at (-0.7,0.2) {$\ttt_1=$};
\tyoung(0cm,0cm,789<10><11><12>,,123456)
\Yfillopacity{0}
\Ylinethick{2pt}
\tgyoung(0cm,0cm,_3_3,,_3_3)
\node[align=left] at (4.1,0.2) {$\ttt_2=$};
\Ylinethick{0.5pt}
\tyoung(4.8cm,0cm,456<10><11><12>,,123789)
\Ylinethick{2pt}
\tgyoung(4.8cm,0cm,_3_3,,_3_3)
\node[align=left] at (8.9,0.2) {$\ttt_3=$};
\Ylinethick{0.5pt}
\tyoung(9.6cm,0cm,456789,,123<10><11><12>)
\Ylinethick{2pt}
\tgyoung(9.6cm,0cm,_3_3,,_3_3)
\node[align=left] at (-0.7,-2) {$\ttt_4=$};
\Ylinethick{0.5pt}
\tyoung(0cm,-2.2cm,123<10><11><12>,,456789)
\Ylinethick{2pt}
\tgyoung(0cm,-2.2cm,_3_3,,_3_3)
\node[align=left] at (4.1,-2) {$\ttt_5=$};
\Ylinethick{0.5pt}
\tyoung(4.8cm,-2.2cm,123789,,456<10><11><12>)
\Ylinethick{2pt}
\tgyoung(4.8cm,-2.2cm,_3_3,,_3_3)
\node[align=left] at (8.9,-2) {$\ttt_6=$};
\Ylinethick{0.5pt}
\tyoung(9.6cm,-2.2cm,123456,,789<10><11><12>)
\Ylinethick{2pt}
\tgyoung(9.6cm,-2.2cm,_3_3,,_3_3)
\end{tikzpicture}
\]
These tableaux correspond to $v_{\ttt_1}=v(1,2,3,4,5,6)$, $v_{\ttt_2}=v(1,2,3,7,8,9)$, $v_{\ttt_3}=v(1,2,3,10,11,12)$, $v_{\ttt_4}=v(4,5,6,7,8,9)$, $v_{\ttt_5}=v(4,5,6,10,11,12)$ and $v_{\ttt_6}=v(7,8,9,10,11,12)$ in $\spe{((6),(6))}$.

\end{eg}

The following easy lemma is our motivation for introducing this definition.

\begin{lem}\label{lem:eis}
For any $\ttt \in \std\la$, $v_\ttt \in e(\bfi_\la) \spe\la$ if and only if $\ttt \in \calt_e$.
\end{lem}

In particular, this lemma tells us that for any endomorphism $\phi$ of $\spe\la$,
\[
\phi(z_\la) = \sum_{\ttt \in \calt_e} a_\ttt v_\ttt \quad \text{for some } a_\ttt\in\bbf.
\]

\begin{lem}
For all $\ttt \in \calt_e$, $\deg\ttt = j$.
\end{lem}

\begin{proof}
Any nodes in the first component of $\ttt$ cannot contribute to the degree, since there can't be any nodes above them.
In the second component, each $e$-brick contributes $+1$ to the degree.
If the first component is empty when adding such a brick, this comes from only having an addable $0$-node higher up in the diagram.
Otherwise, there is an addable $0$-node, as well as one addable $(e{-}1)$-node and one removable $(e{-}1)$-node.
\end{proof}

\begin{lem}\label{lem:ys}
For all $\ttt \in \calt_e$ and $1\leq r \leq n$, $y_r v_\ttt = 0$.
\end{lem}

\begin{proof}
Since $v_\ttt$ has degree $j$, $\deg(y_r v_\ttt) = j+2$.
But $y_r v_\ttt \in e(\bfi_\la)\spe\la$, so $\deg(y_r v_\ttt) = j$.
This contradiction yields the result.
\end{proof}

Similarly, the following result computes the actions of many $\psi_r$ generators on $v_\ttt \in e(\bfi_\la)\spe\la$.

\begin{lem}\label{lem:psis}
For all $\ttt \in \calt_e$ and $1\leq r < n$ with $r \not\equiv 0 \pmod e$, $\psi_r v_\ttt = 0$.
\end{lem}

\begin{proof}
We know that
\[
\psi_r v_\ttt = \; \sum_{\mathclap{\substack{\ttt \in \std\la\\ \bfi_\tts = \bfi_{s_r \ttt}}}} \; a_\tts v_\tts \quad \text{for some } a_\tts \in \bbf,
\]
and the result follows since no standard $\la$-tableau can have residue sequence $\bfi_{s_r \ttt}$.
\end{proof}

In order to calculate an endomorphism of $\spe\la$, it remains to understand the action of the generators $\psi_{re}$ on basis vectors $v_\ttt$.
In general, $\psi_{re}$ does not annihilate $v_\ttt$.
We will find an endomorphism of $\spe\la$ which maps $z_\la$ to a \emph{linear combination} of elements $v_\ttt$ which we will show is annihilated by $\psi_{re}$ if $r\neq j$.
First, we will introduce some necessary notation for working with tableaux in $\calt_e$.

For any $\ttt \in \calt_e$, we number the $e$-bricks in the order of their entries, i.e.~$\ttt$ comprises of bricks $1, 2, \dots, j+k$.
Then we have brick transpositions and their corresponding $\psi$ expressions, which we will denote by $\Psi_r$.
In particular, the brick transposition which transposes the $r$th and $(r+1)$th bricks corresponds to
\[
\Psi_r = \psid{re}{(r-1)e+1} \chaind{re+1}{(r-1)e+2} \clapdots \chaind{(r+1)e-1}{re}.
\]
As with our $\psi$ generators, we introduce the shorthand $\Psid xy = \Psi_x \Psi_{x-1} \dots \Psi_y$ and $\Psiu yx = \Psi_y \Psi_{y+1} \dots \Psi_x$.

Note that for any $\ttt \in \calt_e$, $w_\ttt$ is fully commutative since the reading word is 321-avoiding.
We can write $v_{\ttt}$ in terms of the $e$-bricks lying in the second component of $\ttt$, i.e.~we may write $v_\ttt = v(B_{i_1}, B_{i_2},\dots, B_{i_j})$, for $1 \leq i_1 < i_2 < \dots < i_j \leq j+k$.
Then
\[
v_\ttt = \Psid{i_1-1}{1} \Psid{i_2-1}{2} \dots \Psid{i_j-1}{j} z_\la.
\]
Analogously, we write $\underline{v}(B_{i_1}, B_{i_2}, \dots, B_{i_k})$ for the standard basis vector of $\spe\la$ indexed by the standard $\la$-tableau that has the $e$-bricks $B_{i_1}, B_{i_2}, \dots, B_{i_k}$ lying in its \emph{first} component.

\begin{eg}
As in the previous example, let $e=3$ and $\la = ((6),(6))$.
Then, for example,
\[
\Psi_1 = \psid{3}{1} \psid{4}{2} \psid{5}{3} = \psi_3 \psi_2 \psi_1 \psi_4 \psi_3 \psi_2 \psi_5 \psi_4 \psi_3.
\]
The six tableaux in $\calt_e$, given in the previous example, are determined by which two bricks (from the available bricks $1,2,3,4$) are in the second component, and correspond to the basis elements
\begin{alignat*}3
&z_\la, \qquad\qquad && \Psi_2 z_\la, \qquad\qquad & \Psid32 z_\la &= \Psi_3 \Psi_2 z_\la,\\
\Psi_1 &\Psi_2 z_\la, \qquad\qquad & \Psi_1 \Psid32 z_\la &= \Psi_1 \Psi_3 \Psi_2 z_\la, \qquad\qquad & \Psid21 \Psid32 z_\la &= \Psi_2 \Psi_1 \Psi_3 \Psi_2 z_\la.
\end{alignat*}
\end{eg}

The following lemma is easy to see from the relations in the KLR algebras and their Specht modules, and we will use it frequently without reference.

\begin{lem}
\begin{enumerate}[label=(\roman*)]
\item If $|r-s|>1$, then $\Psi_r \Psi_s = \Psi_s \Psi_r$.

\item If $\la =((ke),(je))$, for some $j,k \geq1$ and $r\neq j$, then $\Psi_r z_\la = 0$.
\end{enumerate}
\end{lem}

The following proposition will be crucial in our computations, particularly when $e>2$.
For $e=2$, it follows from the proof of \cite[Lemma 5.5]{ls14}, where the first author proves the result for hook partitions in level one, assuming that the residue sequence $\bfi_\la$ alternates between $0$ and $1$.
Its proof for $e>3$ is long and technical, requiring many preliminary lemmas, and we relegate this to~\cref{sec:calc}.

\newpage

\begin{prop}\label{prop:cancellation}
Suppose that $e\in\{2,3,\dots\}$, and let $\la = ((ke),(je))$, for some $j, k \geq 1$.
If $v \in e(\bfi_\la) \spe\la$ and $1\leq r \leq j+k-1$, then
\begin{enumerate}[label=(\roman*)]
\item\label{prop:cancellation1} $\psi_{re} \Psi_r v = -2 \psi_{re} v$;

\item\label{prop:cancellation2} for $r<j+k-1$, $\psi_{re} \Psi_{r+1} \Psi_r v = \psi_{re} v$;

\item\label{prop:cancellation3} for $r>1$, $\psi_{re} \Psi_{r-1} \Psi_r v = \psi_{re} v$.
\end{enumerate}
\end{prop}

%
%
%
%
%
%
%
%

\begin{thm}\label{keje}
Suppose that $k\geq j \geq 1$.
\begin{enumerate}[label=(\roman*)]
\item\label{keje1} Let $\la=((ke),(je))$.
Then there is an endomorphism of $\spe\la$ defined by
\[
\phi(z_\la) = \;
\sum_{\mathclap{\substack{0\leq i \leq k-1\\0 \leq l \leq j-1}}} \;
(j-l)(k-i)\Psiu{j-l}{j-1}\Psid{j+i}{j} z_\la.
\]

\item\label{keje2} Let $\la=((je),(ke))$.
Then there is an endomorphism of $\spe\la$ defined by
\[
\phi(z_\la) = \;
\sum_{\mathclap{\substack{0\leq i \leq j-1\\0 \leq l \leq k-1}}} \;
(k-l)(j-i)\Psiu{k-l}{k-1}\Psid{k+i}{k} z_\la.
\]

\end{enumerate}
\end{thm}

\begin{proof}
\begin{enumerate}[label=(\roman*)]
\item 
By \cref{lem:eis,lem:ys,lem:psis}, we just need to show that $\psi_{re}\phi(z_\la) = 0$, for all $r\in\{1,\dots,j-1\}\cup\{j+1,\dots,k+j-1\}$.
So we shall fix $r$ and look at the action of $\psi_{re}$ on each summand of $\phi(z_\la)$.
We will use \cref{prop:cancellation} many times in this proof, without further reference.

We first suppose that $r\in\{1,\dots,j-1\}$.
If $j-l<r-1$, we observe that
\[
\psi_{re}\Psiu{j-l}{j-1}\Psid{j+i}{j} z_\la
=\psi_{re}\Psi_{j-l}\Psi_{j-l+1} \dots\Psi_{r-2}
\cancel{\Psi_{r-1}\Psi_{r}} \Psi_{r+1} \Psi_{r+2} \dots \Psi_{j-1} \Psid{j+i}{j} z_\la = 0.
\]
Similarly, if $j-l>r+1$, we have
\[
\psi_{re}\Psiu{j-l}{j-1}\Psid{j+i}{j} z_\la
=\Psiu{j-l}{j-1}\Psid{j+i}{j}\psi_{re} z_\la=0.
\]

If $j-l = r-1$, we obtain the reduced expression
\[
\psi_{re} \Psi_{r-1} \Psi_{r}
\Psi_{r+1}\Psi_{r+2}\dots\Psi_{j-1}\Psid{j+i}{j} z_\la
=\psi_{re}
\Psi_{r+1}\Psi_{r+2}\dots\Psi_{j-1}\Psid{j+i}{j} z_\la.
\]
If $j-l=r$, we obtain the reduced expression
\[
\psi_{re}\Psi_{r}
\Psi_{r+1}\Psi_{r+2}\dots\Psi_{j-1}\Psid{j+i}{j} z_\la
=-2\psi_{re}
\Psi_{r+1}\Psi_{r+2}\dots\Psi_{j-1}\Psid{j+i}{j} z_\la.
\]
If $j-l=r+1$, we immediately obtain the reduced expression
\[
\psi_{re}\Psi_{r+1}\Psi_{r+2}\dots\Psi_{j-1}\Psid{j+i}{j} z_\la.
\]

Thus the only summands of $\phi(z_\la)$ which are not killed by $\psi_{re}$ are those corresponding to $j-l \in \{r-1, r, r+1\}$, and for a fixed $i$ all three yield the same basis vector, so we must check the coefficients to show that they cancel.
If $r>1$, then the corresponding coefficients are $(r-1)(k-i)$, $-2r(k-i)$, and $(r+1)(k-i)$, respectively, so they cancel.
If $r=1$, we do not have a term corresponding to $j-l=r-1$, so we only have the latter two terms, which clearly cancel.

We now suppose that $r\in\{j+1, \dots, j+k-1\}$.
If $j+i < r-1$, we have
\[
\psi_{re}\Psiu{j-l}{j-1}
\Psi_{j+i}\Psi_{j+i-1}\dots\Psi_{j} z_\la
=\Psiu{j-l}{j-1}
\Psi_{j+i}\Psi_{j+i-1}\dots\Psi_{j}\psi_{re} z_\la
=0.
\]
Similarly, if $j+i > r+1$, we have
\[
\psi_{re} \Psiu{j-l}{j-1}
\Psid{j+i}{r+2} \Psi_{r+1} \Psi_r
\Psid{r-1}{j} z_\la
= \psi_{re} \Psiu{j-l}{j-1}
\Psid{j+i}{r+3}
\Psid{r-1}{j} \Psi_{r+2} z_\la
= 0.
\]

If $j+i=r-1$, we see that the following expression is reduced
\[
\psi_{re} \Psiu{j-l}{j-1}
\Psi_{r-1} \Psi_{r-2} \dots \Psi_{j} z_\la.
\]
If $j+i=r$, we obtain the reduced expression
\[
\psi_{re} \Psiu{j-l}{j-1}
\Psi_{r} \Psi_{r-1} \dots \Psi_{j} z_\la
=-2 \psi_{re} \Psiu{j-l}{j-1}
\Psi_{r-1} \dots \Psi_{j} z_\la.
\]
If $j+i=r+1$, we obtain the reduced expression
\[
\psi_{re}\Psiu{j-l}{j-1}
\Psi_{r+1} \Psi_{r} \Psi_{r-1} \Psi_{r-2} \dots \Psi_{j} z_\la
= \psi_{re} \Psiu{j-l}{j-1}
\Psi_{r-1} \Psi_{r-2} \dots \Psi_{j} z_\la.
\]

As in the previous case, it is an easy check to verify that the coefficients $(j-l)(k-r+j+1)$, $-2(j-l)(k-r+j)$ and $(j-l)(k-r+j-1)$ cancel.
We note that here $r = j+k-1$ is the exceptional case, for which there is no term corresponding to $j+i = r+1$, but the $j+i = r-1$ and $j+i = r$ terms contribute $2(j-l)$ and $-2(j-l)$, respectively, so that the terms still cancel.

\item Similar to the first part.\qedhere
\end{enumerate}
\end{proof}

\begin{eg}
Let $e=3$, $\kappa=(0,0)$ and $\la=((9),(9))$.
Then we have the following endomorphism of $\spe\la$
\begin{align*}
\phi(z_\la)
= \; 
&9 \Psi_3 z_\la + 6 \Psi_4 \Psi_3 z_\la + 3 \Psi_5 \Psi_4 \Psi_3 z_\la + 6 \Psi_2 \Psi_3 z_\la + 4 \Psi_2 \Psi_4 \Psi_3 z_\la\\
&+ 2 \Psi_2 \Psi_5 \Psi_4 \Psi_3 z_\la
+ 3 \Psi_1 \Psi_2 \Psi_3 z_\la + 2 \Psi_1 \Psi_2 \Psi_4 \Psi_3 z_\la + \Psi_1 \Psi_2 \Psi_5 \Psi_4 \Psi_3 z_\la,
\end{align*}

where the summands correspond to the tableaux

\[
\begin{tikzpicture}
\node at (0,0.2) [text width = 0.8cm] {9} ;
\tyoung(0cm,0cm,789<13><14><15><16><17><18>,,123456<10><11><12>)
\Yfillopacity{0}
\Ylinethick{2pt}
\tgyoung(0cm,0cm,_3_3_3,,_3_3_3)
\Ylinethick{0.5pt}
\node at (5.3,0.2) [text width = 0.8cm] {6} ;
\tyoung(5.3cm,0cm,789<10><11><12><16><17><18>,,123456<13><14><15>)
\Yfillopacity{0}
\Ylinethick{2pt}
\tgyoung(5.3cm,0cm,_3_3_3,,_3_3_3)
\Ylinethick{0.5pt}
\node at (10.6,0.2) [text width = 0.8cm] {3} ;
\tyoung(10.6cm,0cm,789<10><11><12><13><14><15>,,123456<16><17><18>)
\Yfillopacity{0}
\Ylinethick{2pt}
\tgyoung(10.6cm,0cm,_3_3_3,,_3_3_3)
\Ylinethick{0.5pt}
\node at (0,-2) [text width = 0.8cm] {6} ;
\tyoung(0cm,-2.2cm,456<13><14><15><16><17><18>,,123789<10><11><12>)
\Yfillopacity{0}
\Ylinethick{2pt}
\tgyoung(0cm,-2.2cm,_3_3_3,,_3_3_3)
\Ylinethick{0.5pt}
\node at (5.3,-2) [text width = 0.8cm] {4} ;
\tyoung(5.3cm,-2.2cm,456<10><11><12><16><17><18>,,123789<13><14><15>)
\Yfillopacity{0}
\Ylinethick{2pt}
\tgyoung(5.3cm,-2.2cm,_3_3_3,,_3_3_3)
\Ylinethick{0.5pt}
\node at (10.6,-2) [text width = 0.8cm] {2} ;
\tyoung(10.6cm,-2.2cm,456<10><11><12><13><14><15>,,123789<16><17><18>)
\Yfillopacity{0}
\Ylinethick{2pt}
\tgyoung(10.6cm,-2.2cm,_3_3_3,,_3_3_3)
\Ylinethick{0.5pt}
\node at (0,-4.2) [text width = 0.8cm] {3} ;
\tyoung(0cm,-4.4cm,123<13><14><15><16><17><18>,,456789<10><11><12>)
\Yfillopacity{0}
\Ylinethick{2pt}
\tgyoung(0cm,-4.4cm,_3_3_3,,_3_3_3)
\Ylinethick{0.5pt}
\node at (5.3,-4.2) [text width = 0.8cm] {2} ;
\tyoung(5.3cm,-4.4cm,123<10><11><12><16><17><18>,,456789<13><14><15>)
\Yfillopacity{0}
\Ylinethick{2pt}
\tgyoung(5.3cm,-4.4cm,_3_3_3,,_3_3_3)
\Ylinethick{0.5pt}
\node at (10.6,-4.2) [text width = 0.8cm] {1} ;
\tyoung(10.6cm,-4.4cm,123<10><11><12><13><14><15>,,456789<16><17><18>)
\Yfillopacity{0}
\Ylinethick{2pt}
\tgyoung(10.6cm,-4.4cm,_3_3_3,,_3_3_3)
\end{tikzpicture}
\]

\end{eg}

\begin{prop}\label{kee}
Suppose that $\la=((e),(ke))$ or $\la=((ke),(e))$ with $k>1$.
Then $\spe\la$ is decomposable if and only if $\nchar\bbf\nmid k+1$. 
\end{prop}

\begin{proof}
We first let $\la=((e),(ke))$.
For any non-trivial endomorphism $\phi$ of $\spe\la$, we have
\[
\phi(z_\la) = \sum_{j=1}^k a_j  \Psiu{j}{k} z_\la \text{ for some } a_j \in \bbf.
\]
Since $\psi_{re} \phi(z_\la) = 0$ for all $r\in \{1,2,\dots,k-1\}$, and by \cref{prop:cancellation} we have
\[
\psi_{re} \Psiu{j}{k} z_\la = \begin{cases}
\psi_{re} \Psiu{j}{k} z_\la & \text{if } j=r-1 \text{ or } r+1,\\
-2 \psi_{re} \Psiu{j}{k} z_\la & \text{if } j=r,\\
0 & \text{otherwise,}
\end{cases}
\]
it follows that $a_1 = \frac{1}{2} a_2$ and $2 a_j = a_{j-1} + a_{j+1}$ for $j\in \{2,3, \dots, k-1\}$.
One can check that, up to scalar multiplication, the only non-trivial endomorphism of $\spe\la$ is thus the one given in \cref{keje}, which simplifies to
\[
\phi(z_\la) = \sum_{j=1}^k j  \Psiu{j}{k} z_\la.
\]
Using \cref{prop:cancellation}, we find that 
\begin{align*}
\phi^2(z_\la) &= (\sum_{j=1}^k j  \Psiu{j}{k}) (\sum_{j=1}^k j  \Psiu{j}{k}) z_\la\\
&= (\sum_{j=1}^k j  \Psiu{j}{k}) ((k{-}1) \Psiu{k-1}{k} \!\!\! + k \Psi_k) z_\la\\
&= (k{-}1) \phi(z_\la) - (2k) \phi(z_\la),
\end{align*}
and so $-\frac{1}{k+1}\phi$ is an idempotent when $\nchar\bbf\nmid k+1$, or equivalently, when $(k+1,\nchar\bbf)=1$.
Moreover, it is clear that $\phi^2=0$ when $\nchar\bbf\mid k+1$, so that there are no non-trivial idempotent endomorphisms.

We now let $\la=((ke),(e))$.
We similarly find that there exists only a single non-trivial endomorphism of $\spe\la$, up to scalar multiplication, which is defined by
\[
\phi(z_\la) = \sum_{j=1}^k (k+1-j) \Psid{j}{1} z_\la.
\]

Analogously, $-\frac{1}{k+1}\phi$ is an idempotent when $\nchar\bbf\nmid k+1$ and $\phi^2=0$ when $\nchar \bbf \mid k+1$.
\end{proof}

\begin{cor}\label{cor:nonrepdecomp}
Let $1\leq a<e$ and $0\leq b<e$ such that $a+b<e$.
Then $\spe{((e+a,1^b),(ke+a,1^b))}$ and $\spe{((ke+a,1^b),(e+a,1^b))}$ are decomposable if and only if $(k+1,\nchar\bbf)=1$ with $k>1$.
\end{cor}

\begin{proof}
The result follows from \cref{kee} by applying \cref{inducetrivfurther}.
\end{proof}

\begin{thm}\label{thm:kejeeigen}
Suppose that $k\geq j>1$.
Let $\la=((ke),(je))$, and let $\phi$ be the endomorphism of $\spe\la$ from \cref{keje}\ref{keje1}.
Then
	\begin{enumerate}[label=(\roman*)]

		\item $\phi$ has an eigenvalue $-j(k+1)$ with corresponding eigenvector
		\[
		v(B_{k+1},B_{k+2},B_{k+3},\dots,B_{k+j});
		\]

		\item $\phi$ has an eigenvalue $-(j-1)(k+2)$ with corresponding eigenvector

		\vspace{5pt}

		\begin{itemize}
			\setlength\itemsep{5pt}
			\item $v(B_2,B_3,B_{k+2},B_{k+3},\dots,B_{2k-1})
			-v(B_1,B_{k+2},B_{k+3},\dots,B_{2k})$ if $k=j$,
			\item $\displaystyle{\sum_{i=1}^{k-j+2}iv(B_2,B_4,B_6,\dots,B_{2j-2},B_{k+j-i+1})}$ if $k>j$;
		\end{itemize}

		\vspace{5pt}

		\item\label{thm:kejeeigen3} if $k>j>2$, then $\phi$ has an eigenvalue $-(j-2)(k+3)$ with corresponding eigenvector
		\begin{align*}
		&\sum_{i=1}^{k-j+3}\frac{1}{2}i(i+1)
		v(B_2,B_4,B_6,\dots,B_{2j-4},B_{k+j-i},B_{k+j-i+1})\\
		&+\sum_{i=1}^{k-j+2}i(i+1)v(B_2,B_4,B_6,\dots,B_{2j-4},B_{k+j-i-1},B_{k+j-i+1})\\
		&+\sum_{i=1}^{k-j+1}\sum_{l=1}^i
		l(i+2)v(B_2,B_4,B_6,\dots,B_{2j-4},B_{k+j-i-2},B_{k+j-l+1});
		\end{align*}

		\item if $k>j=2$, then $\phi$ has an eigenvalue $0$ with corresponding eigenvector
		\[
		\sum_{i=0}^k
		\frac{1}{2}(i+1)(i+2)
		v(B_{k-i+1},B_{k-i+2})
		+\sum_{l=0}^{k-1}\sum_{i=1}^{k-l-1}
		(i+1)(k-l+1)
		v(B_{l+1},B_{k-i+2}).
		\]

	\end{enumerate}
\end{thm}

Analogously, we also have the following theorem.

\begin{thm}\label{thm:jekeeigen}

Suppose that $k\geq j>1$.
Let $\la =((je),(ke))$, and let $\phi$ be the endomorphism of $\spe\la$ from \cref{keje}\ref{keje2}.
Then
	\begin{enumerate}[label=(\roman*)]

		\item $\phi$ has an eigenvalue $-j(k+1)$ with corresponding eigenvector
		\[
		\underline{v}(B_{k+1},B_{k+2},B_{k+3},\dots,B_{k+j});
		\]

		\item $\phi$ has an eigenvalue $-(j-1)(k+2)$ with corresponding eigenvector

		\vspace{5pt}

		\begin{itemize}
			\setlength\itemsep{5pt}
			\item $\underline{v}(B_2,B_3,B_{k+2},B_{k+3},\dots,B_{2k-1})
			-\underline{v}(B_1,B_{k+2},B_{k+3},\dots,B_{2k})$ if $k=j$,
			\item $\displaystyle{\sum_{i=1}^{k-j+2}i
				\underline{v}(B_2,B_4,B_6,\dots,B_{2j-2},B_{k+j-i+1})}$ if $k>j$;
		\end{itemize}
 
		\vspace{5pt}

		\item if $k>j>2$, then $\phi$ has an eigenvalue $-(j-2)(k+3)$ with corresponding eigenvector
		\begin{align*}
		&\sum_{i=1}^{k-j+3}\frac{1}{2}i(i+1)
		\underline{v}(B_2,B_4,B_6,\dots,B_{2j-4},B_{k+j-i},B_{k+j-i+1})\\
		&+\sum_{i=1}^{k-j+2}i(i+1)\underline{v}(B_2,B_4,B_6,\dots,B_{2j-4},B_{k+j-i-1},B_{k+j-i+1})\\
		&+\sum_{i=1}^{k-j+1}\sum_{l=1}^i
		l(i+2)\underline{v}(B_2,B_4,B_6,\dots,B_{2j-4},B_{k+j-i-2},B_{k+j-l+1});
		\end{align*}

		\item if $k>j=2$, then $\phi$ has an eigenvalue $0$ with corresponding eigenvector
		\[
		\sum_{i=0}^k
		\frac{1}{2}(i+1)(i+2)
		\underline{v}(B_{k-i+1},B_{k-i+2})
		+\sum_{l=0}^{k-1}\sum_{i=1}^{k-l-1}
		(i+1)(k-l+1)
		\underline{v}(B_{l+1},B_{k-i+2}).
		\]
		\end{enumerate}
\end{thm}

Below, we will prove \cref{thm:kejeeigen}\ref{thm:kejeeigen3}.
This is the most difficult part of \cref{thm:kejeeigen} to prove, and the others are proved analogously.
Likewise, similar calculations prove \cref{thm:jekeeigen}.

\begin{proof}[Proof of \cref{thm:kejeeigen}\ref{thm:kejeeigen3}]

Let $\eta = \Psi_1 \Psid{3}{2} \Psid{5}{3} \dots \Psid{2j-7}{j-3}\Psid{2j-5}{j-2}$, and

\begin{align*}
\psi_{w_{(1,s)}} &:=\eta\cdot\Psid{k+j-s-1}{j-1}\Psid{k+j-s}{j}
\quad\left(1\leq s \leq k-j+3\right),\\
\psi_{w_{(2,s)}} &:=\eta\cdot\Psid{k+j-s-2}{j-1}\Psid{k+j-s}{j}
\quad\left(1\leq s \leq k-j+2\right),\\
\psi_{w_{(r,s)}} &:=\eta\cdot
\Psid{k+j-s-3}{j-1}\Psid{k+j-r+2}{j}
\quad\left( 1\leq s\leq k-j+1, 3\leq r \leq s+2 \right),
\end{align*}
and define $v_{r_s} = \psi_{w_{(r,s)}} z_\la$ in all cases above.
We can now write
\begin{align*}
&\sum_{s=1}^{k-j+3}\tfrac{1}{2}s(s+1)
v(B_2,B_4,B_6,\dots,B_{2j-4},B_{k+j-s},B_{k+j-s+1})\\
&+\sum_{s=1}^{k-j+2}s(s+1)v(B_2,B_4,B_6,\dots,B_{2j-4},B_{k+j-s-1},B_{k+j-s+1})\\
&+\sum_{s=1}^{k-j+1}\sum_{r=3}^{s+2}
(r-2)(s+2)v(B_2,B_4,B_6,\dots,B_{2j-4},B_{k+j-s-2},B_{k+j-r+3})\\
&=\sum_{s=1}^{k-j+3}\tfrac{1}{2}s(s+1)v_{1_s}
+\sum_{s=1}^{k-j+2}s(s+1)v_{2_s}
+\sum_{s=1}^{k-j+1}\sum_{r=3}^{s+2}(r-2)(s+2)v_{r_s}.
\end{align*}
We want to show that
\begin{align*}
&\left(\sum_{s=1}^{k-j+3}\tfrac{1}{2}s(s+1) \psi_{w_{(1,s)}}
+\sum_{s=1}^{k-j+2}s(s+1) \psi_{w_{(2,s)}}
+\sum_{s=1}^{k-j+1}\sum_{r=3}^{s+2}(r-2)(s+2) \psi_{w_{(r,s)}}\right)
\cdot
\phi(z_\la)\\
&=-(j-2)(k+3)
\left(\sum_{s=1}^{k-j+3}\tfrac{1}{2}s(s+1)v_{1_s}
+\sum_{s=1}^{k-j+2}s(s+1)v_{2_s}
+\sum_{s=1}^{k-j+1}\sum_{r=3}^{s+2}(r-2)(s+2)v_{r_s}\right).
\end{align*}
First suppose we are in one of the following cases, where $l$ is the index in the summation form for $\phi(z_\la)$ from \cref{keje}:
\begin{itemize}
	\item $s\leq k-j$ for all $r$ and $l$;
	\item $s = k-j+1$ and  $r=1$ or $r=2$ or ($r\geq 3$ and $l\leq j-2$);
	\item $s=k-j+2$ and $l\leq j-2$;
	\item $s=k-j+3$ and $l\leq j-3$.
\end{itemize}
In these cases, we will show that
\begin{align*}
\psi_{w_{(r,s)}}\cdot\Psiu{j-l}{j-1}\Psid{j+i}{j} z_\la=
\begin{cases}
-2v_{r_s}&\text{if $l=i$,}\\
v_{r_s}&\text{if $l=i \pm 1$,}\\
0&\text{otherwise.}
\end{cases}
\end{align*}
Suppose that $l=0$.
If $i=0$ or $1$, we have
\[
\psi_{w_{(r,s)}}\cdot\Psi_jz_\la=-2v_{r_s} \quad \text{or} \quad \psi_{w_{(r,s)}}\cdot\cancel{\Psi_{j+1}\Psi_j} z_\la
=v_{r_s}, \quad \text{respectively.}
\]
If $i>1$, we have
\[
\psi_{w_{(r,s)}}\cdot\Psid{j+i}{j} z_\la
=\eta\cdot\Psid{i_{j-1}-1}{j-1}
\Psid{i_j-1}{j}\Psid{j+i}{j+2}
\cancel{\Psi_{j+1}\Psi_j} z_\la
=0.
\]
Now suppose $l=1$.
If $i=0$, we have
\begin{align*}
\psi_{w_{(r,s)}}\cdot\cancel{\Psi_{j-1}\Psi_j} z_\la
=v_{r_s}.
\end{align*}
If $i=1$, we have
\begin{align*}
\psi_{w_{(r,s)}} \cdot \Psi_{j-1} \Psi_{j+1} \Psi_j z_\la
&= \eta \cdot \Psid{i_{j-1}-1}{j-1} \Psid{i_j-1}{j+1}
\cancel{\Psi_j \Psi_{j+1}} \Psi_{j-1}\Psi_j z_\la
=-2\eta \cdot \Psid{i_{j-1}-1}{j-1} \Psid{i_j-1}{j} z_\la
=-2v_{r_s}.
\end{align*}
If $i=2$, we have
\begin{align*}
\psi_{w_{(r,s)}}\cdot \Psi_{j-1} \Psi_{j+2} \Psi_{j+1} \Psi_j z_\la
&=\eta \cdot
\Psid{i_{j-1}-1}{j-1} \Psid{i_j-1}{j+2} \cancel{\Psi_{j+1} \Psi_{j+2}} \Psi_j
\Psi_{j-1}\Psi_{j+1} \Psi_j z_\la\\
&=\eta\cdot
\Psid{i_{j-1}-1}{j-1}\cancel{\Psi_j\Psi_{j-1}}
\Psid{i_j-1}{j} z_\la\\
&=v_{r_s}.
\end{align*}
If $i>2$, we have
\begin{align*}
\psi_{w_{(r,s)}} \Psi_{j-1} \Psid{j+i}{j} z_\la
&=\eta\cdot
\Psid{i_{j-1}-1}{j-1} \Psid{i_j-1}{j+i}
\cancel{\Psi_{j+i-1} \Psi_{j+i}}\Psid{j+i-2}{j-1}
\Psid{j+i-1}{j} z_\la\\
&=\eta\cdot
\Psid{i_{j-1}-1}{j}
\Psid{j+i-2}{j+1}\Psi_{j-1}
\cancel{\Psi_j\Psi_{j-1}}
\Psid{i_j-1}{j} z_\la\\
&=\eta\cdot
\Psid{i_{j-1}-1}{j-1}
\Psid{j+i-2}{j+2}
\Psid{i_j-1}{j+3}
\Psi_{j+1}\cancel{\Psi_{j+2}\Psi_{j+1}}\Psi_jz_\la\\
&=0.
\end{align*}

Now assuming that $l>1$, we have
\begin{align*}
&\psi_w{_{(r,s)}}\cdot
\Psiu{j-l}{j-1}\Psid{j+i}{j} z_\la\\
= \; &\Psi_1 \Psid{3}{2} \Psid{5}{3} \dots
\Psid{2j-2l-1}{j-l} \Psid{2j-2l+1}{j-l} \Psid{2j-2l+3}{j-l+1} \dots \Psid{2j-5}{j-3} \Psid{i_{j-1}-1}{j-2}
\Psid{i_j-1}{j-1} \Psid{j+i}{j} z_\la\\
= \; &\Psi_1 \Psid{3}{2} \Psid{5}{3} \dots
\Psid{2j-2l-1}{j-l+1} \Psid{2j-2l+1}{j-l+2}
\Psi_{j-l} \cancel{\Psi_{j-l+1}\Psi_{j-l}}
\Psid{2j-2l+3}{j-l+1} \dots \Psid{2j-5}{j-3}
\Psid{i_{j-1}-1}{j-2} \Psid{i_j-1}{j-1}
\Psid{j+i}{j} z_\la\\
= \; &\Psi_1 \Psid{3}{2} \Psid{5}{3} \dots \Psid{2j-2l-1}{j-l}
\Psid{2j-2l+1}{j-l+3}
\Psid{2j-2l+3}{j-l+4}
\Psi_{j-l+2} \cancel{\Psi_{j-l+3} \Psi_{j-l+2}} \Psi_{j-l+1}
\Psid{2j-2l+5}{j-l+2} \dots\\
&\dots \Psid{2j-5}{j-3}
\Psid{i_{j-1}-1}{j-2} \Psid{i_j-1}{j-1}
\Psid{j+i}{j} z_\la\\
= \; &\Psi_1 \Psid{3}{2} \Psid{5}{3} \dots \Psid{2j-2l+1}{j-l+1}
\Psid{2j-2l+3}{j-l+5} \Psid{2j-2l+5}{j-l+6}
\Psi_{j-l+4} \cancel{\Psi_{j-l+5}\Psi_{j-l+4}}
\Psid{j-l+3}{j-l+2}
\Psid{2j-2l+7}{j-l+3}
\dots\\
&\dots\Psid{2j-5}{j-3}
\Psid{i_{j-1}-1}{j-2} \Psid{i_j-1}{j-1}
\Psid{j+i}{j} z_\la\\
\vdots \; \; &\\
=
\; &\Psi_1 \Psid{3}{2} \Psid{5}{3} \dots \Psid{2j-7}{j-3}
\Psid{2j-5}{j+l-3}
\Psid{i_{j-1}-1}{j+l-2}
\Psi_{j+l-4} \cancel{\Psi_{j+l-3} \Psi_{j+l-4}}
\Psid{j+l-5}{j-2} \Psid{i_j-1}{j-1}
\Psid{j+i}{j} z_\la\\
=
\; &\Psi_1 \Psid{3}{2} \Psid{5}{3} \dots \Psid{2j-7}{j-3}
\Psid{2j-5}{j-2}
\Psid{i_{j-1}-1}{j+l-1}
\Psid{i_j-1}{j+l}
\Psi_{j+l-2} \cancel{\Psi_{j+l-1} \Psi_{j+l-2}}
\Psid{j+l-3}{j-1}
\Psid{j+i}{j} z_\la\\
=
\; &\Psi_1 \Psid{3}{2} \Psid{5}{3} \dots \Psid{2j-7}{j-3}
\Psid{2j-5}{j-2}
\Psid{i_{j-1}-1}{j-1}
\Psid{i_j-1}{j+l}
\Psid{j+i}{j} z_\la.
\end{align*}

If $l=i$, observe
\[
\Psid{i_j-1}{j+i} \Psid{j+i}{j} z_\la
=-2\Psid{i_j-1}{j} z_\la.
\]

If $l=i-1$, observe
\[
\Psid{i_j-1}{j+i-1}\cancel{\Psi_{j+i} \Psi_{j+i-1}} \Psid{j+i-2}{j} z_\la
=\Psid{i_j-1}{j} z_\la.
\]

If $l=i+1$, observe
\[
\Psid{i_j-1}{j+i+1} \Psid{j+i}{j} z_\la = \Psid{i_j-1}{j} z_\la.
\]

If $l>i+1$, observe
\[
\Psid{i_j-1}{j+l} \Psid{j+i}{j} z_\la
=\Psid{i_j-1}{j+l+1} \Psid{j+i}{j} \Psi_{j+l} z_\la
=0.
\]

If $l<i-1$, observe
\begin{align*}
\Psid{i_j-1}{j+l} \Psid{j+i}{j} z_\la
&=\Psid{i_j-1}{j+l} \Psid{j+i}{j+l+2}
\cancel{\Psi_{j+l+1} \Psi_{j+l}}
\Psid{j+l-1}{j} z_\la
=\Psid{i_j-1}{j+l} \Psid{j+i}{j+l+3}
\Psid{j+l-1}{j} \Psi_{j+l+2} z_\la
=0.
\end{align*}

We now suppose that we are not in the listed cases and that we lie in the \emph{exceptional cases}.

First let $r=1$, $s=k-j+3$ and $l=j-2$.
Similarly to above, we obtain
\begin{align*}
\psi_w{_{(r,s)}}\cdot
\Psiu{j-l}{j-1}\Psid{j+i}{j} z_\la = \;&\dots = 
\Psi_1\Psid{3}{2}\Psid{5}{4}\dots\Psid{2j-7}{j-3}
\Psi_{2j-5}\Psi_{2j-6}
\Psid{2j-4}{j-2}
\Psid{2j-3}{j-1}\Psid{j+i}{j} z_\la\\
=\; &\Psi_1\Psid{3}{2}\Psid{5}{4}\dots\Psid{2j-7}{j-3}
\Psi_{2j-5}
\Psi_{2j-4}\Psi_{2j-6}\cancel{\Psi_{2j-5}\Psi_{2j-6}}
\Psid{2j-7}{j-2}
\Psid{2j-3}{j-1}\Psid{j+i}{j} z_\la\\
=\; &\Psi_1\Psid{3}{2}\Psid{5}{4}\dots\Psid{2j-7}{j-3}
\Psid{2j-5}{j-2}
\Psi_{2j-4}
\cancel{\Psi_{2j-3}\Psi_{2j-4}}
\Psid{2j-5}{j-1}\Psid{j+i}{j} z_\la\\
=\; &\Psi_1\Psid{3}{2}\Psid{5}{4}\dots\Psid{2j-7}{j-3}
\Psid{2j-5}{j-2}
\Psid{2j-4}{j-1}\Psid{j+i}{j} z_\la.
\end{align*}
If $i\geq j-3$, then this expression is clearly reduced.
If $i<j-3$, we have
\begin{align*}
&\Psi_1\Psid{3}{2}\Psid{5}{4}\dots\Psid{2j-7}{j-3}
\Psid{2j-5}{j-2}
\Psid{2j-4}{j-1}\Psid{j+i}{j} z_\la\\
=\; &\Psi_1\Psid{3}{2}\Psid{5}{4}\dots\Psid{2j-7}{j-3}
\Psid{2j-5}{j-2}
\Psid{2j-4}{j+i+1}
\Psi_{j+i}\cancel{\Psi_{j+i-1}\Psi_{j+i}}
\Psid{j+i-2}{j-1}\Psid{j+i-1}{j} z_\la\\
=\; &\Psi_1\Psid{3}{2}\Psid{5}{4}\dots\Psid{2j-7}{j-3}
\Psid{2j-5}{j+i-1}
\Psi_{j+i-2}\cancel{\Psi_{j+i-3}\Psi_{j+i-2}}
\Psid{j+i-4}{j-2}
\Psid{j+i-3}{j-1}\Psid{2j-4}{j} z_\la\\
\vdots \; \; &\\
=\; &\Psi_1\Psid{3}{2}\Psid{5}{4}\dots
\Psid{2j-2i-5}{j-i-2}
\Psid{2j-2i-3}{j-i+1}
\Psi_{j-i}\cancel{\Psi_{j-i-1}\Psi_{j-i}}
\Psid{2j-2i-1}{j-i+1}\dots \Psid{2j-5}{j-1}\Psid{2j-4}{j} z_\la\\
=\; &\Psi_1\Psid{3}{2}\Psid{5}{4}\dots
\Psid{2j-2i-7}{j-i-3} \Psid{2j-2i-5}{j-i-1}
\Psid{2j-2i-3}{j-i}
\Psid{2j-2i-1}{j-i+1}\dots \Psid{2j-5}{j-1}\Psid{2j-4}{j}
\Psi_{j-i-2} z_\la\\
=\; &0.
\end{align*}

If instead $l=j-1$, we similarly obtain
\begin{align*}
&\psi_{w_{(r,s)}} \cdot \Psiu{j-l}{j-1} \Psid{j+i}{j} z_\la
=
\eta \cdot \Psid{2j-3}{j-1}\Psid{j+i}{j} z_\la,
\end{align*}
which is clearly reduced if $i\geq j-2$.
If $i< j-2$, then the expression becomes zero.

Let $r=1$, $s=k-j+2$, and $l=j-1$.
Then the above expression becomes
\[
\eta\cdot\Psid{2j-3}{j-1}\Psid{j+i}{j} z_\la,
\]
which is clearly reduced if $i \geq j-2$.
If $i< j-2$, then the expression becomes zero, as before.

Now let $r=2$, $s=k-j+2$, and $l = j-1$.
Then the above expression becomes
\[
\eta\cdot\Psid{2j-2}{j-1}\Psid{j+i}{j} z_\la,
\]
which is clearly reduced if $i \geq j-1$.
If $i< j-2$, then the expression becomes zero, while for $i = j-2$ we get
\[
\eta \cdot \Psid{2j-4}{j-1}\Psid{2j-2}{j} z_\la.
\]

Finally, let $r>2$, $s=k-j+1$, and $l=j-1$.
Then the above expression becomes
\begin{align*}
&\eta\cdot\Psid{k+j-r+2}{j-1}\Psid{j+i}{j} z_\la,
\end{align*}
which is clearly reduced if $i\geq k-r+3$.
If $j-2 \leq i \leq k-r+2$, then
\begin{align*}
\eta \cdot \Psid{k+j-r+2}{j-1} \Psid{j+i}{j} z_\la
&=\eta \cdot \Psid{k+j-r+2}{j+i+1}
\Psi_{j+i} \cancel{\Psi_{j+i-1} \Psi_{j+i}}
\Psid{j+i-2}{j-1} \Psid{j+i-1}{j} z_\la
=\eta\cdot\Psid{j+i-2}{j-1}\Psid{k+j-r+2}{j} z_\la,
\end{align*}
which is reduced.
If $i < j-2$, then the expression becomes zero.

We summarise the exceptional cases.
First suppose $r=1$, $s=k-j+3$, and $l=j-2$.
Then

\begin{align*}
\psi_{w_{(1,s)}} \cdot \Psiu{2}{j-1} \Psid{j+i}{j} z_\la
=\begin{cases}
0&\text{if $i\leq j-4$,}\\
v_{1_{s}}&\text{if $i=j-3$,}\\
v_{2_{s-1}}&\text{if $i=j-2$,}\\
v_{(k-i+2)_{s-2}}&\text{if $i\geq j-1$.}
\end{cases}
\end{align*}
If $s=k-j+2$ or $s=k-j+3$ and $l=j-1$, then
\begin{align*}
\psi_{w_{(1,s)}}\cdot \Psiu{1}{j-1}\Psid{j+i}{j} z_\la
=\begin{cases}
0&\text{if $i\leq j-3$,}\\
v_{1_{k-j+2}}&\text{if $i=j-2$,}\\
v_{2_{k-j+1}}&\text{if $i=j-1$,}\\
v_{(k-i+2)_{k-j}}&\text{if $i\geq j$.}
\end{cases}
\end{align*}

Now suppose that $r=2$.
If $s=k-j+2$ and $l=j-1$, then
\begin{align*}
\psi_{w_{(2,s)}}\cdot\Psiu{1}{j-1}\Psid{j+i}{j} z_\la
=\begin{cases}
0&\text{if $i\leq j-3$,}\\
v_{2_{s}}&\text{if $i=j-2$,}\\
v_{1_{s-1}}&\text{if $i=j-1$,}\\
v_{2_{s-2}}&\text{if $i=j$,}\\
v_{(k-i+2)_{s-3}}&\text{if $i\geq j+1$.}
\end{cases}
\end{align*}

Finally, suppose that $s=k-j+1$.
If $r\in\{3,\dots,k-j+3\}$ and $l=j-1$, then
\begin{align*}
\psi_{w_{(r,s)}}\cdot\Psiu{1}{j-1}\Psid{j+i}{j} z_\la
=\begin{cases}
0&\text{if $i\leq j-3$,}\\
v_{r_{k-i-1}}&\text{if $j-2\leq i \leq k-r+1$,}\\
v_{2_{r-2}}&\text{if $i=k-r+2$,}\\
v_{1_{r-3}}&\text{if $i=k-r+3$,}\\
v_{2_{r-4}}&\text{if $i=k-r+4$,}\\
v_{(k-i+2)_{r-5}}&\text{if $i\geq k-r+5$.}
\end{cases}
\end{align*}

In the non-exceptional cases, we determine the coefficient $\alpha \in \bbz$ in
\[
\psi_{w_{(r,s)}} \cdot \phi(z_\la)
= \alpha v_{r_s}.
\]
If $s\leq k-j$, we obtain
\begin{align*}
\psi_{w_{(r,s)}}\cdot\phi(z_\la)
&=
\psi_{w_{(r,s)}}\cdot
\; \sum_{\mathclap{\substack{0\leq i \leq k-1\\0 \leq l \leq j-1}}}\; (j-l)(k-i)
\Psiu{j-l}{j-i} \Psid{j+i}{j} z_\la\\
&=
\sum_{l=0}^{j-1}
-2(j-l)(k-l)v_{r_s}
+\sum_{l=0}^{j-1}
(j-l)(k-l-1)v_{r_s}
+\sum_{l=1}^{j-1}
(j-l)(k-l+1) v_{r_s}\\
&=
-j(k+1)v_{r_s}.
\end{align*}
Similarly, we also obtain $\alpha=-j(k+1)$ in the other non-exceptional cases.

By combining this with the exceptional cases, we now determine $\alpha_{r_s}$ where
\begin{align*}
&\left(\sum_{s=1}^{k-j+3}\tfrac{1}{2}s(s+1)\psi_{w_{(1,s)}}
+\sum_{s=1}^{k-j+2}s(s+1)\psi_{w_{(2,s)}}
+\sum_{s=1}^{k-j+1}\sum_{r=3}^{s+2}(r-2)(s+2)\psi_{w_{(r,s)}}\right)
\cdot
\phi(z_\la)\\
&=
\sum_{s=1}^{k-j+3}\alpha_{1_s}v_{1_s}
+\sum_{s=1}^{k-j+2}\alpha_{2_s}v_{2_s}
+\sum_{s=1}^{k-j+1}\sum_{r=3}^{s+2}\alpha_{r_s}v_{r_s}.
\end{align*}

We first let $r=1$.
If $s=k-j+3$, we know from above that
\begin{align*}
\alpha_{1_s}v_{1_s}
&=\tfrac{1}{2}s(s+1)
\psi_{w_{(1,s)}}\cdot
\left(\ \ \sum_{\mathclap{\substack{0\leq i \leq k-1\\0 \leq l \leq j-3}}}\ \ 
(j-l)(k-i)\Psiu{j-l}{j-1}\Psid{j+i}{j} z_\la
+2s\Psiu{1}{j-1}\Psid{2j-3}{j} z_\la
\right)\\
&=\tfrac{1}{2}s(s+1)
\left( -j(k+1) v_{1_{s}} + 2s v_{1_{s}} \right)\\
&=-(j-2)(k+3)\cdot\tfrac{1}{2}s(s+1) v_{1_{s}},
\end{align*}
hence $\alpha_{1_{s}}=-(j-2)(k+3)\cdot\frac{1}{2}s(s+1)$, as required.

If $s=k-j+2$, then we have
\begin{align*}
\alpha_{1_s} v_{1_s}
&=\tfrac{1}{2}s(s+1)
\psi_{w_{(1,s)}}\cdot
\left(\ \ \sum_{\mathclap{\substack{0\leq i \leq k-1\\0 \leq l \leq j-2}}}\ \ 
(j-l)(k-i)\Psiu{j-l}{j-1}\Psid{j+i}{j} z_\la
+s\Psiu{1}{j-1}\Psid{2j-2}{j} z_\la
\right)\\
&\phantom{=} \; +\tfrac{1}{2}(s+1)(s+2)\psi_{w_{(1,s+1)}}
\cdot s\Psiu{1}{j-1}\Psid{2j-2}{j} z_\la\\
&= \tfrac{1}{2}s(s+1)
\left( -j(k+1)+s \right)v_{1_s}
+\tfrac{1}{2}s(s+1)(s+2)v_{1_s}\\
&= -(j-2)(k+3)\cdot\tfrac{1}{2}(s+1)(s+2)v_{1_s}.
\end{align*}

If $s=k-j+1$, then we have
\begin{align*}
\alpha_{1_s}v_{1_s}
&=\tfrac{1}{2}s(s+1)\psi_{w_{(1,s)}}
\cdot\phi(z_\la)
+(s+1)(s+2)\psi_{w_{(2,s+1)}}
\cdot s
\Psiu{1}{j-1}\Psid{2j-1}{j} z_\la\\
&=
-(j-2)(k+3)\cdot\tfrac{1}{2}s(s+1)v_{1_s}.
\end{align*}

For $1\leq{s}\leq k-j$, we have
\begin{align*}
\alpha_{1_s}v_{1_s}
&=\tfrac{1}{2}s(s+1)\psi_{w_{(1,s)}}\cdot\phi(z_\la)
+(s+1)(k-j+3)\psi_{w_{(s+3,k-j+1)}}\cdot
s\Psiu{1}{j-1}\Psid{j+k-s}{j} z_\la\\
&=-(j-2)(k+3)\cdot\tfrac{1}{2}s(s+1)v_{1_s}.
\end{align*}

We now suppose that $r=2$.
If $s=k-j+2$, then

\begin{align*}
\alpha_{2_s}v_{2_s}
&= s(s+1)\psi_{w_{(2,s)}}\cdot
\ \ \sum_{\mathclap{\substack{0\leq i \leq k-1\\0 \leq l \leq j-2}}}\ \ 
(j-l)(k-i)\Psiu{j-l}{j-1}\Psid{j+i}{j} z_\la\\
&\phantom{=} \; +\tfrac{1}{2}(s+1)(s+2)\psi_{w_{(1,s+1)}}
\cdot 2s\Psiu{2}{j-1}\Psid{2j-2}{j} z_\la
+s(s+1)\psi_{w_{(2,s)}}\cdot
s\Psiu{1}{j-1}\Psid{2j-2}{j} z_\la\\
&=-s(s+1)j(k+1) v_{2_s} + s(s+1)(s+2)v_{2_s} + s^2(s+1)v_{2_s}\\
&=-(j-2)(k+3)\cdot s(s+1)v_{2_s}.
\end{align*}

If $s=k-j+1$, then
\begin{align*}
\alpha_{2_s}v_{2_s}
&=
s(s+1)\psi_{w_{(2,k-j+1)}}\cdot\phi(z_\la)
+\tfrac{1}{2}(s+1)(s+2)\psi_{w_{(1,s+1)}}\cdot s\Psiu{1}{j-1}\Psid{2j-1}{j} z_\la\\
&\phantom{=} \; +\tfrac{1}{2}(s+2)(s+3)\psi_{w_{(1,s+2)}}\cdot s\Psiu{1}{j-1}\Psid{2j-1}{j} z_\la
+s(s+2)\psi_{w_{(s+2,s)}}\cdot s\Psiu{1}{j-1}\Psid{2j-1}{j} z_\la\\
&=
-j(k+1)s(s+1)v_{2_s}
+s(s+2)\left(\tfrac{1}{2}(s+1)+\tfrac{1}{2}(s+3)+s\right)v_{2_s}\\
&=-(j-2)(k+3)\cdot s(s+1)v_{2_s}.
\end{align*}

If $s=k-j$, then
\begin{align*}
\alpha_{2_s}v_{2_s}
&=
s(s+1)\psi_{w_{(2,s)}}\cdot\phi_{z_\la}
+(s+2)(s+3)\psi_{w_{(2,s+2)}}\cdot s\Psiu{1}{j-1}\Psid{2j}{j} z_\la\\
&\phantom{=} \; +s(s+3)\psi_{w_{(s+2,s+1)}}\cdot s\Psiu{1}{j-1}\Psid{2j}{j} z_\la\\
&=-j(k+1)s(s+1)v_{2_s}
+s(s+3)\left(2s+2\right)v_{2_s}\\
&=-(j-2)(k+3)\cdot s(s+1)v_{2_s}.
\end{align*}

If $s\in\{1,\dots,k-j-1\}$, then
\begin{align*}
\alpha_{2_s}v_{2_s}
&=s(s+1)\psi_{w_{(2,s)}}\cdot\phi(z_\la)
+s(k-j+3)\phi_{w_{(s+2,k-j+1)}}\cdot s\Psiu{1}{j-1}\Psid{j+k-s}{j} z_\la\\
&\phantom{=} \; +(s+2)(k-j+3)\psi_{w_{(s+4,k-j+1)}}\cdot s\Psiu{1}{j-1}\Psid{j+k-s}{j} z_\la\\
&=
-j(k+1)s(s+1)v_{2_s}
+s(k-j+3)(2s+2)v_{2_s}\\
&=
-(j-2)(k+3)\cdot s(s+1)v_{2_s}.
\end{align*}

Finally, suppose that $r\in\{3,\dots,k-j+3\}$.
If $s=k-j+1$, then
\begin{align*}
\alpha_{r_s}v_{r_s}
&=(r-2)(s+2)\psi_{w_{(r,s)}}\cdot
\sum_{\mathclap{\substack{0\leq i \leq k-1\\0 \leq l \leq j-2}}}(j-l)(k-i)\Psiu{j-l}{j-1}\Psid{j+i}{j} z_\la\\
&\phantom{=} \; +\tfrac{1}{2}(s+2)(s+3)\psi_{w_{(1,s+2)}}\cdot 2(r-2)
\Psiu{2}{j-1}\Psid{j+k-r+2}{j} z_\la\\
&\phantom{=} \; +(r-2)(s+2)\psi_{w_{(r,s)}}\cdot
(s+1)\Psiu{1}{j-1} \Psid{2j-2}{j} z_\la\\
&=-(j-2)(k+3)\cdot (r-2)(s+2)v_{r_{k-j+1}}.
\end{align*}

If $s=k-j$, then
\begin{align*}
\alpha_{r_s}v_{r_s}
&=(r-2)(s+2) \psi_{w_{(r,s)}} \cdot \phi(z_\la)\\
&\phantom{=} \; +\tfrac{1}{2}(s+2)(s+3)\psi_{w_{(1,s+2)}} \cdot
(r-2)\Psiu{1}{j-1} \Psid{j+k-r+2}{j} z_\la\\
&\phantom{=} \; +\tfrac{1}{2}(s+3)(s+4) \psi_{w_{(1,s+3)}}
\cdot(r-2) \Psiu{1}{j-1} \Psid{j+k-r+2}{j} z_\la\\
&\phantom{=} \; +(r-2)(s+3) \psi_{w_{(r,s+1)}}\cdot
(s+1)\Psiu{1}{j-1} \Psid{2j-1}{j} z_\la\\
&=
(r-2)(s+2)\left(-j(k+1)+2(s+3)\right) v_{r_s}\\
&=
-(j-2)(k+3)\cdot (r-2)(s+2)v_{r_s}.
\end{align*}

If $s=k-j-1$, then
\begin{align*}
\alpha_{r_s}v_{r_s}
&=(r-2)(s+2) \psi_{w_{(r,s)}} \cdot \phi(z_\la)
+(s+3)(s+4) \psi_{w_{(2,s+3)}} \cdot
(r-2) \Psiu{1}{j-1} \Psid{j+k-r+2}{j} z_\la\\
&\phantom{=} \; +(r-2)(s+4) \psi_{w_{(r,s+2)}}\cdot
(k-j) \Psiu{1}{j-1} \Psid{2j-1}{j} z_\la\\
&=-(j-2)(k+3) \cdot (r-2)(s+2) v_{r_s}.
\end{align*}

If $s\in\{1,\dots,k-j-2\}$, then
\begin{align*}
\alpha_{r_s} v_{r_s}
&=(r-2)(s+2) \psi_{w_{(r,s)}} \cdot \phi_{z_\la}\\
&\phantom{=} \; +(r-2)(k-j+3) \psi_{w_{(r,k-j+1)}}\cdot
(s+1)\Psiu{1}{j-1} \Psid{j+k-s-1}{j} z_\la\\
&\phantom{=} \; +(s+3)(k-j+3) \psi_{w_{(s+5,k-j+1)}} \cdot
(r-2) \Psiu{1}{j-1} \Psid{j+k-r+2}{j} z_\la\\
&=(r-2)\left(-j(k+1)(s+2)+(k-j+3)(2s+4)\right) v_{r_s}\\
&=-(j-2)(k+3)\cdot (r-2)(s+2) v_{r_s}.
\end{align*}

Hence
\[
\alpha_{r_s}
=
-(j-2)(k+3)\cdot
\begin{cases}
\tfrac{1}{2}s(s+1)
&\text{if $r=1$;}\\
s(s+1)
&\text{if $r=2$;}\\
(r-2)(s+2)
&\text{if $r\geq 3$,}\\
\end{cases}
\]
for all $s$, as required.
\end{proof}

\begin{cor}\label{kejedecomp}
Let $\la = ((ke),(je))$ or $((1^{ke}),(1^{je}))$ with $k, j > 1$.
\begin{enumerate}[label=(\roman*)]
\item If $j+k$ is odd, then $\spe\la$ is decomposable.

\item If $j+k$ is even and $\nchar \bbf \neq 2$, then $\spe\la$ is decomposable.
\end{enumerate}	
\end{cor}

\begin{proof}
By \cref{thm:kejeeigen,thm:jekeeigen} parts (i) and (ii), there is an endomorphism $\phi$ of $\spe\la$ with two eigenvalues differing by $j-k-2$.
If $j=k$, then $j+k$ is even and these eigenvalues differ by $2$, and are hence these are distinct if $\nchar \bbf \neq 2$.

If $k>j>2$, then by part (iii), we have a third eigenvalue differing from the second by $j-k-4$.
If $\nchar \bbf \mid (j-k-2)$, then the first two eigenvalues are equal, but the third is distinct if $\nchar \bbf \neq 2$.
If $\nchar \bbf = 2$, then $\nchar \bbf \nmid (j-k-2)$ when $j+k$ is odd.

If $k> j = 2$, then by part (iv), we have $0$ as our third eigenvalue.
If $k$ is odd, we have at least two distinct eigenvalues in any characteristic.
Likewise, if $k$ is even and $\nchar \bbf \neq 2$, we have at least two distinct eigenvalues.

Thus, in each case, the generalised eigenspace decomposition of $\spe\la$ has at least two direct summands, which are easily seen to be $\mathscr{R}_n^\La$-modules.
\end{proof}

By applying \cref{transdec} and the method in the proof of \cref{inducetriv,inducetrivfurther}, and combining with \cref{smallbihooks,kee,cor:nonrepdecomp}, we obtain \cref{mainresult}.

\section{Quantum characteristic two}\label{sec:e=2}

We now focus our attention on the case where $e=2$.

Thankfully most of the difficult work here is already done, and we are able to use results from \cite{ls14,gm80,fs16} on Specht modules over the level one cyclotomic quiver Hecke algebra to obtain a lot of decomposable Specht modules with little effort.
We collect the results we will need.

\begin{thmc}{gm80}{Theorem 4.5}\label{lev1murph}
If $e = \nchar \bbf = 2$, $a\geq b$, and either $\la$ or $\la'$ is $(a,1^b)$, then $\spe\la$ is decomposable if and only if $a+b$ is odd and $a-1 \not\equiv b \pmod{2^L}$, where $2^{L-1} \leq b < 2^L$.
\end{thmc}

\begin{thmc}{ls14}{Theorem 6.8}\label{lev1spey}
If $e=2$, $\nchar \bbf \neq 2$, $a>b$, and either $\la$ or $\la'$ is $(a,1^b)$, then $\spe\la$ is decomposable if and only if $a+b$ is odd and either
\begin{enumerate}[label=(\roman*)]
\item $a > b \geq 4$; or

\item $a > b = 2$ or $3$, and $\nchar \bbf \nmid \lceil \frac{a}{2}\rceil$.
\end{enumerate}
\end{thmc}

When looking for decomposable Specht modules, the cases $\kappa = (0,1)$ and $\kappa = (0,0)$ must be treated separately.
However, our first result is independent of this.

\begin{thm}\label{lev1tolev2}
Let $\kappa$ be arbitrary, with corresponding $\La = \La_\kappa$.
Let $\mu$ be a hook partition of $n$ such that $\spe \mu$ is a decomposable $\mathscr{R}_n^{\La_0}$-module (cf.~\cref{lev1murph,lev1spey}).
Then for any partition $\nu$ of $m$, the Specht modules $\spe{(\mu, \nu)}$ and $\spe{(\nu, \mu)}$ are decomposable $\mathscr{R}_{m+n}^{\La}$-modules.
\end{thm}

\begin{proof}
Since $\spe\mu$ is decomposable, it has a non-trivial idempotent endomorphism $\phi$ determined by
\[
\phi(z_\mu) = \; \sum_{\mathclap{\ttt \in \std \mu}} \; a_\ttt v_\ttt \quad \text{for some } a_\ttt \in \bbf.
\]
Via the embedding $\mathscr{R}_{m} \hookrightarrow \mathscr{R}_{m+n}$, it is easy to see that $\phi$ also defines an idempotent endomorphism of $\spe{(\nu, \mu)}$.
For the other Specht module,
\[
\hat\phi(z_{(\mu,\nu)}) = \; \sum_{\mathclap{\ttt \in \std \mu}} \; a_\ttt \operatorname{shift}_m(\psi_\ttt) z_{(\mu,\nu)}
\]
is a non-trivial idempotent endomorphism of $\spe{(\mu,\nu)}$, where $\operatorname{shift}_m:\mathscr{R}_{m} \hookrightarrow \mathscr{R}_{m+n}$ is the \emph{shift map} (cf.~\cite[\S 2.6.1]{fs16}).
\end{proof}

\begin{rem}
\cref{lev1tolev2} readily extends to higher levels, i.e.~we can use \cref{lev1murph,lev1spey} to construct many decomposable Specht modules in higher levels, regardless of the chosen $\kappa$.
Similarly, for any finite $e$, we may embed the bihooks of \cref{mainresult} into higher levels whenever we have a repeat in the $e$-multicharge.
\end{rem}

The following result is the $e=2$ extension of \cref{mainresult}.

\begin{thm}\label{kejefore=2}
Let $\kappa = (0,0)$, and let $\la = ((2k + a, 1^b),(2j + a, 1^b))$, $((b+1, 1^{2j+a-1}), (b+1, 1^{2k+a-1}))$, $((2k+a, 1^b),(a, 1^{2j+b}))$, or $((a,1^{2k+b}),(2j+a, 1^b))$ for some $j,k\geq 1$, $0< a \leq 2$ and $0 \leq b < 2$ with $a+b \neq 2$, or for $a=b=0$.
\begin{enumerate}[label=(\roman*)]
\item For $j, k > 1$, if $j+k$ is even and $\nchar \bbf \neq 2$, or if $j+k$ is odd, then $\spe\la$ is decomposable.

\item If $j=1$ or $k=1$, then $\spe\la$ is decomposable if and only if $\nchar\bbf \nmid j+k$.
\end{enumerate}
\end{thm}

\begin{proof}
We extend \cref{mainresult} and use the fact that $\spe{((2k),(2j))} \cong \spe{((2k),(1^{2j}))}$ when $e=2$, $\kappa_1=\kappa_2$.
(In general, their Specht presentations only differ in the idempotent relation, and if $e=2$ then the residues in the second components coincide.)
\end{proof}

\begin{conj}\label{e2conj}
Let $e=2$.
If $\kappa_1 \neq \kappa_2$, then \cref{lev1tolev2} provides a complete list of decomposable Specht modules indexed by bihooks.
If $\kappa_1=\kappa_2$, and $\nchar \bbf = 0$, then \cref{smallbihooks,lev1tolev2,kejefore=2} provide a complete list of decomposable Specht modules indexed by bihooks.
\end{conj}

\begin{rem}
For $\nchar \bbf \in \{0,2,3,5\}$, we have checked in GAP that \cref{e2conj} holds for all i) $n\leq 14$ if $\kappa_1 \neq \kappa_2$; ii) $n\leq 13$ if $\kappa_1 = \kappa_2$.
If $\kappa_1 = \kappa_2$ and $\nchar \bbf = 2$, we note that the decomposables appearing in the remark after \cref{enot2conj} are also decomposable here.
However, we have also found further examples of decomposable Specht modules not accounted for.
Namely the bihooks $((3,1^2),(3))$, $((7,1^2),(3))$, $((5,1^4),(3))$, $((7),(3,1^2))$, and $((5,1^4),(5))$ if $\nchar \bbf = 2$, $((5,1^2),(3))$, $((5,1^2),(5))$, and $((5,1^2),(7))$ if $\nchar \bbf = 3$, and $((9,1^2),(3))$ if $\nchar \bbf = 5$, along with all bihooks obtained from these by conjugating, transposing, and our induction arguments, index decomposable Specht modules if $\kappa_1 = \kappa_2$.
This list is exhaustive for $\nchar\bbf\in\{0,2,3,5\}$ and $n\leq 13$.
\end{rem}

\section{Proof of \texorpdfstring{\cref{smallphivt,prop:cancellation}}{Propositions~\ref{smallphivt} and~\ref{prop:cancellation}}}
\label{sec:calc}

In this section, we complete the long calculations necessary in proving \cref{smallphivt,prop:cancellation}.
We begin by setting out some notation which will hopefully help the reader follow the calculations.

For any reduced expression $w = s_{r_1}\dots s_{r_m}\in \sss$, observe that

\begin{align*}
\psi_r\psi_{r\pm 1}\psi_r\psi_wz_\la
&=\psi_r\psi_{r\pm 1}\psi_r\psi_w e(\bfi_\la)z_\la\\
&=\psi_r\psi_{r\pm 1}\psi_r e(s_{r_1}\dots s_{r_m}\cdot \bfi_\la)\psi_wz_\la\\
&=\psi_r\psi_{r\pm 1}\psi_r e(i_1,\dots,i_r,i_{r+1},i_{r+2},\dots,i_n)\psi_wz_\la.
\end{align*}

Since the KLR `braid relations' only depend on the residues $i_r$, $i_{r+1}$ and $i_{r+2}$ of the idempotent $e(s_{r_1}\dots s_{r_m}\cdot \bfi_\la)$, we will instead write the above expression as
\[
(\psi_r\psi_{r\pm 1}\psi_r (i_r,i_{r+1},i_{r+2}))\psi_w z_\la.
\]
Similarly, since $y_{r+1}\psi_r e(\bfi)\psi_w z_\la$ only depends on the residues $i_r$ and $i_{r+1}$, we will write this expression as $(y_{r+1}\psi_r(i_r, i_{r+1}))\psi_w z_\la$.

In fact, whenever we apply the KLR relations in our computations, we will analogously abbreviate idempotents to the two or three necessary, consecutive residues, to help the reader identify which relations are being applied, and which case of the relation is applicable.

For both propositions, we will break the calculation apart with several preliminary lemmas.
First, we will focus on \cref{smallphivt}.

\subsection{Proof of \texorpdfstring{\cref{smallphivt}}{Proposition~\ref{smallphivt}}}
\begin{lem}\label{smallkeylemmas}
Suppose that $e>3$.

\begin{enumerate}[label=(\roman*)]

\item\label{smallkeylemmas1} Let $1\leq r\leq a-2$ and $1\leq s\leq e-r-1$.
Then
\[
\psi_{2e-2r-s} \psid{2e-2r-s-1}{e-r} \chaind {2e-2r-s} {e-r+1} \clapdots \chaind {2e-r-s-1} {e} z_\la = 0.
\]

\item\label{smallkeylemmas2} Let $1\leq r\leq a-1$.
Then
\[
\psid{2e-r-1}{2e-2r+1} \psid{2e-2r+1}{e-r+2} \chaind{2e-2r+2}{e-r+3} \clapdots \chaind{2e-r-1}{e} z_\la
=
\psid{2e-2r}{e-r+2} \chaind{2e-2r+1}{e-r+3} \clapdots \chaind{2e-r-2}{e} z_\la.
\]

\item\label{smallkeylemmas3} Let $1\leq r\leq a-1$.
Then
\[
\psid{e+b+1}{b+2} \chaind{e+b+2}{b+3} \clapdots \chaind{2e-r-1}{e-r} \psid{2e-2r}{e-r+1} \chaind{2e-2r+1}{e-r+2} \clapdots \chaind{2e-r-1}{e} z_\la
=
\psid{2b+2}{b+2} \chaind{2b+3}{b+3} \clapdots \chaind{e+b}{e} z_\la.
\]
\end{enumerate}
\end{lem}

\begin{proof}
\begin{enumerate}[label=(\roman*)]
\item We argue by reverse induction on $s$.
For the base case, we let $s = e-r-1$.
Then
\begin{align*}
&\psi_{2e-2r-s} \psid{2e-2r-s-1}{e-r} \chaind {2e-2r-s} {e-r+1} \clapdots \chaind {2e-r-s-1} {e} z_\la\\
= \; & (\psi_{e-r+1} \psi_{e-r} \psi_{e-r+1} (a{-}r{-}1,a{-}r,0)) \psi_{e-r+2} \psi_{e-r+3} \dots \psi_e z_\la\\
= \; & \psi_{e-r} \psi_{e-r+1} \dots \psi_e \psi_{e-r} z_\la\\
= \; & 0.
\end{align*}
Now suppose that $s<e-r-1$.
Then
\begin{align*}
&\psi_{2e-2r-s} \psid{2e-2r-s-1}{e-r} \chaind {2e-2r-s} {e-r+1} \clapdots \chaind {2e-r-s-1} {e} z_\la\\
= \; & (\psi_{2e-2r-s} \psi_{2e-2r-s-1} \psi_{2e-2r-s} (a{-}r{-}1, a{-}r, a{-}r{-}s{-}1)) \psid{2e-2r-s-2}{e-r} \chaind {2e-2r-s-1} {e-r+1} \cdot\\
&\psid{2e-2r-s+1}{e-r+2} \chaind {2e-2r-s+2} {e-r+3} \clapdots \chaind {2e-r-s-1} {e} z_\la\\
= \; & \psi_{2e-2r-s-1} \psi_{2e-2r-s} \dots \psi_{2e-r-s-1} \psi_{2e-2r-s-1} \psid{2e-2r-s-2}{e-r} \chaind {2e-2r-s-1} {e-r+1} \clapdots \chaind {2e-r-s-2} {e} z_\la\\
= \; & 0 \text{ by the induction hypothesis.}
\end{align*}

\item We argue by induction on $r$.
For the base case, when $r=1$, both sides of the inequality are equal to $z_\la$ by definition.
Thus, we may assume that $r>1$.
Then
\begin{align*}
&\psid{2e-r-1}{2e-2r+1} \psid{2e-2r+1}{e-r+2} \chaind{2e-2r+2}{e-r+3} \clapdots \chaind{2e-r-1}{e} z_\la\\
= \; & \psid{2e-r-1}{2e-2r+2} (\psi_{2e-2r+1}^2 (a{-}r{+}1,a{-}r{-}1)) \psid{2e-2r}{e-r+2} \psid{2e-2r+2}{e-r+3} \chaind{2e-2r+3}{e-r+4} \clapdots \chaind{2e-r-1}{e} z_\la\\
= \; & \psid{2e-2r}{e-r+2} \psid{2e-r-1}{2e-2r+2} \psid{2e-2r+2}{e-r+3} \chaind{2e-2r+3}{e-r+4} \clapdots \chaind{2e-r-1}{e} z_\la\\
= \; & \psid{2e-2r}{e-r+2} \chaind{2e-2r+1}{e-r+3} \clapdots \chaind{2e-r-2}{e} z_\la \text{ by the induction hypothesis.}
\end{align*}

\item We argue by reverse induction on $r$.
For the base case, when $r = a-1$, both sides of the inequality are equal by definition, so there is nothing to prove.
So we assume that $r<a-1$.
Then
\begin{align*}
&\psid{e+b+1}{b+2} \chaind{e+b+2}{b+3} \clapdots \chaind{2e-r-1}{e-r} \psid{2e-2r}{e-r+1} \chaind{2e-2r+1}{e-r+2} \clapdots \chaind{2e-r-1}{e} z_\la\\
= \; & \psid{e+b+1}{b+2} \chaind{e+b+2}{b+3} \clapdots \chaind{2e-r-2}{e-r-1} \psid{2e-r-1}{2e-2r+1}
(\psi_{2e-2r} \psi_{2e-2r-1} \psi_{2e-2r}(a{-}r{-}1, a{-}r, a{-}r{-}1))\cdot\\
&\psid{2e-2r-2}{e-r} \chaind{2e-2r-1}{e-r+1} \psid{2e-2r+1}{e-r+2} \chaind{2e-2r+2}{e-r+3} \clapdots \chaind{2e-r-1}{e} z_\la\\
= \; & \psid{e+b+1}{b+2} \chaind{e+b+2}{b+3} \clapdots \chaind{2e-r-2}{e-r-1} \psid{2e-r-1}{2e-2r+1} \psiu{2e-2r-1}{2e-r-1} \psi_{2e-2r-1} \psid{2e-2r-2}{e-r} \chaind{2e-2r-1}{e-r+1} \clapdots \chaind{2e-r-2}{e} z_\la\\
&+ \psid{e+b+1}{b+2} \chaind{e+b+2}{b+3} \clapdots \chaind{2e-r-2}{e-r-1} \psid{2e-2r-2}{e-r} \chaind{2e-2r-1}{e-r+1}
\left(\psid{2e-r-1}{2e-2r+1} \psid{2e-2r+1}{e-r+2} \chaind{2e-2r+2}{e-r+3} \clapdots \chaind{2e-r-1}{e} z_\la \right)\\
= \; & 0 + \psid{e+b+1}{b+2} \chaind{e+b+2}{b+3} \clapdots \chaind{2e-r-2}{e-r-1} \psid{2e-2r-2}{e-r} \chaind{2e-2r-1}{e-r+1}
\psid{2e-2r}{e-r+2} \chaind{2e-2r+1}{e-r+3} \clapdots \chaind{2e-r-2}{e} z_\la \text{ by parts~\ref{smallkeylemmas1} and~\ref{smallkeylemmas2}}\\
= \; & \psid{2b+2}{b+2} \chaind{2b+3}{b+3} \clapdots \chaind{e+b}{e} z_\la \text{ by the induction hypothesis.}\qedhere
\end{align*}
\end{enumerate}
\end{proof}

\begin{lem}\label{smallkeylemmasctd}
Suppose that $e>3$.
	
\begin{enumerate}[label=(\roman*)]

\item\label{smallkeylemmasctd1} Let $1\leq r \leq b$, $0\leq s \leq a-3$, and $0\leq k\leq b$.
Then
\begin{align*}
&\psid{2b-r+s+2}{2b-r+s+3-k} \psid{e+b-r}{b-r+s+2} \psid{2b-r+s+3-k}{b-r+s+3} \psid{2b-r+s+4}{b-r+s+4} \chaind{2b-r+s+5}{b-r+s+5} \clapdots \chaind{e+b}{e} \psid{b-r+s+1}{b-r+2} \psid{e-r}{b-r+3} \chaind{e-r+1}{b-r+4} \clapdots \chaind{e-2}{b+1} z_\la\\
\\
&=\psid{2b-r+s+2}{2b-r+s+2-k} \psid{e+b-r}{b-r+s+2} \psid{2b-r+s+2-k}{b-r+s+3} \psid{2b-r+s+4}{b-r+s+4} \chaind{2b-r+s+5}{b-r+s+5} \clapdots \chaind{e+b}{e} \psid{b-r+s+1}{b-r+2} \psid{e-r}{b-r+3} \chaind{e-r+1}{b-r+4} \clapdots \chaind{e-2}{b+1} z_\la.
\end{align*}

\item\label{smallkeylemmasctd2} Let $1\leq r \leq b$ and $0\leq s \leq a-3$.
Then
\begin{align*}
&\psid{2b-r+2}{b-r+2} \chaind{2b-r+3}{b-r+3} \clapdots \chaind{2b-r+s+1}{b-r+s+1} \psid{e+b-r}{b-r+s+2} \psid{2b-r+s+3}{b-r+s+3} \chaind{2b-r+s+4}{b-r+s+4} \clapdots \chaind{e+b}{e} \psid{b-r+s+1}{b-r+2} \psid{e-r}{b-r+3} \chaind{e-r+1}{b-r+4} \clapdots \chaind{e-2}{b+1} z_\la\\
\\
&=\psid{2b-r+2}{b-r+2} \chaind{2b-r+3}{b-r+3} \clapdots \chaind{2b-r+s+2}{b-r+s+2} \psid{e+b-r}{b-r+s+3} \psid{2b-r+s+4}{b-r+s+4} \chaind{2b-r+s+5}{b-r+s+5} \clapdots \chaind{e+b}{e} \psid{b-r+s+2}{b-r+2} \psid{e-r}{b-r+3} \chaind{e-r+1}{b-r+4} \clapdots \chaind{e-2}{b+1} z_\la.
\end{align*}

\item\label{smallkeylemmasctd3} Let $1\leq r \leq b$.
Then
\[
\psid{e+b-r}{b-r+2} \psid{2b-r+3}{b-r+3} \chaind{2b-r+4}{b-r+4} \clapdots \chaind{e+b}{e} \psid{e-r}{b-r+3} \chaind{e-r+1}{b-r+4} \clapdots \chaind{e-2}{b+1} z_\la
=\psid{2b-r+2}{b-r+2} \chaind{2b-r+3}{b-r+3} \clapdots \chaind{e+b}{e} \psid{e-r-1}{b-r+2} \chaind{e-r}{b-r+3} \clapdots \chaind{e-2}{b+1} z_\la.
\]

\item\label{smallkeylemmasctd4} The equation in part~\ref{smallkeylemmasctd1} also holds if $r=b+1$, $s=0$ and $0\leq k\leq b-1$.
That is,
\begin{align*}
&\psid{b+1}{b+2-k} \psid{e-1}{1} \psid{b+2-k}{2} \psid{b+3}{3} \chaind{b+4}{4} \clapdots \chaind{e+b}{e} \psid{a-1}{2} \chaind{a}{3} \clapdots \chaind{e-2}{b+1} z_\la\\
\\
&= \psid{b+1}{b+1-k} \psid{e-1}{1} \psid{b+1-k}{2} \psid{b+3}{3} \chaind{b+4}{4} \clapdots \chaind{e+b}{e} \psid{a-1}{2} \chaind{a}{3} \clapdots \chaind{e-2}{b+1} z_\la.
\end{align*}

\item\label{smallkeylemmasctd5} We have
\[
\psid{e-1}{1} \chaind{e}{2} \clapdots \chaind{e+b-1}{b+1} \psid{2b+2}{b+2} \chaind{2b+3}{b+3} \clapdots \chaind{e+b}{e} z_\la = \psid{b+1}{2} \psid{e-1}{1} \psi_2 \psid{b+3}{3} \chaind{b+4}{4} \clapdots \chaind{e+b}{e} \psid{a-1}{2} \chaind{a}{3} \clapdots \chaind{e-2}{b+1} z_\la.
\]

\end{enumerate}
\end{lem}

\begin{proof}
\begin{enumerate}[label=(\roman*)]
\item We have
\begin{align*}
&\psid{2b-r+s+2}{2b-r+s+3-k} \psid{e+b-r}{b-r+s+2} \psid{2b-r+s+3-k}{b-r+s+3} \psid{2b-r+s+4}{b-r+s+4} \chaind{2b-r+s+5}{b-r+s+5} \clapdots \chaind{e+b}{e} \psid{b-r+s+1}{b-r+2} \psid{e-r}{b-r+3} \chaind{e-r+1}{b-r+4} \clapdots \chaind{e-2}{b+1} z_\la\\
\\
&=\psid{2b-r+s+2}{2b-r+s+3-k} \psid{e+b-r}{2b-r+s+4-k} (\psi_{2b-r+s+3-k} \psi_{2b-r+s+2-k} \psi_{2b-r+s+3-k} e(\bfj)) \cdot\\
&\qquad \cdot \psid{2b-r+s-1-k}{b-r+s+2} \psid{2b-r+s+2-k}{b-r+s+3}\psid{2b-r+s+4}{b-r+s+4} \chaind{2b-r+s+5}{b-r+s+5} \clapdots \chaind{e+b}{e} \psid{b-r+s+1}{b-r+2} \psid{e-r}{b-r+3} \chaind{e-r+1}{b-r+4} \clapdots \chaind{e-2}{b+1} z_\la\\
\intertext{where positions $2b-r+s+2-k$, $2b-r+s+3-k$, and $2b-r+s+4-k$ of $\bfj$ are positions $b-r+2$, $b+s+2$ and $e+b+1-k$ of $\bfi_\la$, respectively; i.e.~the corresponding residues are $e-b+r-1$, $s+1$, and $e-b+k$, respectively.
Since $1\leq s+1 \leq a-2$ and $a \leq e-b+r-1 \leq e-1$, it is clear that the corresponding braid relation never produces an error term, and the result follows:}
&=\psid{2b-r+s+2}{2b-r+s+2-k} \psid{e+b-r}{b-r+s+2} \psid{2b-r+s+2-k}{b-r+s+3} \psid{2b-r+s+4}{b-r+s+4} \chaind{2b-r+s+5}{b-r+s+5} \clapdots \chaind{e+b}{e} \psid{b-r+s+1}{b-r+2} \psid{e-r}{b-r+3} \chaind{e-r+1}{b-r+4} \clapdots \chaind{e-2}{b+1} z_\la.
\end{align*}

\item Apply part~\ref{smallkeylemmasctd1} for $k=0$, then $k=1$, and so on up to and including the case $k=b$.

\item Apply part~\ref{smallkeylemmasctd2} for $s=0$, then $s=1$, and so on up to and including the case $s=a-3$.

\item The proof is identical to part~\ref{smallkeylemmasctd1}, except now we notice that the third residue of the relevant triple is $e-b\leq e-b+k\leq e-1$, while the second is $1$.

\item Apply part~\ref{smallkeylemmasctd3} for $r=1$, then $r=2$, and so on up to and including the case $r=b$, to yield
\[
\psid{e-1}{1} \psid{e}{2} \chaind{e+1}{3} \clapdots \chaind{e+b-1}{b+1} \psid{2b+2}{b+2} \chaind{2b+3}{b+3} \clapdots \chaind{e+b}{e} z_\la = \psid{e-1}{1} \psid{b+2}{2} \chaind{b+3}{3} \clapdots \chaind{e+b}{e} \psid{a-1}{2} \chaind{a}{3} \clapdots \chaind{e-2}{b+1} z_\la.
\]
Then, applying part~\ref{smallkeylemmasctd4} for $k=0$, then $k=1$, and so on up to and including the case $k=b-1$ gives
\[
\psid{e-1}{1} \psid{b+2}{2} \chaind{b+3}{3} \clapdots \chaind{e+b}{e} \psid{a-1}{2} \chaind{a}{3} \clapdots \chaind{e-2}{b+1} z_\la =
\psid{b+1}{2} \psid{e-1}{1} \psi_2 \psid{b+3}{3} \chaind{b+4}{4} \clapdots \chaind{e+b}{e} \psid{a-1}{2} \chaind{a}{3} \clapdots \chaind{e-2}{b+1} z_\la.\qedhere
\]
\end{enumerate}
\end{proof}

\begin{lem}\label{lem:nextreduction}
Suppose that $e>3$.
	
\begin{enumerate}[label=(\roman*)]
	
\item\label{lem:nextreduction1}
For $3\leq x \leq a-1$, we have
\begin{align*}
&\psid{e-1}{b+x} \psid{b+1}{2} \chaind{b+2}{3} \clapdots \chaind{b+x-1}{x} \psid{b+x}{x} \chaind{b+x+1}{x+1} \clapdots \chaind{e+b}{e} \psid{a-1}{2} \chaind{a}{3} \clapdots \chaind{e-2}{b+1} z_\la\\
= \; &\psid{e-1}{b+x+1} \psid{b+1}{2} \chaind{b+2}{3} \clapdots \chaind{b+x}{x+1} \psid{b+x+1}{x+1} \chaind{b+x+2}{x+2} \clapdots \chaind{e+b}{e} \psid{a-1}{2} \chaind{a}{3} \clapdots \chaind{e-2}{b+1} z_\la.
\end{align*}

\item\label{lem:nextreduction2}
We have
\[
\psid{b+1}{2} \psid{e-1}{3} \psid{b+3}{3} \chaind{b+4}{4} \clapdots \chaind{e+b}{e} \psid{a-1}{2} \chaind{a}{3} \clapdots \chaind{e-2}{b+1} z_\la = \psid{b+1}{2} \chaind{b+2}{3} \clapdots \chaind{e-1}{a} \psid{e}{a} \chaind{e+1}{a+1} \clapdots \chaind{e+b}{e} \psid{a-1}{2} \chaind{a}{3} \clapdots \chaind{e-2}{b+1} z_\la.
\]
\end{enumerate}
\end{lem}

\begin{proof}
\begin{enumerate}[label=(\roman*)]
\item
We have
\begin{align*}
&\psid{e-1}{b+x} \psid{b+1}{2} \chaind{b+2}{3} \clapdots \chaind{b+x-1}{x} \psid{b+x}{x} \chaind{b+x+1}{x+1} \clapdots \chaind{e+b}{e} \psid{a-1}{2} \chaind{a}{3} \clapdots \chaind{e-2}{b+1} z_\la\\
= \; &\psid{e-1}{b+x+1} \psid{b+1}{2} \chaind{b+2}{3} \clapdots \chaind{b+x-2}{x-1} (\psi_{b+x} \psi_{b+x-1} \psi_{b+x}(0, x-1, -b))\cdot\\
&\psid{b+x-2}{x} \psid{b+x-1}{x} \psid{b+x+1}{x+1} \chaind{b+x+2}{x+2} \clapdots \chaind{e+b}{e} \psid{a-1}{2} \chaind{a}{3} \clapdots \chaind{e-2}{b+1} z_\la\\
= \; &\psid{e-1}{b+x+1} \psid{b+1}{2} \chaind{b+2}{3} \clapdots \chaind{b+x-2}{x-1} \psi_{b+x-1} \psi_{b+x} (\psi_{b+x-1} \psi_{b+x-2} \psi_{b+x-1}(0,x-1, -b+1))\cdot\\
&\psid{b+x-3}{x} \psid{b+x-2}{x} \psid{b+x+1}{x+1} \chaind{b+x+2}{x+2} \clapdots \chaind{e+b}{e} \psid{a-1}{2} \chaind{a}{3} \clapdots \chaind{e-2}{b+1} z_\la\\
\vdots \;\; &\\
= \; &\psid{e-1}{b+x+1} \psid{b+1}{2} \chaind{b+2}{3} \clapdots \chaind{b+x-2}{x-1} \psid{b+x-1}{x+1} \psid{b+x}{x+2} (\psi_{x+1} \psi_x \psi_{x+1}(0, x-1, -1))\cdot\\
= \; & \psi_x \psid{b+x+1}{x+1} \chaind{b+x+2}{x+2} \clapdots \chaind{e+b}{e} \psid{a-1}{2} \chaind{a}{3} \clapdots \chaind{e-2}{b+1} z_\la\\
= \; &\psid{e-1}{b+x+1} \psid{b+1}{2} \chaind{b+2}{3} \clapdots \chaind{b+x}{x+1} (\psi_x^2(x-1, 0)) \psid{b+x+1}{x+1} \chaind{b+x+2}{x+2} \clapdots \chaind{e+b}{e} \psid{a-1}{2} \chaind{a}{3} \clapdots \chaind{e-2}{b+1} z_\la\\
= \; &\psid{e-1}{b+x+1} \psid{b+1}{2} \chaind{b+2}{3} \clapdots \chaind{b+x}{x+1} \psid{b+x+1}{x+1} \chaind{b+x+2}{x+2} \clapdots \chaind{e+b}{e} \psid{a-1}{2} \chaind{a}{3} \clapdots \chaind{e-2}{b+1} z_\la
\end{align*}

\item By applying part~\ref{lem:nextreduction1} repeatedly, we have
\begin{align*}
\psid{b+1}{2} \psid{e-1}{3} \psid{b+3}{3} \chaind{b+4}{4} \clapdots \chaind{e+b}{e} \psid{a-1}{2} \chaind{a}{3} \clapdots \chaind{e-2}{b+1} z_\la
=&\psid{e-1}{b+3} \psid{b+1}{2} \psid{b+2}{3} \psid{b+3}{3} \chaind{b+4}{4} \clapdots \chaind{e+b}{e} \psid{a-1}{2} \chaind{a}{3} \clapdots \chaind{e-2}{b+1} z_\la\\
= \;&\psid{e-1}{b+4} \psid{b+1}{2} \chaind{b+2}{3} \chaind{b+3}{4} \psid{b+4}{4} \psid{b+5}{5} \clapdots \chaind{e+b}{e} \psid{a-1}{2} \chaind{a}{3} \clapdots \chaind{e-2}{b+1} z_\la\\
\vdots \;\; &\\
= \;&\psid{b+1}{2} \chaind{b+2}{3} \clapdots \chaind{e-1}{a} \psid{e}{a} \chaind{e+1}{a+1} \clapdots \chaind{e+b}{e} \psid{a-1}{2} \chaind{a}{3} \clapdots \chaind{e-2}{b+1} z_\la.\qedhere
\end{align*}
\end{enumerate}
\end{proof}

\begin{lem}\label{lem:nexttreduction}
Suppose that $e>3$.
For $0\leq x \leq b-1$ and $1 \leq y \leq b-x+1$, we have
\[
\psi_{a+2x+y} \psid{e+x+y}{a+x+y} \chaind{e+x+y+1}{a+x+y+1} \clapdots \chaind{e+b}{e} \psid{a-1}{2} \chaind{a}{3} \clapdots \chaind{e-2}{b+1} z_\la = 0
\]
\end{lem}

\begin{proof}
We fix $x$, and prove the statement by reverse induction on $y$.
If $y = b-x+1$, the expression is
\[
\psi_{e+x+1} \psid{a-1}{2} \chaind{a}{3} \clapdots \chaind{e-2}{b+1} z_\la = 0.
\]
If $y< b-x+1$, then
\begin{align*}
& \psi_{a+2x+y} \psid{e+x+y}{a+x+y} \chaind{e+x+y+1}{a+x+y+1} \clapdots \chaind{e+b}{e} \psid{a-1}{2} \chaind{a}{3} \clapdots \chaind{e-2}{b+1} z_\la\\
= \; &\psid{e+x+y}{a+2x+y+2} (\psi_{a+2x+y} \psi_{a+2x+y+1} \psi_{a+2x+y}(\alpha, -x, -x-1)) \psid{a+2x+y-1}{a+x+y} \chaind{e+x+y+1}{a+x+y+1} \clapdots \chaind{e+b}{e} \psid{a-1}{2} \chaind{a}{3} \clapdots \chaind{e-2}{b+1} z_\la\\
= \;&\psid{e+x+y}{a+x+y} \psi_{a+2x+y+1} \psid{e+x+y+1}{a+x+y+1} \chaind{e+x+y+2}{a+x+y+2} \clapdots \chaind{e+b}{e} \psid{a-1}{2} \chaind{a}{3} \clapdots \chaind{e-2}{b+1} z_\la\\
= \;&0 \text{ by the induction hypothesis,}
\end{align*}
where $\alpha = -y-x-1$ if $y\leq b-x-1$ and $\alpha = a-1$ if $y= b-x$.
\end{proof}

\begin{lem}\label{lem:penulttreduction}
Suppose that $e>3$.
For $0\leq x \leq b-1 < e-2$ or $0\leq x\leq b-2 = e-3$, we have
\begin{enumerate}[label=(\roman*)]
\Item\label{lem:penulttreduction1}
\begin{align*}
\left(\psiu{a+x}{a+2x-1} \psid{a+2x-1}{a+x}\right)
\psid{e+x+1}{a+x+1} \chaind{e+x+2}{a+x+2} \clapdots \chaind{e+b}{e} \psid{a-1}{2} \chaind{a}{3} \clapdots \chaind{e-2}{b+1} z_\la
= \psid{e+x+1}{a+x+1} \chaind{e+x+2}{a+x+2} \clapdots \chaind{e+b}{e} \psid{a-1}{2} \chaind{a}{3} \clapdots \chaind{e-2}{b+1} z_\la;
\end{align*}

\Item\label{lem:penulttreduction2}
\begin{align*}
&\psid{b+1}{2} \chaind{b+2}{3} \clapdots \chaind{e-2}{a-1} \psid{e-1}{a+x} \chaind{e}{a+x+1} \clapdots \chaind{e+x-1}{a+2x} \psid{e+x}{a+x} \chaind{e+x+1}{a+x+1} \clapdots \chaind{e+b}{e} \psid{a-1}{2} \chaind{a}{3} \clapdots \chaind{e-2}{b+1} z_\la\\
&= -\psid{b+1}{2} \chaind{b+2}{3} \clapdots \chaind{e-2}{a-1} \psid{e-1}{a+(x+1)} \chaind{e}{a+(x+1)+1} \clapdots \chaind{e+(x+1)-1}{a+2(x+1)} \psid{e+(x+1)}{a+(x+1)} \chaind{e+(x+1)+1}{a+(x+1)+1} \clapdots \chaind{e+b}{e} \psid{a-1}{2} \chaind{a}{3} \clapdots \chaind{e-2}{b+1} z_\la.
\end{align*}
\end{enumerate}
\end{lem}

\begin{proof}
\begin{enumerate}[label=(\roman*)]
\item

Firstly, it is obvious that the statement holds when $x=0$, since the term $\psiu{a}{a-1}\psid{a-1}{a}$ is trivial.
Now assuming that $x>0$, we have

\begin{align*}
&\psiu{a+x}{a+2x-1} \psid{a+2x-1}{a+x} \psid{e+x+1}{a+x+1} \chaind{e+x+2}{a+x+2} \clapdots \chaind{e+b}{e} \psid{a-1}{2} \chaind{a}{3} \clapdots \chaind{e-2}{b+1} z_\la\\
= \; &\psiu{a+x}{a+2x-2} (\psi_{a+2x-1}^2(-x-1,-x+1)) \psid{a+2x-2}{a+x} \psid{e+x+1}{a+x+1} \chaind{e+x+2}{a+x+2} \clapdots \chaind{e+b}{e} \psid{a-1}{2} \chaind{a}{3} \clapdots \chaind{e-2}{b+1} z_\la\\
= \; &\psiu{a+x}{a+2x-3} (\psi_{a+2x-2}^2(-x-1,-x+2)) \psid{a+2x-3}{a+x} \psid{e+x+1}{a+x+1} \chaind{e+x+2}{a+x+2} \clapdots \chaind{e+b}{e} \psid{a-1}{2} \chaind{a}{3} \clapdots \chaind{e-2}{b+1} z_\la\\
\vdots \;\; &\\
= \; &(\psi_{a+x}^2(-x-1,0)) \psid{e+x+1}{a+x+1} \chaind{e+x+2}{a+x+2} \clapdots \chaind{e+b}{e} \psid{a-1}{2} \chaind{a}{3} \clapdots \chaind{e-2}{b+1} z_\la\\
= \; &\psid{e+x+1}{a+x+1} \chaind{e+x+2}{a+x+2} \clapdots \chaind{e+b}{e} \psid{a-1}{2} \chaind{a}{3} \clapdots \chaind{e-2}{b+1} z_\la.
\end{align*}

\item Applying \cref{lem:nexttreduction}, we have
\begin{align*}
&\psid{b+1}{2} \chaind{b+2}{3} \clapdots \chaind{e-2}{a-1} \psid{e-1}{a+x} \chaind{e}{a+x+1} \clapdots \chaind{e+x-1}{a+2x} \psid{e+x}{a+x} \chaind{e+x+1}{a+x+1} \clapdots \chaind{e+b}{e} \psid{a-1}{2} \chaind{a}{3} \clapdots \chaind{e-2}{b+1} z_\la\\
= \; &\psid{b+1}{2} \chaind{b+2}{3} \clapdots \chaind{e-2}{a-1} \psid{e-1}{a+x} \chaind{e}{a+x+1} \clapdots \chaind{e+x-2}{a+2x-1} \psid{e+x-1}{a+2x+1} \psi_{a+2x} \psid{e+x}{a+x} \chaind{e+x+1}{a+x+1} \clapdots \chaind{e+b}{e} \psid{a-1}{2} \chaind{a}{3} \clapdots \chaind{e-2}{b+1} z_\la\\
= \; &\psid{b+1}{2} \chaind{b+2}{3} \clapdots \chaind{e-2}{a-1} \psid{e-1}{a+x} \chaind{e}{a+x+1} \clapdots \chaind{e+x-2}{a+2x-1} \psid{e+x-1}{a+2x+1} \chaind{e+x}{a+2x+2}\cdot\\
&(\psi_{a+2x} \psi_{a+2x+1} \psi_{a+2x}(-x-1, -x, -x-1)) \psid{a+2x-1}{a+x}  \psid{e+x+1}{a+x+1} \chaind{e+x+2}{a+x+2} \clapdots \chaind{e+b}{e} \psid{a-1}{2} \chaind{a}{3} \clapdots \chaind{e-2}{b+1} z_\la\\
= \; &0 - \psid{b+1}{2} \chaind{b+2}{3} \clapdots \chaind{e-2}{a-1} \psid{e-1}{a+x} \chaind{e}{a+x+1} \clapdots \chaind{e+x-2}{a+2x-1}\cdot\\
\; & \psid{e+x-1}{a+2x+1} \chaind{e+x}{a+2x+2} \psid{a+2x-1}{a+x}  \psid{e+x+1}{a+x+1} \chaind{e+x+2}{a+x+2} \clapdots \chaind{e+b}{e} \psid{a-1}{2} \chaind{a}{3} \clapdots \chaind{e-2}{b+1} z_\la\\
= \; &- \psid{b+1}{2} \chaind{b+2}{3} \clapdots \chaind{e-2}{a-1} \psid{e-1}{a+x+1} \chaind{e}{a+x+2} \clapdots \chaind{e+x}{a+2x+2} \left(\psiu{a+x}{a+2x-1} \psid{a+2x-1}{a+x}\right) \psid{e+x+1}{a+x+1} \chaind{e+x+2}{a+x+2} \clapdots \chaind{e+b}{e} \psid{a-1}{2} \chaind{a}{3} \clapdots \chaind{e-2}{b+1} z_\la\\
= \; &-\psid{b+1}{2} \chaind{b+2}{3} \clapdots \chaind{e-2}{a-1} \psid{e-1}{a+(x+1)} \chaind{e}{a+(x+1)+1} \clapdots \chaind{e+(x+1)-1}{a+2(x+1)} \psid{e+(x+1)}{a+(x+1)} \chaind{e+(x+1)+1}{a+(x+1)+1} \clapdots \chaind{e+b}{e} \psid{a-1}{2} \chaind{a}{3} \clapdots \chaind{e-2}{b+1} z_\la,\\
&\text{ by part~\ref{lem:penulttreduction1}.}\qedhere
\end{align*}
\end{enumerate}
\end{proof}

\begin{lem}\label{lem:penultreduction}
Suppose that $e>3$.
For $2\leq x \leq b+1$, we have
\[
\psiu{x}{a+x-3} \psid{a+x-3}{x} \chaind{a+x-2}{x+1} \clapdots \chaind{e-2}{b+1} z_\la = \psid{a+x-2}{x+1} \chaind{a+x-1}{x+2} \clapdots \chaind{e-2}{b+1} z_\la.
\]
\end{lem}

\begin{proof}
Firstly, if $a<3$, then the left-hand side and the right-hand side are both equal to the generator $z_\la$.
If $a\geq 3$, we have 
\begin{align*}
\psiu{x}{a+x-3} \psid{a+x-3}{x} \chaind{a+x-2}{x+1} \clapdots \chaind{e-2}{b+1} z_\la &= \psiu{x}{a+x-4} (\psi_{a+x-3}^2(-x+1,a-2)) \psid{a+x-4}{x} \psid{a+x-2}{x+1} \chaind{a+x-1}{x+2} \clapdots \chaind{e-2}{b+1} z_\la\\
&= \psiu{x}{a+x-5} (\psi_{a+x-4}^2(-x+1,a-3)) \psid{a+x-5}{x} \psid{a+x-2}{x+1} \chaind{a+x-1}{x+2} \clapdots \chaind{e-2}{b+1} z_\la\\
&\;\; \vdots\\
&= (\psi_x^2(-x+1,1)) \psid{a+x-2}{x+1} \chaind{a+x-1}{x+2} \clapdots \chaind{e-2}{b+1} z_\la\\
&= \psid{a+x-2}{x+1} \chaind{a+x-1}{x+2} \clapdots \chaind{e-2}{b+1} z_\la.\qedhere
\end{align*}
\end{proof}

Repeated application of \cref{lem:penultreduction} yields the following corollary.

\begin{cor}\label{lem:finalreduction}
Suppose that $e>3$.
We have
\[
\psiu{b+1}{e-2} \chainu{b}{e-3} \clapdots \chainu{2}{a-1} \psid{a-1}{2} \chaind{a}{3} \clapdots \chaind{e-2}{b+1} z_\la = z_\la
\]
\end{cor}

We are now ready to prove \cref{smallphivt}, using the above lemmas.

\begin{proof}[Proof of \cref{smallphivt}]
In order to prove that $\phi(v_\ttt) = (-1)^{b+1} 2 v_\ttt$, it suffices to prove that
\[
\psid{2e-1}{e} \!\!\! \cdot \, \psid{e}{1} \chaind{e+1}{2} \clapdots \chaind{2e-1}{e} z_\la = (-1)^{b+1} 2 \psid{2e-1}{e} z_\la \quad \text{or} \quad (-1)^{e-1} 2 \psid{2e-1}{e} z_\la.
\]
We first suppose that $e=3$ and $\la = ((2,1),(2,1))$ or $((1^3),(1^3))$.
Then we have
\begin{align*}
\psid{5}{3}\cdot\psid{3}{1}\psid{4}{2}\psid{5}{3}z_\la
= \psid{5}{4}(\psi_3^2(0,1))\psid{2}{1}\psid{4}{2}\psid{5}{3}z_\la
&= \psid{5}{4}(y_3-y_4)\psid{2}{1}\psid{4}{2}\psid{5}{3}z_\la.
\end{align*}
The first term becomes
\begin{align*}
 \psid{5}{4}(y_3\psi_2(0,2)) \psi_1 \psid{4}{2} \psid{5}{3} z_\la
& = \psid{5}{4} \psi_2(y_2 \psi_1(0,0)) \psid{4}{2} \psid{5}{3} z_\la\\
& = \psid{5}{4} \psi_2(\psi_1 y_1+1) \psid{4}{2} \psid{5}{3} z_\la\\
& = \psid{5}{4} \psi_4(\psi_2 \psi_3 \psi_2(2,0,2))\psid{5}{3} z_\la\\
& = \psid{5}{4} \psi_4 (\psi_3 \psi_2 \psi_3-1) \psid{5}{3} z_\la\\
& = \psid{5}{4} \psid{4}{2}\psi_5(\psi_3 \psi_4 \psi_3(1,0,2)) z_\la - \psi_5 (\psi_4^2(2,1)) \psid{5}{3} z_\la\\
& = \psid{5}{4} \psid{4}{2} \psi_5(\psi_4 \psi_3 \psi_4) z_\la - \psi_5 (y_5 - y_4) \psid{5}{3} z_\la\\
& = 0 - \psi_5 (y_5 \psi_5 (1,1)) \psid{4}{3} z_\la + \psi_5^2 (y_4 \psi_4(1,2)) \psi_3 z_\la\\
& = - \psi_5 (\psi_5 y_6 - 1) \psid{4}{3} z_\la + \psi_5^2 \psid{4}{3} y_5 z_\la\\
&= 0 + \psid{5}{3} z_\la + 0.
\end{align*}
The second term becomes
\begin{align*}
 - \psid{5}{4}\psid{2}{1}(y_4\psi_4(2,1))\psid{3}{2} \psid{5}{3} z_\la
& = - \psid{5}{4} \psid{2}{1}\psid{4}{2} (y_5 \psi_5(1,1)) \psid{4}{3} z_\la\\
& = - \psid{5}{4} \psid{2}{1}\psid{4}{2} (\psi_5 y_6 - 1) \psid{4}{3} z_\la\\
& = -0 + \psid{5}{4} \psid{2}{1} (\psi_4 \psi_3 \psi_4 (2,1,2)) \psi_2 \psi_3 z_\la\\
& = \psid{5}{4}\psid{2}{1} (\psi_3\psi_4\psi_3-1) \psi_2 \psi_3 z_\la\\
& = \psid{5}{4}\psid{2}{1}\psi_3\psi_4(\psi_3\psi_2\psi_3(2,1,0)) z_\la - \psid{5}{4}(\psi_2\psi_1\psi_2(0,2,0)) \psi_3 z_\la\\
& = \psid{5}{4}\psid{2}{1}\psi_3\psi_4(\psi_2\psi_3\psi_2) z_\la - \psid{5}{4}(\psi_1\psi_2\psi_1-1) \psi_3 z_\la\\
& = 0 - 0 + \psid{5}{3}z_\la.
\end{align*}
We now suppose that either $e=3$ and $\la=((3),(3))$, or $e>3$ and $a>2$.
Then we have
\begin{align*}
\psid{2e-1}{e} \!\!\! \cdot \, \psid{e}{1} \chaind{e+1}{2} \clapdots \chaind{2e-1}{e} z_\la &= \psid{2e-1}{e+1} (\psi_e^2(0,a{-}1)) \psid{e-1}{1} \psid{e+1}{2} \chaind{e+2}{3} \clapdots \chaind{2e-1}{e} z_\la\\
&= \psid{2e-1}{e+2} \psid{e-1}{1} (\psi_{e+1}^2(-1,a{-}1)) \psid{e}{2} \psid{e+2}{3} \chaind{e+3}{4} \clapdots \chaind{2e-1}{e} z_\la\\
&= \psid{2e-1}{e+3} \psid{e-1}{1}\psid{e}{2} (\psi_{e+2}^2(-2,a{-}1))  \psid{e+1}{3} \chaind{e+3}{4} \clapdots \chaind{2e-1}{e} z_\la\\
&\;\; \vdots\\
&= \psid{2e-1}{e+b+1} \psid{e-1}{1} \chaind{e}{2} \; \clapdots \chaind{e+b-2}{b} (\psi_{e+b}^2(a,a{-}1)) \psid{e+b-1}{b+1} \psid{e+b+1}{b+2} \chaind{e+b+2}{b+3} \clapdots \chaind{2e-1}{e} z_\la\\
&= \psid{2e-1}{e+b+1} \psid{e-1}{1} \chaind{e}{2} \; \clapdots \chaind{e+b-2}{b} (y_{e+b+1} - y_{e+b}) \psid{e+b-1}{b+1} \psid{e+b+1}{b+2} \chaind{e+b+2}{b+3} \clapdots \chaind{2e-1}{e} z_\la.
\end{align*}
It thus suffices to show that
\begin{align*}
y_{e+b+1} \psid{e-1}{1} \chaind{e}{2} \; \clapdots \chaind{e+b-1}{b+1} \psid{e+b+1}{b+2} \chaind{e+b+2}{b+3} \clapdots \chaind{2e-1}{e} z_\la
&= (-1)^{b+1} \psid{e+b}{e} z_\la\\
&= -y_{e+b} \psid{e-1}{1} \chaind{e}{2} \; \clapdots \chaind{e+b-1}{b+1} \psid{e+b+1}{b+2} \chaind{e+b+2}{b+3} \clapdots \chaind{2e-1}{e} z_\la.
\end{align*}
We will prove the first equality here, for which we have set up the relevant computational lemmas -- the second equality may be proved in a similar manner.
\begin{align*}
& y_{e+b+1} \psid{e-1}{1} \chaind{e}{2} \; \clapdots \chaind{e+b-1}{b+1} \psid{e+b+1}{b+2} \chaind{e+b+2}{b+3} \clapdots \chaind{2e-1}{e} z_\la\\
= \; &\psid{e-1}{1} \chaind{e}{2} \; \clapdots \chaind{e+b-1}{b+1}
(y_{e+b+1} \psi_{e+b+1} (1,a{-}1))  \psid{e+b}{b+2} \psid{e+b+2}{b+3} \chaind{e+b+3}{b+4} \clapdots \chaind{2e-1}{e} z_\la\\
= \; &\psid{e-1}{1} \chaind{e}{2} \; \clapdots \chaind{e+b-1}{b+1} \psid{e+b+1}{b+2} (y_{e+b+2} \psi_{e+b+2} (2,a{-}1)) \psid{e+b+1}{b+3} \psid{e+b+3}{b+4} \chaind{e+b+4}{b+5} \clapdots \chaind{2e-1}{e} z_\la\\
\vdots \; \; &\\
= \; &\psid{e-1}{1} \chaind{e}{2} \; \clapdots \chaind{e+b-1}{b+1} \psid{e+b+1}{b+2} \chaind{e+b+2}{b+3} \clapdots \chaind{2e-2}{e-1} (y_{2e-1} \psi_{2e-1} (a{-}1,a{-}1)) \psid{2e-2}{e} z_\la\\
= \; &-\psid{e-1}{1} \chaind{e}{2} \; \clapdots \chaind{e+b-1}{b+1} \psid{e+b+1}{b+2} \chaind{e+b+2}{b+3} \clapdots \chaind{2e-2}{e-1} \psid{2e-2}{e} z_\la.
\end{align*}
If $e=3$, then this expression becomes
\begin{align*}
-\psid{2}{1}\psid{4}{2}\psid{4}{3}z_{((3),(3))}
& = -\psid{2}{1}(\psi_4\psi_3\psi_4(1,2,1))
\psi_2\psi_3z_{((3),(3))}\\
& = -\psid{2}{1}(\psi_3\psi_4\psi_3+1)
\psi_2\psi_3z_{((3),(3))}\\
& = -\psid{2}{1} \psi_3 \psi_4 (\psi_3\psi_2\psi_3(1,2,0)) z_{((3),(3))} -(\psi_2\psi_1\psi_2(0,1,0))\psi_3z_{((3),(3))}\\
& = -\psid{2}{1}\psi_3\psi_4\psi_2\psi_3\psi_2z_{((3),(3))} - (\psi_1 \psi_2 \psi_1 + 1) \psi_3 z_{((3),(3))}\\
& = 0 - 0 - \psi_3z_{((3),(3))},
\end{align*}
as required.

If $e>3$, then the above expression becomes
\begin{align*}
& -\psid{e-1}{1} \chaind{e}{2} \; \clapdots \chaind{e+b-1}{b+1} \psid{2b+2}{b+2} \chaind{2b+3}{b+3} \clapdots \chaind{e+b}{e} z_\la \text{ by \cref{smallkeylemmas}\ref{smallkeylemmas3}}\\
= \; &-\psid{b+1}{2} \psid{e-1}{1} \psi_2 \psid{b+3}{3} \chaind{b+4}{4} \clapdots \chaind{e+b}{e} \psid{a-1}{2} \chaind{a}{3} \clapdots \chaind{e-2}{b+1} z_\la \text{ by \cref{smallkeylemmasctd}\ref{smallkeylemmasctd5}}\\
= \; &-\psid{b+1}{2} \psid{e-1}{3} (\psi_2 \psi_1 \psi_2 (0,1,0)) \psid{b+3}{3} \chaind{b+4}{4} \clapdots \chaind{e+b}{e} \psid{a-1}{2} \chaind{a}{3} \clapdots \chaind{e-2}{b+1} z_\la\\
= \; &-\psid{b+1}{2} \psid{e-1}{3} (\psi_1 \psi_2 \psi_1 +1) \psid{b+3}{3} \chaind{b+4}{4} \clapdots \chaind{e+b}{e} \psid{a-1}{2} \chaind{a}{3} \clapdots \chaind{e-2}{b+1} z_\la\\
= \; &0 -\psid{b+1}{2} \psid{e-1}{3} \psid{b+3}{3} \chaind{b+4}{4} \clapdots \chaind{e+b}{e} \psid{a-1}{2} \chaind{a}{3} \clapdots \chaind{e-2}{b+1} z_\la.
\end{align*}
Now, we handle the case $b=0$ separately.
In this case, the above expression is
\begin{align*}
-\psid{b+1}{2} \psid{e-1}{3} \psid{b+3}{3} \chaind{b+4}{4} \clapdots \chaind{e+b}{e} \psid{a-1}{2} \chaind{a}{3} \clapdots \chaind{e-2}{b+1} z_\la
&= -\psid{e-1}{3} \psiu{3}{e} z_\la\\
&= -\psid{e-1}{4} \cancel{(\psi_3^2 (2,0))} \psiu{4}{e} z_\la\\
&= -\psid{e-1}{5} \cancel{(\psi_4^2 (3,0))} \psiu{5}{e} z_\la\\
&\;\; \vdots\\
&= -\cancel{(\psi_{e-1}^2(-2,0))} \psi_e z_\la\\
&= -\psi_e z_\la,
\end{align*}
in which case our proof is complete here.\hlabel{b0pf}

Now suppose that $b>0$.
Then by \cref{lem:nextreduction} followed by repeated application of \cref{lem:penulttreduction}\ref{lem:penulttreduction2},
\begin{align*}
&-\psid{b+1}{2} \psid{e-1}{3} \psid{b+3}{3} \chaind{b+4}{4} \clapdots \chaind{e+b}{e} \psid{a-1}{2} \chaind{a}{3} \clapdots \chaind{e-2}{b+1} z_\la\\
= \; &-\psid{b+1}{2} \chaind{b+2}{3} \clapdots \chaind{e-1}{a} \psid{e}{a} \chaind{e+1}{a+1} \clapdots \chaind{e+b}{e} \psid{a-1}{2} \chaind{a}{3} \clapdots \chaind{e-2}{b+1} z_\la\\
= \; &(-1)^2 \psid{b+1}{2} \chaind{b+2}{3} \clapdots \chaind{e-2}{a-1} \psid{e-1}{a+1} \chaind{e}{a+2} \psid{e+1}{a+1} \chaind{e+2}{a+2} \clapdots \chaind{e+b}{e} \psid{a-1}{2} \chaind{a}{3} \clapdots \chaind{e-2}{b+1} z_\la\\
\vdots \; \; &\\
= \; &(-1)^b \psid{b+1}{2} \chaind{b+2}{3} \clapdots \chaind{e-2}{a-1} \psi_{e-1} \psi_e \dots \psi_{e+b-2} \psid{e+b-1}{e-1} \chaind{e+b}{e} \psid{a-1}{2} \chaind{a}{3} \clapdots \chaind{e-2}{b+1} z_\la\\
= \; &(-1)^{b+1} \psid{b+1}{2} \chaind{b+2}{3} \clapdots \chaind{e-2}{a-1} \psid{e+b}{e} \psid{a-1}{2} \chaind{a}{3} \clapdots \chaind{e-2}{b+1} z_\la\\
= \; &(-1)^{b+1} \psid{e+b}{e} \psid{b+1}{2} \chaind{b+2}{3} \clapdots \chaind{e-2}{a-1} \psid{a-1}{2} \chaind{a}{3} \clapdots \chaind{e-2}{b+1} z_\la\\
= \; &(-1)^{b+1} \psid{e+b}{e} \psiu{b+1}{e-2} \chainu{b}{e-3} \clapdots \chainu{2}{a-1} \psid{a-1}{2} \chaind{a}{3} \clapdots \chaind{e-2}{b+1} z_\la\\
= \; &(-1)^{b+1} \psid{e+b}{e} z_\la \text{ by \cref{lem:finalreduction}, completing the proof.}
\end{align*}
If $e>3$ and $\la = ((2,1^{e-2}),(2,1^{e-2}))$ or $((1^e),(1^e))$, then we have
\begin{align*}
\psid{2e-1}{e} \!\!\! \cdot \, \psid{e}{1} \chaind{e+1}{2} \clapdots \chaind{2e-1}{e} z_\la &= \psid{2e-1}{e+1} (\psi_e^2(0,1)) \psid{e-1}{1} \psid{e+1}{2} \chaind{e+2}{3} \clapdots \chaind{2e-1}{e} z_\la\\
&=  \psid{2e-1}{e+1} (y_e - y_{e+1}) \psid{e-1}{1} \psid{e+1}{2} \chaind{e+2}{3} \clapdots \chaind{2e-1}{e} z_\la,
\end{align*}
and the proof may be finished in a similar manner to the other cases.
\end{proof}

\subsection{Proof of \texorpdfstring{\cref{prop:cancellation}}{Propositions~\ref{prop:cancellation}}}

Let $\la = ((ke),(je))$ for some $j,k \geq 1$.
In order to prove \cref{prop:cancellation}, we now look at the action of the KLR generators $\psi_1,\dots,\psi_{n-1}$ on an arbitrary basis element $v_{\ttt}\in\spe\la$, where $\ttt$ does not necessarily lie in $\calt_e$.

\begin{lem}\label{lem:psiaction}
Let $\ttt \in \std\la$, $v_\ttt = v(a_1, \dots, a_{je})$, $1\leq r <n$, and $1\leq s < je$ such that $r\not\equiv 2s\pmod{e}$.
\begin{enumerate}[label=(\roman*)]
\item\label{lem:psiaction1} If $a_s=r$, $a_{s+1}=r+1$, then $\psi_r v(a_1, \dots, a_{je}) = 0$.

\item\label{lem:psiaction2} If $s$ is maximal such that $a_s\leq{r-1}$, and $r, r+1 \notin \{a_1,\dots, a_{je}\}$, then $\psi_r v(a_1, \dots, a_{je}) = 0$.
\end{enumerate}
\end{lem}

\begin{proof}
We proceed by induction on $r-s$ on both of the statements.
\begin{enumerate}[label=(\roman*)]
\item For $r=s$, we observe that
\[
\psi_r v(a_1, \dots, a_{je})
=\psi_r \psid{a_{r+2}-1}{r+2} \clapdots \chaind{a_{ke}-1}{ke} z_\la
=\psid{a_{r+2}-1}{r+2} \clapdots \chaind{a_{ke}-1}{ke} \psi_r z_\la
=0.
\]
Now assuming that $r>s$, we have
\begin{align*}
\psi_r v(a_1, \dots, a_{je})
&= \psid{a_1-1}{1} \clapdots \chaind{a_{s-1}-1}{s-1}
(\psi_r\psi_{r-1}\psi_r(s-1, s, r-s-1))
\psid{r-2}{s}\chaind{r-1}{s+1}
\psid{a_{s+2}-1}{s+2} \clapdots \chaind{a_{je}-1}{je} z_\la\\
&= \psid{a_1-1}{1} \clapdots \chaind{a_{s-1}-1}{s-1} \psi_{r-1} \psi_r \psi_{r-1} \psid{r-2}{s}\chaind{r-1}{s+1} \psid{a_{s+2}-1}{s+2} \clapdots \chaind{a_{je}-1}{je} z_\la\\
&= \psid{a_1-1}{1} \clapdots \chaind{a_{s-1}-1}{s-1} \psi_{r-1} \psi_r \psi_{r-1}v(1, \dots, s-1, r-1, r, a_{s+2}, \dots, a_{je})
\end{align*}
which is $0$ by induction if $r\not\equiv 2s+1\pmod e$ or if $r = s+1$ as $\psi_{r-1}$ commutes through to the right.
If $r\equiv{2s+1}\pmod{e}$ and $r>s+1$, then we continue:
\begin{align*}
& \psid{a_1-1}{1} \clapdots \chaind{a_{s-1}-1}{s-1} \psi_{r-1} \psi_r (\psi_{r-1}\psi_{r-2}\psi_{r-1}(s-1,s,s-1)) \psid{r-3}{s} \chaind{r-2}{s+1} \psid{a_{s+2}-1}{s+2} \clapdots \chaind{a_{je}-1}{je} z_\la\\
= \; &\psid{a_1-1}{1} \clapdots \chaind{a_{s-1}-1}{s-1} \psi_{r-1}\psi_r (\psi_{r-2}\psi_{r-1}\psi_{r-2} + 1) \psid{r-3}{s}\chaind{r-2}{s+1} \psid{a_{s+2}-1}{s+2} \clapdots \chaind{a_{je}-1}{je} z_\la.
\end{align*}
The first term becomes
\begin{align*}
& \psid{a_1-1}{1} \clapdots \chaind{a_{s-1}-1}{s-1} \psi_{r-1}\psi_r \psi_{r-2}\psi_{r-1}\psi_{r-2} \psid{r-3}{s}\chaind{r-2}{s+1} \psid{a_{s+2}-1}{s+2} \clapdots \chaind{a_{je}-1}{je} z_\la\\
= \; &\psid{a_1-1}{1} \clapdots \chaind{a_{s-1}-1}{s-1} \psi_{r-1}\psi_r \psi_{r-2} \psi_{r-1}\psi_{r-2} v(1, \dots, s-1, r-2, r-1, a_{s+2}, \dots, a_{je})\\
= \; &0 \text{ by induction.}
\end{align*}

If $s = je - 1$, then the second term becomes
\[
\psid{a_1-1}{1} \clapdots \chaind{a_{je-2}-1}{je-2} \psi_{r-1}\psi_r \psid{r-3}{je-1}\chaind{r-2}{je} z_\la
= \psid{a_1-1}{1} \clapdots \chaind{a_{je-2}-1}{je-2} \psi_{r-1} \psid{r-3}{je-1} \chaind{r-2}{je}\psi_r z_\la = 0.
\]
If $s < je - 1$, then the second term becomes
\begin{align*}
&\psid{a_1-1}{1} \clapdots \chaind{a_{s-1}-1}{s-1} \psi_{r-1} \psid{r-3}{s}\psi_r \psid{r-2}{s+1} \psid{a_{s+2}-1}{s+2} \clapdots \chaind{a_{je}-1}{je} z_\la\\
= \; &\psid{a_1-1}{1} \clapdots \chaind{a_{s-1}-1}{s-1} \psi_{r-1} \psid{r-3}{s}\psi_r v(1,\dots,s,r-1,a_{s+2},\dots,a_{je})\\
= \; &0
\end{align*}
by the inductive hypothesis of (ii) as $a_{s+1}\leq r-1$, $a_{s+2}\geq r+2$, and $r\not\equiv 2(s+1) \pmod e$.

\item For $r=s+1$, we have
\[
\psi_{s+1}v(a_1,\dots,a_{je}) = \psi_{s+1}\psid{a_{s+1}-1}{s+1} \clapdots \chaind{a_{je}-1}{je} z_\la.
\]
We observe that the first $s+2$ residues in the residue sequence of  $s_{s+1}\s{a_{s+1}-1}{s+1} \dots \s{a_{je}-1}{je}\ttt_\la$ are $0, 1, \dots, s-1, 1, 0$.
There exists no $\tts \in \std\la$ with such a residue sequence, and hence $\psi_{s+1}v(a_1,\dots,a_{je})=0$.

Now assuming that $r>s+1$, we argue by induction on $\ell(w_\ttt)$.
For the base case, the minimal length is obtained when $s=je-1$ and $v_\ttt = v(1, 2, \dots, je-1, r-2)$.
Then
\begin{align*}
\psi_r v(1, 2, \dots, je-1, r+2) &= (\psi_r \psi_{r+1} \psi_r (-1, r, r+1)) \psid{r-1}{je} z_\la\\
&= \psi_{r+1} \psi_r \psi_{r+1} \psid{r-1}{je} z_\la\\
&= \psi_{r+1} \psi_r \psid{r-1}{je} \psi_{r+1} z_\la\\
&= 0 \text{ since $r+1 \not \equiv -1 \pmod e$ by our residue hypothesis.}
\end{align*}

Now for $\ell(w_\ttt)$ arbitrary, we have
\begin{align*}
& \psi_r v(a_1,\dots,a_{je})\\
= \; &\psid{a_1-1}{1} \clapdots \chaind{a_s-1}{s} \psid{a_{s+1}-1}{r+2} (\psi_r\psi_{r+1}\psi_r(s, r-s-1, r-s)) \psid{r-1}{s+1}\psid{a_{s+2}-1}{s+2} \clapdots \chaind{a_{je}-1}{je} z_\la,
\end{align*}
where $r-s\not\equiv s\pmod{e}$.
\begin{align*}
& \psid{a_1-1}{1} \clapdots \chaind{a_s-1}{s} \psid{a_{s+1}-1}{r+2} \psi_{r+1} \psi_r \psi_{r+1} \psid{r-1}{s+1} \psid{a_{s+2}-1}{s+2} \clapdots \chaind{a_{je}-1}{je} z_\la\\
= \; &\psid{a_1-1}{1} \clapdots \chaind{a_s-1}{s} \psid{a_{s+1}-1}{r+2} \psi_{r+1} \psi_r \psi_{r+1} v(1,\dots,s,r,a_{s+2},\dots,a_{je}) = 0
\end{align*}	
by the inductive hypothesis if $r+1 \not\equiv 2(s+1) \pmod e$.
If $r+1 \equiv 2s+2 \pmod e$, then
\begin{align*}
&\psid{a_1-1}{1} \clapdots \chaind{a_{s+1}-1}{s+1} \psid{a_{s+2}-1}{r+3} (\psi_{r+1}\psi_{r+2}\psi_{r+1}(s+1, s, s+1)) \psid{r}{s+2} \psid{a_{s+3}-1}{s+3} \clapdots \chaind{a_{je}-1}{je} z_\la\\
= \; &\psid{a_1-1}{1} \clapdots \chaind{a_{s+1}-1}{s+1} \psid{a_{s+2}-1}{r+3} (\psi_{r+2}\psi_{r+1}\psi_{r+2} + 1) \psid{r}{s+2} \psid{a_{s+3}-1}{s+3} \clapdots \chaind{a_{je}-1}{je} z_\la.
\end{align*}
If $s = je - 2$, then this becomes
\[
\psid{a_1-1}{1} \clapdots \chaind{a_{je-1}-1}{je-1} \psid{a_{je}-1}{r+3} (\psi_{r+2}\psi_{r+1}\psi_{r+2} + 1) \psid{r}{je} z_\la = 0.
\]
If $s < je - 2$, then the first term becomes
\begin{align*}
&\psid{a_1-1}{1} \clapdots \chaind{a_{s+1}-1}{s+1} \psid{a_{s+2}-1}{r+3} \psi_{r+2}\psi_{r+1}\psi_{r+2} \psid{r}{s+2} \psid{a_{s+3}-1}{s+3} \clapdots \chaind{a_{je}-1}{je} z_\la\\
= \; &\psid{a_1-1}{1} \clapdots \chaind{a_{s+1}-1}{s+1} \psid{a_{s+2}-1}{r+3} \psi_{r+2}\psi_{r+1}\psi_{r+2} v(1, \dots, s+1, r+1, a_{s+3}, \dots, a_{je})\\
= \; &0 \text{ by induction as $r+2 \not\equiv 2s+4 \pmod e$.}
\end{align*}
Now, the second term becomes
\begin{align*}
&\psid{a_1-1}{1} \clapdots \chaind{a_s-1}{s} \psid{a_{s+1}-1}{s+1} \psid{a_{s+2}-1}{r+3} \psid{r}{s+2} \psid{a_{s+3}-1}{s+3} \clapdots \chaind{a_{je}-1}{je} z_\la\\
= \; &\psid{a_1-1}{1} \clapdots \chaind{a_s-1}{s} \psid{a_{s+1}-1}{r+1} \psid{a_{s+2}-1}{r+3} \psi_r \psid{r-1}{s+1} \chaind{r}{s+2} \psid{a_{s+3}-1}{s+3} \clapdots \chaind{a_{je}-1}{je} z_\la\\
= \; &\psid{a_1-1}{1} \clapdots \chaind{a_s-1}{s} \psid{a_{s+1}-1}{r+1} \psid{a_{s+2}-1}{r+3} \psi_r v(1, \dots, s, r, r+1, a_{s+3}, \dots, a_{je})\\
= \; &0 \text{ by the inductive hypothesis on (i), as $r\not\equiv 2s+2 \pmod e$.}\qedhere
\end{align*}
\end{enumerate}

\end{proof}

\begin{cor}\label{cor:psiaction}
Let $1\leq r <n$, $1\leq s < je$ with $r\geq s+1$ and $r\equiv{2s}\pmod{e}$.
Then
\begin{enumerate}[label=(\roman*)]
\item\label{cor:psiaction1} $\psi_r v(1, \dots, s, r+2, a_{s+2}, \dots, a_{je}) = v(1, \dots, s, r, a_{s+2}, \dots, a_{je})$;

\item\label{cor:psiaction2} $\psi_r v(1, \dots, s-1, r, r+1, a_{s+2}, \dots, a_{je}) = v(1, \dots, s-1, r-1, r, a_{s+2}, \dots, a_{je})$.
\end{enumerate}
\end{cor}

\begin{proof}
\begin{enumerate}[label=(\roman*)]
\item We have
\begin{align*}
&\psi_r v(1,\dots,s,r+2,a_{s+2},\dots,a_{je})\\
= \; &\psi_r\psid{r+1}{s+1}\psid{a_{s+2}-1}{s+2} \clapdots \chaind{a_{je}-1}{je} z_\la\\
= \; &(\psi_r \psi_{r+1} \psi_r (s,s-1,s)) \psid{r-1}{s+1} \psid{a_{s+2}-1}{s+2} \clapdots \chaind{a_{je}-1}{je} z_\la\\
= \; &(\psi_{r+1} \psi_r \psi_{r+1} + 1) \psid{r-1}{s+1} \psid{a_{s+2}-1}{s+2} \clapdots \chaind{a_{je}-1}{je} z_\la\\
= \; &\psi_{r+1}\psi_r\psi_{r+1}v(1,\dots, s, r, a_{s+2}, \dots, a_{je}) + v(1,\dots,s,r,a_{s+2},\dots,a_{je}),
\end{align*} 
and $\psi_{r+1}v(1,\dots,s,r,a_{s+2},\dots,a_{je})= 0$ by \cref{lem:psiaction}\ref{lem:psiaction2}.
\item We have
\begin{align*}
&\psi_rv(1,\dots,s-1,r,r+1,a_{s+2},\dots,a_{je})\\
= \; &\psi_r\psid{r-1}{s}\chaind{r}{s+1}\psid{a_{s+2}-1}{s+2} \clapdots \chaind{a_{je}-1}{je} z_\la\\
= \; &(\psi_r\psi_{r-1}\psi_r(s-1,s,s-1)) \psid{r-2}{s} \chaind{r-1}{s+1} \psid{a_{s+2}-1}{s+2} \clapdots \chaind{a_{je}-1}{je} z_\la\\
= \; &(\psi_{r-1}\psi_r\psi_{r-1} + 1) \psid{r-2}{s} \chaind{r-1}{s+1} \psid{a_{s+2}-1}{s+2} \clapdots \chaind{a_{je}-1}{je} z_\la\\
= \; &\psi_{r-1}\psi_r\psi_{r-1}v(1, \dots, s-1, r-1, r, a_{s+2}, \dots, a_{je}) + v(1, \dots, s-1, r-1, r, a_{s+2}, \dots, a_{je}),
\end{align*}
and the first term is $0$ 
by \cref{lem:psiaction}\ref{lem:psiaction1} since $r-1\not\equiv 2s\pmod{e}$.\qedhere
\end{enumerate}
\end{proof}

\begin{lem}\label{lem:psipsi}
Let $1 \leq s \leq i \leq r < n$, $s+r-i \leq je$, $a_{s-1}<i$, and $i \not\equiv x \pmod e$ for any $x\in \{2s-2, 2s-1, \dots, 2s+r-i\}$.
Then
\begin{align*}
&\psid{r}{i} v(a_1, \dots, a_{s-1}, i+1, i+2, \dots, r+1, a_{s+r-i+1}, \dots, a_{je})\\
= \; &v(a_1, \dots, a_{s-1}, i, i+1, \dots, r, a_{s+r-i+1}, \dots, a_{je}).
\end{align*}
\end{lem}

\begin{proof}
Suppose that $a_{s+l-i-1}<l$ for all $l \in \{i, i+1, \dots, r\}$.
Then the result follows directly from the KLR relations since, for all $l$,
\begin{align*}
&\psi_l v(a_1, \dots, a_{s+l-i-1}, l+1, a_{s+l-i+1}, a_{s+l-i+2}, \dots, a_{je})\\
= \; & \psid{a_1-1}{1} \clapdots \chaind{a_{s+l-i-1}-1}{s+l-i-1} \psi_l^2(s+l-i-1, i-s) \psid{l-1}{s+l-i} \psid{a_{s+l-i+1}-1}{s+l-i+1} \chaind{a_{s+l-i+1}-1}{s+l-i+1} \clapdots \chaind{a_{je}-1}{je} z_\la\\
= \; & \psid{a_1-1}{1} \clapdots \chaind{a_{s+l-i-1}-1}{s+l-i-1} v(1, \dots, s+l-i-1, l, a_{s+l-i+1}, a_{s+l-i+2}, \dots, a_{je})
\end{align*}
if $i \not \equiv 2s+l-i-2, 2s+l-i-1, 2s+l-i\pmod e$.
\end{proof}

\begin{cor}\label{cor:psidownaction}
Suppose that $1 \leq s \leq i \leq r < n$, $s+r-i < je$, $a_{s-1}\leqslant i-2$, $i\equiv 2s\pmod{e}$ and $r-i+2<e$.
Then
\begin{align*}
&\psid{r}{i} v(a_1, \dots, a_{s-1}, i, i+1, i+2, \dots, r+1, a_{s+r-i+2}, \dots, a_{je})\\
&= v(a_1, \dots, a_{s-1}, i-1, i, i+1, \dots, r, a_{s+r-i+2}, \dots, a_{je}).
\end{align*}	
\end{cor}

\begin{proof}
	Since $i\equiv 2s\pmod{e}$, we apply \cref{cor:psiaction}\ref{cor:psiaction2} to give us
	\begin{align*}
	&\psid{r}{i+1} \psi_i v(a_1, \dots, a_{s-1}, i, i+1, i+2, \dots, r+1, a_{s+r-i+2}, \dots, a_{je})\\
	&= \psid{r}{i+1} v(a_1, \dots, a_{s-1}, i-1, i, i+2, \dots, r+1, a_{s+r-i+2}, \dots, a_{je}).
	\end{align*}
	We now obtain our desired result by applying \cref{lem:psipsi} since $i+1\not\equiv x\pmod{e}$ for all $x\in\{2s,2s+1,\dots,2s+r-i+1\}$ (note that $x$ runs over $r-i+2 < e$ terms).
\end{proof}

\begin{lem}\label{lem:strings}
Let $1\leqslant r \leqslant je$. If $r\not\equiv 1\pmod{e}$, then $y_r\psi_r\psi_{r+1}\dots\psi_{je} z_\la=0$.
\end{lem}

\begin{proof}
We proceed by induction on $\ell(w_{\ttt})$, where the minimal length is obtained when $r=je$.

For $r=je$, we have
\begin{align*}
&(y_{je}\psi_{je}(-1,0))z_\la
=\psi_{je}y_{je+1} z_\la
=0.
\end{align*}

Now assuming that $r<je$,
\begin{align*}
(y_r\psi_r(r-1,0))\psi_{r+1}\psi_{r+2}\dots\psi_{je} z_\la
&=\psi_ry_{r+1}\psi_{r+1}\psi_{r+2}\psi_{r+3}\dots\psi_{je} z_\la=0
\end{align*}
by induction if $r\not\equiv 0\pmod{e}$.
If $r\equiv 0\pmod{e}$, then this term becomes
\begin{align*}
\psi_r(y_{r+1}\psi_{r+1}(0,0))\psi_{r+2}\psi_{r+3}\dots\psi_{je} z_\la
&=
\psi_r(\psi_{r+1}y_{r+2}+1)\psi_{r+2}\psi_{r+3}\dots\psi_{je} z_\la.
\end{align*}
The second term becomes $\psi_{r+2}\psi_{r+3}\dots\psi_{je}(\psi_rz_\la)=0$, whilst the first term is $0$ by induction.
\end{proof}

\begin{lem}\label{newlem:up}
Let $1\leqslant s<i\leqslant r<n$ and $s<je$, and suppose that $r\equiv 2s\pmod{e}$ and $r-i+2<e$. Then
\[
\psiu{i}{r}
v(1,\dots,s,r+2,a_{s+2},\dots,a_{je})
=v(1,\dots,s,i,a_{s+2},\dots,a_{je}).
\]
\end{lem}

\begin{proof}
The proof is similar to the proof of \cref{cor:psidownaction}.
\end{proof}

\begin{lem}\label{lem:newyaction}
	
Let $0 \leq s \leq je-e$ and $v_\ttt = v(a_1, \dots, a_{je})$.
Then
\begin{enumerate}[label=(\roman*)]
\item\label{lem:yaction2} 
If $a_{s+e} = r$ for some $1\leq r \leq n$ such that $r \not \equiv 2s, 2s+1 \pmod{e}$ and $r-1,r+1, r+2, r+3, \dots, r+e-2 \not \in \{a_1, \dots, a_{je}\}$, then $y_{r-1} v_\ttt = 0$.

\item\label{lem:yaction1}
If $a_{s+e} = r$ for some $1\leq r \leq n$ such that $r \not \equiv 2s, 2s+1 \pmod{e}$ and $r+1, r+2, r+3, \dots, r+e-2 \not \in \{a_1, \dots, a_{je}\}$, then $y_r v_\ttt = 0$.

\item\label{lem:yaction3}
If for some $1\leq r < n$, we have $a_{s+i}=r-e+i$ for all $i \in \{1,\dots,e-1\}$, $a_{s+e}=r+1$, $r \equiv 2s \pmod e$ and $r+2,r+3,\dots, r+e \notin \{a_{s+e+1}, \dots, a_{je}\}$, then $\psi_r v_\ttt = 0$.
\end{enumerate}
\end{lem}

\begin{proof}
We proceed by simultaneous induction on $r-s$ on each of the three statements. Note that we apply \cref{cor:psiaction} without further reference.
\begin{enumerate}[label=(\roman*)]
\item Our base case is when $r = s+e+1$, so that $s \not \equiv 0,1\pmod{e}$ and $a_{s+e-1}=s+e-1$.
We prove this by induction on $\ell(w_\ttt)$.
For the base case, the minimal length is obtained when $s+e+1 = je$.
We thus have
\begin{align*}
y_{r-1} v_\ttt &= (y_{je-1} \psi_{je-1} (-2,0)) \psid{a_{je}-1}{je} z_\la\\
&= \psi_{je-1} \psid{a_{je}-1}{je+1} (y_{je}\psi_{je}(-1,0)) z_\la
=\psi_{je-1} \psid{a_{je}-1}{je+1} \psi_{je} y_{je+1} z_\la
= 0.
\end{align*}
Now suppose that $s+e+1<je$, and assume without loss of generality that $v_{\ttt}= v(1,\dots,s+e-1,s+e+1,s+2e,\dots,je+e-1)$.
Then we have
\begin{align*}
y_{r-1} v_{\ttt} 
&= (y_{s+e} \psi_{s+e}(s-1,0)) \psid{s+2e-1}{s+e+1} \chaind{s+2e}{s+e+2} \clapdots \chaind{je+e-2}{je} z_\la\\
&= \psi_{s+e} y_{s+e+1} \psid{s+2e-1}{s+e+1} \chaind{s+2e}{s+e+2} \clapdots \chaind{je+e-2}{je} z_\la\\
&= \psi_{s+e} \psid{s+2e-1}{s+e+2} y_{s+e+1} \psi_{s+e+1} \psid{s+2e}{s+e+2} \chaind{s+2e+1}{s+e+3} \clapdots \chaind{je+e-2}{je} z_\la\\
&= \psi_{s+e} \psid{s+2e-1}{s+e+2} y_{s+e+1} v(1,\dots,s+e,s+e+2,s+2e+1,\dots,je+e-3)\\
&=0 \text{ by induction if $s \not \equiv -1 \pmod{e}$.}
\intertext{ If $s \equiv -1 \pmod{e}$, then}
& \psi_{s+e}
\psid{s+2e-1}{s+e+2} (y_{s+e+1} \psi_{s+e+1}(-1,0)) \psid{s+2e}{s+e+2} \chaind{s+2e+1}{s+e+3} \clapdots \chaind{je+e-2}{je} z_\la\\
= \; & \psi_{s+e} \psid{s+2e-1}{s+e+1} y_{s+e+2} \psid{s+2e}{s+e+2} \chaind{s+2e+1}{s+e+3} \clapdots \chaind{je+e-2}{je} z_\la\\
= \; & \psi_{s+e} \psid{s+2e-1}{s+e+1} \psid{s+2e}{s+e+3} (y_{s+e+2} \psi_{s+e+2}(0,0)) \psid{s+2e+1}{s+e+3} \chaind{s+2e+2}{s+e+4} \clapdots \chaind{je+e-2}{je} z_\la\\
= \; & \psi_{s+e} \psid{s+2e-1}{s+e+1} \psid{s+2e}{s+e+3} (\psi_{s+e+2} y_{s+e+3} - 1) \psid{s+2e+1}{s+e+3} \chaind{s+2e+2}{s+e+4} \clapdots \chaind{je+e-2}{je} z_\la\\
= \; & \psi_{s+e} \psid{s+2e-1}{s+e+1} \psid{s+2e}{s+e+2} \psid{s+2e+1}{s+e+4} \chaind{s+2e+2}{s+e+5}
\clapdots \chaind{je+e-2}{je+1} y_{s+e+3} \psi_{s+e+3}\psi_{s+e+4}\dots\psi_{je}z_\la\\
\; & -\psi_{s+e} \psid{s+2e-1}{s+e+2} \psid{s+2e}{s+e+3} \psid{s+2e+1}{s+e+3} \chaind{s+2e+2}{s+e+4} \clapdots \chaind{je+e-2}{je} \psi_{s+e+1} z_\la\\
= \; & 0 \text{ by \cref{lem:strings}.}
\end{align*}
Next, we assume that $r>s+e+1$, and again argue by induction on $\ell(w_\ttt)$.
For the base case, the minimal length is obtained when $s+e=je$ and $v_\ttt = v(1, \dots, je-1, r)$.
Then
\[
y_{r-1} v_\ttt = (y_{r-1} \psi_{r-1} (-1, r-1)) \psid{r-2}{je} z_\la = 0.
\]
Now for $s+e < je$, we may assume by induction (on $\ell(w_\ttt)$) that $v_\ttt = v(1, \dots, s+e-1, r, r+e-1, r+e, \dots, r+je-s-2)$ and we have
\begin{align*}
y_{r-1} v_\ttt &= (y_{r-1} \psi_{r-1} (s-1, r-s-1)) \psid{r-2}{s+e} \psid{r+e-2}{s+e+1} \chaind{r+e-1}{s+e+2} \clapdots \chaind{r+je-s-3}{je} z_\la\\
&= \psid{r-1}{s+e} \psid{r+e-2}{r+1} y_r \psid{r}{s+e+1} \psid{r+e-1}{s+e+2} \chaind{r+e}{s+e+3} \clapdots \chaind{r+je-s-3}{je} z_\la\\
&= \psid{r-1}{s+e} \psid{r+e-2}{r+1} y_r v(1, \dots, s+e, r+1, r+e, \dots, r+je-s-2),\\
&=0 \text{ by induction if $r\not\equiv 2s+2 \pmod{e}$.}
\end{align*}
If $r \equiv 2s+2 \pmod{e}$, then we have
\begin{align*}
&\psid{r-1}{s+e} \psid{r+e-2}{r+1} (y_r\psi_r(s,s+1)) \psid{r-1}{s+e+1} \psid{r+e-1}{s+e+2} \chaind{r+e}{s+e+3} \clapdots \chaind{r+je-s-3}{je} z_\la\\
= \; &\psid{r-1}{s+e} \psid{r+e-2}{s+e+1} y_{r+1} \psid{r+e-1}{s+e+2} \chaind{r+e}{s+e+3} \clapdots \chaind{r+je-s-3}{je} z_\la\\
= \; &0 \text{ if $s+e=je-1$.}
\end{align*}
Now suppose that $s+e < je - 1$.
Then the above term becomes
\begin{align*}
&\psid{r-1}{s+e}\psid{r+e-2}{s+e+1} \psid{r+e-1}{r+2} (y_{r+1}\psi_{r+1}(s+1,s+1)) \psid{r}{s+e+2} \psid{r+e}{s+e+3} \chaind{r+e+1}{s+e+4} \clapdots \chaind{r+je-s-3}{je} z_\la\\
= \; &\psid{r-1}{s+e}\psid{r+e-2}{s+e+1} \psid{r+e-1}{r+2} (\psi_{r+1} y_{r+2} - 1) \psid{r}{s+e+2} \psid{r+e}{s+e+3} \chaind{r+e+1}{s+e+4} \clapdots \chaind{r+je-s-3}{je} z_\la\\
= \; &0 \text{ if $s+e=je-2$}.
\end{align*}
If $s+e < je - 2$, the first term becomes
\begin{align*}
&\psid{r-1}{s+e} \psid{r+e-2}{s+e+1} \chaind{r+e-1}{s+e+2} \psid{r+e}{r+3} y_{r+2} \psid{r+2}{s+e+3} \psid{r+e+1}{s+e+4} \chaind{r+e+2}{s+e+5} \clapdots \chaind{r+je-s-3}{je} z_\la\\
= \; &\psid{r-1}{s+e} \psid{r+e-2}{s+e+1} \chaind{r+e-1}{s+e+2}\psid{r+e}{r+3} y_{r+2} v(1, \dots, s+e+2, r+3, r+e+2, \dots, r+je-s-2),\\
= \; &0 \text{ by induction since $r \not\equiv 2s+3, 2s+4 \pmod{e}$.}
\end{align*}
Then the second term above becomes
\begin{align*}
& -\psid{r-1}{s+e} \psid{r+e-2}{s+e+1} \psid{r}{s+e+2} \psid{r+e-1}{r+3} \psid{r+e}{r+4} 
\cancel{\psi_{r+2}\psi_{r+3}\psi_{r+2}}
\psid{r+1}{s+e+3} \psid{r+e+1}{s+e+4} \chaind{r+e+2}{s+e+5} \clapdots \chaind{r+je-s-3}{je} z_\la,
\end{align*}
which is zero if $s+e=je-3$.
If $s+e<je-3$, then we continue
\begin{align*}
& -\psid{r-1}{s+e} 
\psid{r+e-2}{s+e+1} \psid{r}{s+e+2} \chaind{r+1}{s+e+3} \psid{r+e-1}{r+3} \chaind{r+e}{r+4} \psid{r+e+1}{s+e+4} \chaind{r+e+2}{s+e+5} \clapdots \chaind{r+je-s-3}{je} z_\la\\
= & -\psid{r-1}{s+e} \left( \psid{r+e-2}{r+2} \chaind{r+e-1}{r+3} \chaind{r+e}{r+4}\right) \psi_{r+1} \cancel{(\psi_r\psi_{r-1}\psi_r)} \psid{r-2}{s+e+1} \psid{r-1}{s+e+2} \psid{r+1}{s+e+3} \psid{r+e+1}{s+e+4} \chaind{r+e+2}{s+e+5} \clapdots \chaind{r+je-s-3}{je} z_\la\\
= & -\psid{r-1}{s+e} \left( \psid{r+e-2}{r+2} \chaind{r+e-1}{r+3} \chaind{r+e}{r+4}\right) \psi_{r+1} v(1,\dots,s+e,r-1,r,r+2,r+e+2,\dots,r+je-s-2)\\
= & \; 0 \text{ by the inductive hypothesis of part \ref{lem:yaction3} if $e=3$.}
\end{align*}
We apply \cref{newlem:up} without further reference.
If $e>3$, we have
\begin{align*}
& -\psid{r-1}{s+e} 
\left( \psid{r+e-2}{r+2} \chaind{r+e-1}{r+3} \chaind{r+e}{r+4}\right)
\psid{r-2}{s+e+1} \chaind{r-1}{s+e+2} 
\cancel{(\psi_{r+1}^2(s+2,s))}
\psid{r}{s+e+3} \psid{r+e+1}{s+e+4} \chaind{r+e+2}{s+e+5} \clapdots \chaind{r+je-s-3}{je} z_\la\\
= & -\psid{r-1}{s+e} \left( \psid{r+e-2}{r+3} \chaind{r+e-1}{r+4} \chaind{r+e}{r+5} \chaind{r+e+1}{r+6}\right)\\
& \ \cdot \psiu{r+2}{r+4} v(1,\dots,s+e,r-1,r,r+1,r+6,r+e+3,\dots,r+je-s-2)\\
= & -\psid{r-1}{s+e} \left( \psid{r+e-2}{r+3} \chaind{r+e-1}{r+4} \chaind{r+e}{r+5} \chaind{r+e+1}{r+6}\right)\\
& \ \cdot \psi_{r+2}v(1,\dots,s+e,r-1,r,r+1,r+3,r+e+3,\dots,r+je-s-2)\\
= & \; 0 \text{ by the inductive hypothesis of part \ref{lem:yaction3} if $e=4$.}
\end{align*}
If $e>4$, then we have
\begin{align*}
& -\psid{r-1}{s+e} \left( \psid{r+e-2}{r+3} \chaind{r+e-1}{r+4} \chaind{r+e}{r+5} \chaind{r+e+1}{r+6} \right) \psid{r-2}{s+e+1} \chaind{r-1}{s+e+2} \chaind{r}{s+e+3} \cancel{(\psi_{r+2}^2(s+3,s))} \psid{r+1}{s+e+4} \!\!\psid{r+e+2}{s+e+5} \chaind{r+e+3}{s+e+6} \clapdots \!\! \chaind{r+je-s-3}{je} \!\!\! z_\la\\
= &  -\psid{r-1}{s+e} \left( \psid{r+e-2}{r+4} \chaind{r+e-1}{r+5} \clapdots 
\chaind{r+e+2}{r+8} \right)\\
& \ \cdot \psiu{r+3}{r+6} v(1, \dots, s{+}e, r{-}1, r, r{+}1, r{+}2, r{+}8, r{+}e{+}4, \dots, r{+}je{-}s{-}2)\\
= & -\psid{r-1}{s+e} \left( \psid{r+e-2}{r+4} \chaind{r+e-1}{r+5} \clapdots 
\chaind{r+e+2}{r+8} \right)\\
& \ \cdot \psi_{r+3} v(1, \dots, s{+}e, r{-}1, r, r{+}1, r{+}2, r{+}4, r{+}e{+}4, \dots, r{+}je{-}s{-}2)\\
= & \; 0 \text{ by the inductive hypothesis of part \ref{lem:yaction3} if $e=5$.}
\end{align*}
Continuing in this fashion, we eventually obtain
\begin{align*}
& -\psid{r-1}{s+e} \psiu{r+e-2}{r+2e-4} v(1, \dots, s{+}e, r{-}1, r, \dots, r{+}e{-}3, r{+}2e{-}2, r{+}2e{-}1, \dots, je{+}r{-}s{-}2)\\
= & -\psid{r-1}{s+e} \psi_{r+e-2} v(1, \dots, s{+}e, r{-}1, r, \dots, r{+}e{-}3, r{+}e{-}1, r{+}2e{-}1, \dots, je{+}r{-}s{-}2)\\
= & \; 0 \text{ by the inductive hypothesis of part \ref{lem:yaction3}.}
\end{align*}

\item
If $r=s+e$, then the term $\psid{r-1}{s+e}$ is trivial so that
\[
y_r v_\ttt = y_{s+e} \psid{a_{s+e+1}-1}{s+e+1} \clapdots \chaind{a_{je}-1}{je} z_\la
= \psid{a_{s+e+1}-1}{s+e+1} \clapdots \chaind{a_{je}-1}{je} y_{s+e} z_\la = 0.
\]
We now suppose that $r>s+e$, and assume without loss of generality that $v_{\ttt} = v(1,\dots,s+e-1,r,r+e-1,r+e,\dots,je-s+r-2)$.
Then
\begin{align*}
& y_r v_\ttt\\
= \; & (y_r\psi_{r-1}(s-1, r-s-1)) \psid{r-2}{s+e} \psid{r+e-2}{s+e+1} \chaind{r+e-1}{s+e+2} \clapdots \chaind{je-s+r-3}{je} z_\la,\\
= \; & \psi_{r-1}y_{r-1} \psid{r-2}{s+e} \psid{r+e-2}{s+e+1} \chaind{r+e-1}{s+e+2} \clapdots \chaind{je-s+r-3}{je} z_\la\\
= \; &
\begin{cases}
0 & \text{if $r=s+e+1$}\\
\psi_{r-1} y_{r-1} v(1, \dots, s+e-1, r-1, r+e-1, \dots, je-s+r-2) & \text{if $r\geq s+e+2$}
\end{cases}\\
= \; &0 \text{ by induction if $r\not\equiv 2s+2\pmod{e}$.}
\end{align*}
If $r\geq s+e+2$ and $r\equiv 2s+2\pmod{e}$, then
\begin{align*}
& \psi_{r-1} (y_{r-1}\psi_{r-2}(s-1,s)) \psid{r-3}{s+e} \psid{r+e-2}{s+e+1} \clapdots \chaind{je-s+r-3}{je} z_\la\\
= \; & \psi_{r-1} \psi_{r-2} (y_{r-2} \psi_{r-3}(s-1,s-1)) \psid{r-4}{s+e} \psid{r+e-2}{s+e+1} \clapdots \chaind{je-s+r-3}{je} z_\la \\
= \; &
\begin{cases}
0 & \text{if $r=s+e+2$}\\
\psi_{r-1} \psi_{r-2} (\psi_{r-3}y_{r-3}+1) \psid{r-4}{s+e} \psid{r+e-2}{s+e+1} \clapdots \chaind{je-s+r-3}{je} z_\la & \text{if $r\geq s+e+3$.}
\end{cases}
\end{align*}
Assuming $r\geq s+e+3$, the first term of this is
\[
\psi_{r-1} \psi_{r-2} \psi_{r-3} y_{r-3} v(1, \dots, s+e-1, r-3, r+e-1, \dots, je-s+r-2) = 0 \text{ by induction.}
\]
If $je=s+e$, then the second term becomes
\[
\psi_{r-1} \psi_{r-2}\psid{r-4}{je} z_\la
= \psi_{r-1} \psid{r-4}{je}\psi_{r-2} z_\la = 0.
\]
Now suppose that $je > s+e$.
Then the second term becomes
\begin{align*}
&\psi_{r-1} \psid{r-4}{s+e} \psid{r+e-2}{r} \cancel{\psi_{r-2} \psi_{r-1} \psi_{r-2}} \psid{r-3}{s+e+1} \psid{r+e-1}{s+e+2} \chaind{r+e}{s+e+3} \clapdots \chaind{je-s+r-3}{je} z_\la\\
= \; & \psi_{r-1} \psid{r+e-2}{r+1} \psid{r-4}{s+e}\chaind{r-3}{s+e+1} \psid{r+e-1}{r+2} \cancel{\psi_r\psi_{r+1}\psi_r} \psid{r-1}{s+e+2} \psid{r+e}{s+e+3} \chaind{r+e+1}{s+e+4} \clapdots \chaind{je-s+r-3}{je} z_\la\\
= \; & \psid{r+e-2}{r+1} \psid{r+e-1}{r+2} \psi_{r-1} v(1,\dots,s+e-1,r-3,r-2,r,r+e+1,\dots,je-s+r-2)\\
= \; & 0 \text{ by the inductive hypothesis of part \ref{lem:yaction3} if $e=3$.}
\end{align*}
From here, the proof concludes in a similar manner to the proof of part \ref{lem:yaction2}.

\item 
Our base case is when $r=s+e$, so that $s \equiv 0 \pmod{e}$.
We prove this by induction on $\ell(w_{\ttt})$, and assume without loss of generality that $v_{\ttt}=v(1,\dots,s+e-1,s+e+1,s+2e+1,s+2e+2,\dots,je+e)$.
For the base case, the minimal length is obtained when $s+e=je$.
\[
\psi_rv_{\ttt} =(\psi_{s+e}^2(-1,0)) z_\la
=(y_{s+e}-y_{s+e+1}) z_\la
=0.
\]
Now suppose that $s+e<je$.
Then we have
\begin{align*}
\psi_r v_{\ttt} 
& = (\psi_{s+e}^2 (-1,0)) \psid{s+2e}{s+e+1} \chaind{s+2e+1}{s+e+2} \clapdots \chaind{je+e-1}{je} z_\la\\
& = (y_{s+e}-y_{s+e+1}) \psid{s+2e}{s+e+1} \chaind{s+2e+1}{s+e+2} \clapdots \chaind{je+e-1}{je} z_\la\\
& = 0 - \psid{s+2e}{s+e+2} (y_{s+e+1} \psi_{s+e+1}(0,0)) \psid{s+2e+1}{s+e+2} \chaind{s+2e+2}{s+e+3} \clapdots \chaind{je+e-1}{je} z_\la\\
& = - \psid{s+2e}{s+e+2} (\psi_{s+e+1} y_{s+e+2} - 1) \psid{s+2e+1}{s+e+2} \chaind{s+2e+2}{s+e+3} \clapdots \chaind{je+e-1}{je}z_\la.
\end{align*}
The first term becomes
\[
- \psid{s+2e}{s+e+1} \psid{s+2e+1}{s+e+3} \chaind{s+2e+2}{s+e+4} \clapdots \chaind{je+e-1}{je+1} y_{s+e+2} \psi_{s+e+2}\psi_{s+e+3} \dots \psi_{je} z_\la = 0
\text{ by \cref{lem:strings}.}
\]
The second term becomes
\begin{align*}
& \psid{s+2e}{s+e+3} \chaind{s+2e+1}{s+e+4} \cancel{\psi_{s+e+2} \psi_{s+e+3} \psi_{s+e+2}} \psid{s+2e+2}{s+e+3} \chaind{s+2e+3}{s+e+4} \clapdots \chaind{je+e-1}{je} z_\la \\
= \; & \psid{s+2e}{s+e+3} \psid{s+2e+1}{s+e+5} \chaind{s+2e+2}{s+e+6} \cancel{\psi_{s+e+4} \psi_{s+e+5} \psi_{s+e+4}} \psi_{s+e+3} \psid{s+2e+3}{s+e+4} \chaind{s+2e+4}{s+e+5} \clapdots \chaind{je+e-1}{je} z_\la \\
= \; & \psid{s+2e}{s+e+4} \chaind{s+2e+1}{s+e+5} \chaind{s+2e+2}{s+e+6} \psi_{s+e+3} v(1,\dots,s+e+2,s+e+4,s+2e+4,\dots,je+e)\\
= \; & 0 \text{ by induction if $e=3$.}
\end{align*}
If $e>3$, then we have
\begin{align*}
& \psid{s+2e}{s+e+4} \chaind{s+2e+1}{s+e+5} \chaind{s+2e+2}{s+e+6} \cancel{(\psi_{s+e+3}^2(2,0))} \psid{s+2e+3}{s+e+4} \chaind{s+2e+4}{s+e+5} \clapdots \chaind{je+e-1}{je} z_\la\\
= \; &\psid{s+2e}{s+e+4} \chaind{s+2e+1}{s+e+5} \psid{s+2e+2}{s+e+7} \chaind{s+2e+3}{s+e+8} \cancel{\psi_{s+e+6}\psi_{s+e+7}\psi_{s+e+6}} \psid{s+e+5}{s+e+4} \psid{s+2e+4}{s+e+5} \chaind{s+2e+5}{s+e+6} \clapdots \chaind{je+e-1}{je} z_\la\\
= \; &\psid{s+2e}{s+e+4} \psid{s+2e+1}{s+e+6} \chaind{s+2e+2}{s+e+7} \chaind{s+2e+3}{s+e+8} \cancel{(\psi_{s+e+5}^2(3,1))}\psi_{s+e+4} \psid{s+2e+4}{s+e+5} \chaind{s+2e+5}{s+e+6} \clapdots \chaind{je+e-1}{je} z_\la\\
= \; &\psid{s+2e}{s+e+5} \chaind{s+2e+1}{s+e+6} \chaind{s+2e+2}{s+e+7} \chaind{s+2e+3}{s+e+8} \psi_{s+e+4} v(1,\dots,s+e+3,s+e+5,s+2e+5,\dots,je+e)\\
= \; &0 \text{ by induction if $e=4$.}
\end{align*}
If $e>4$, then we have
\begin{align*}
& \psid{s+2e}{s+e+5} \chaind{s+2e+1}{s+e+6} \chaind{s+2e+2}{s+e+7} \chaind{s+2e+3}{s+e+8} \cancel{\psi_{s+e+4}^2(3,0))} \psid{s+2e+4}{s+e+5} \chaind{s+2e+5}{s+e+6} \clapdots \chaind{je+e-1}{je} z_\la \\
= \; & \psid{s+2e}{s+e+5} \chaind{s+2e+1}{s+e+6} \chaind{s+2e+2}{s+e+7} \psid{s+2e+3}{s+e+9} \chaind{s+2e+4}{s+e+10} \cancel{\psi_{s+e+8}\psi_{s+e+9}\psi_{s+e+8}} \psid{s+e+7}{s+e+5} \psid{s+2e+5}{s+e+6} \chaind{s+2e+6}{s+e+7} \clapdots \chaind{je+e-1}{je} z_\la\\
= \; & \psid{s+2e}{s+e+5} \chaind{s+2e+1}{s+e+6} \psid{s+2e+2}{s+e+8} \chaind{s+2e+3}{s+e+9} \chaind{s+2e+4}{s+e+10} \cancel{(\psi_{s+e+7}^2(4,2))} \psid{s+e+6}{s+e+5} \psid{s+2e+5}{s+e+6} \chaind{s+2e+6}{s+e+7} \clapdots \chaind{je+e-1}{je} z_\la \\
= \; & \psid{s+2e}{s+e+5} \psid{s+2e+1}{s+e+7} \chaind{s+2e+2}{s+e+8} \chaind{s+2e+3}{s+e+9} \chaind{s+2e+4}{s+e+10} \cancel{(\psi_{s+e+6}^2(4,1))} \psi_{s+e+5} \psid{s+2e+5}{s+e+6} \chaind{s+2e+6}{s+e+7} \clapdots \chaind{je+e-1}{je} z_\la\\
= \; & \psid{s+2e}{s+e+6} \chaind{s+2e+1}{s+e+7} \clapdots \chaind{s+2e+4}{s+e+10} \psi_{s+e+5} v(1,\dots,s+e+4,s+e+6,s+2e+6,\dots,je+e)\\
= \; & 0 \text{ by induction if $e=5$.}
\end{align*}
If $e>5$, we continue in this way until we obtain
\[
\psid{s+2e}{s+2e+1} \chaind{s+2e+1}{s+2e+2} \clapdots \chaind{s+3e-1}{s+3e} \psi_{s+2e}
v(1,\dots,s+2e-1,s+2e+1,s+3e+1,\dots,je+e) = 0
\]
by induction.

We now suppose that $r>s+e$, and again we will use induction on $\ell(w_\ttt)$, so we may assume without loss of generality that $v_{\ttt}=v(1,\dots,s,r-e+1,r-e+2,\dots,r-1,r+1,r+e+1,r+e+2,\dots,je+r-s)$.

For the base case, we assume that $s+e=je$ and hence $r \equiv s \equiv 0\pmod{e}$. Applying \Cref{cor:psidownaction}, we have
\begin{align*}
\psi_r v_\ttt 
&= \psi_r v(1,\dots,je-e,r-e+1,r-e+2,\dots,r-1,r+1)\\
&= \psi_r \psid{r-e}{je-e+1} \chaind{r-e+1}{je-e+2} \clapdots
\chaind{r-2}{je-1} \psid{r}{je} z_\la\\
&= \psid{r-e}{je-e+1} \chaind{r-e+1}{je-e+2} \clapdots
\chaind{r-2}{je-1} 
(\psi_r^2(-1,0))
\psid{r-1}{je} z_\la\\
&= \psid{r-e}{je-e+1} \chaind{r-e+1}{je-e+2} \clapdots
\chaind{r-2}{je-1} 
(y_r - y_{r+1})
\psid{r-1}{je} z_\la\\
&= \psid{r-e}{je-e+1} \chaind{r-e+1}{je-e+2} \clapdots
\chaind{r-2}{je-1} 
y_r \psid{r-1}{je} z_\la - 0\\
&= \psid{r-e}{je-e+1} \chaind{r-e+1}{je-e+2} \clapdots
\chaind{r-2}{je-1}
(y_r \psi_{r-1}(-1,-1)) 
\psid{r-2}{je} z_\la \\
&= \psid{r-e}{je-e+1} \chaind{r-e+1}{je-e+2} \clapdots
\chaind{r-2}{je-1}
(\psi_{r-1}y_{r-1} + 1) 
\psid{r-2}{je} z_\la \\
&= 0 + \psid{r-e}{je-e+1} \chaind{r-e+1}{je-e+2} \clapdots
\chaind{r-2}{je-1}
\psid{r-2}{je} z_\la
\text{ by the inductive hypothesis of part \ref{lem:yaction1}} \\
&= \psid{r-e}{je-e+1} \chaind{r-e+1}{je-e+2} \clapdots
\chaind{r-3}{je-2}
\cancel{\psi_{r-2}\psiu{r-3}{r-2}}
\psid{r-4}{je-1}
\chaind{r-3}{je} z_\la\\
&= \psid{r-e}{je-e+1} \chaind{r-e+1}{je-e+2} \clapdots
\chaind{r-4}{je-3}
\cancel{\psid{r-3}{r-4}\psiu{r-5}{r-3}}
\psid{r-6}{je-2}
\chaind{r-5}{je-1}
\psid{r-4}{je} z_\la\\
&= \psid{r-e}{je-e+1} \chaind{r-e+1}{je-e+2} \clapdots
\chaind{r-5}{je-4}
\cancel{\psid{r-4}{r-6}\psiu{r-7}{r-4}}
\psid{r-8}{je-3}
\chaind{r-7}{je-2}
\chaind{r-6}{je-1}
\psid{r-5}{je} z_\la\\
 & \; \; \vdots \\
&= \psid{r-e}{je-e+1}
 \cancel{\psid{r-e+1}{r-2e+4} \psiu{r-2e+3}{r-e+1}}
\psid{r-2e+2}{je-e+2}
\chaind{r-2e+3}{je-e+3}
\clapdots
\chaind{r-e-1}{je-1}
\psid{r-e}{je} z_\la \\
&= \psi_{r-e} \cancel{\psid{r-e-1}{r-2e+2} \psiu{r-2e+1}{r-e-1}}
\psid{r-2e}{je-e+1}
\chaind{r-2e+1}{je-e+2}
\clapdots
\chaind{r-e-2}{je-1}
\psid{r-e}{je} z_\la \\
&=   
\psid{r-2e}{je-e+1}
\chaind{r-2e+1}{je-e+2}
\clapdots
\chaind{r-e-2}{je-1}
(\psi_{r-e}^2(-1,0))
\psid{r-e-1}{je} z_\la.
\end{align*}
We repeat the above process $s-j-1$ more times, until we reach
\begin{align*}
(\psi_{je}^2(-1,0)) z_\la
= (y_{je} - y_{je+1}) z_\la
= 0.
\end{align*}
We now suppose that $s+e<je$. We thus have
\begin{align*}
\psi_r v_\ttt & = \psi_r \psid{r-e}{s+1} \chaind{r-e+1}{s+2} \clapdots \chaind{r-2}{s+e-1} \psid{r}{s+e} \psid{r+e}{s+e+1} \chaind{r+e+1}{s+e+2} \clapdots \chaind{je+r-s-1}{je} z_\la \\
& = \psid{r-e}{s+1} \chaind{r-e+1}{s+2} \clapdots \chaind{r-2}{s+e-1} (\psi_r^2(s-1,s)) \psid{r-1}{s+e} \psid{r+e}{s+e+1} \chaind{r+e+1}{s+e+2} \clapdots \chaind{je+r-s-1}{je} z_\la\\
& = \psid{r-e}{s+1} \chaind{r-e+1}{s+2} \clapdots \chaind{r-2}{s+e-1} (y_r - y_{r+1}) \psid{r-1}{s+e} \psid{r+e}{s+e+1} \chaind{r+e+1}{s+e+2} \clapdots \chaind{je+r-s-1}{je} z_\la. \hypertarget{2terms}{\tag{\textdagger}}
\end{align*}
Applying the inductive hypothesis of part \ref{lem:yaction1}, the first term of \hyperlink{2terms}{(\textdagger)} becomes
\begin{align*}
& \psid{r-e}{s+1} \chaind{r-e+1}{s+2} \clapdots \chaind{r-2}{s+e-1} (y_r\psi_{r-1}(s-1,s-1)) \psid{r-2}{s+e} \psid{r+e}{s+e+1} \chaind{r+e+1}{s+e+2} \clapdots \chaind{je+r-s-1}{je} z_\la\\
= & \; \psid{r-e}{s+1} \chaind{r-e+1}{s+2} \clapdots \chaind{r-2}{s+e-1} (\psi_{r-1} y_{r-1} + 1) \psid{r-2}{s+e} \psid{r+e}{s+e+1} \chaind{r+e+1}{s+e+2} \clapdots \chaind{je+r-s-1}{je} z_\la\\
= & \; 0 + \psid{r-e}{s+1} \chaind{r-e+1}{s+2} \clapdots \chaind{r-3}{s+e-2} \cancel{\psi_{r-2} \psi_{r-3} \psi_{r-2}} \psid{r-4}{s+e-1} \psid{r-3}{s+e} \psid{r+e}{s+e+1} \chaind{r+e+1}{s+e+2} \clapdots \chaind{je+r-s-1}{je} z_\la\\
= & \; \psid{r-e}{s+1} \chaind{r-e+1}{s+2} \clapdots \chaind{r-4}{s+e-3} \psi_{r-3}
\cancel{\psi_{r-4}\psi_{r-5}\psi_{r-4}} \psid{r-6}{s+e-2} \chaind{r-5}{s+e-1} \psid{r-3}{s+e} \psid{r+e}{s+e+1} \chaind{r+e+1}{s+e+2} \clapdots \chaind{je+r-s-1}{je} z_\la \\
= & \; \psid{r-e}{r-e-2} \chaind{r-e+1}{r-e-1} \clapdots \chaind{r-4}{r-6} \psi_{r-3} \psid{r-e-3}{s+1} \chaind{r-e-2}{s+2} \clapdots \chaind{r-5}{s+e-1} \psid{r-3}{s+e} \psid{r+e}{s+e+1} \chaind{r+e+1}{s+e+2} \clapdots \chaind{je+r-s-1}{je} z_\la\\
= & \; \psid{r-e}{r-e-2} \chaind{r-e+1}{r-e-1} \clapdots \chaind{r-4}{r-6} \psi_{r-3} v(1, \dots, s, r{-}e{-}2, r{-}e{-}1, \dots, r{-}4, r{-}2, r{+}e{+}1, \dots, je{+}r{-}s)\\
= & \; 0 \text{ by induction on $r-s$ if $e=3$.}
\end{align*}
If $e>3$, then by applying \cref{cor:psidownaction}, we have	
\begin{align*}
& \left(\psid{r-e}{r-e-2} \chaind{r-e+1}{r-e-1} \clapdots  \chaind{r-4}{r-6}\right) \psid{r-e-3}{s+1} \chaind{r-e-2}{s+2} \clapdots \chaind{r-5}{s+e-1} \cancel{(\psi_{r-3}^2(s-1,s-3))} \psid{r-4}{s+e} \psid{r+e}{s+e+1} \chaind{r+e+1}{s+e+2} \clapdots \chaind{je+r-s-1}{je} z_\la\\
= & \left(\psid{r-e}{r-e-2} \chaind{r-e+1}{r-e-1} \clapdots \chaind{r-5}{r-7}\right) \psid{r-e-3}{s+1} \chaind{r-e-2}{s+2} \clapdots \chaind{r-8}{s+e-4}\psi_{r-4}\\
& \ \cdot \psid{r-5}{r-6} v(1,\dots,s+e-4,r-6,r-5,r-4,r-3,r+e+1,\dots,je+r-s)\\
= & \left(\psid{r-e}{r-e-2} \chaind{r-e+1}{r-e-1} \clapdots \chaind{r-5}{r-7}\right) \psid{r-e-3}{s+1} \chaind{r-e-2}{s+2} \clapdots \chaind{r-8}{s+e-4}\\
& \ \cdot \psi_{r-4} v(1,\dots,s+e-4,r-7,r-6,r-5,r-3,r+e+1,\dots,je+r-s) \\
= & \; 0 \text{ by induction if $e=4$.}
\end{align*}
As in parts \ref{lem:yaction2} and \ref{lem:yaction1}, we continue in this fashion for $e>4$, until we eventually obtain
\[
\psi_{r-e} v(1, \dots, s, r{-}2e{+}1, r{-}2e{+}3, \dots, r{-}e{-}1, r{-}e{+}1, r{+}e{+}1, \dots, je{+}r{-}s) = 0
\]
by induction.
Applying the inductive hypothesis of part \ref{lem:yaction2}, the second term of \hyperlink{2terms}{(\textdagger)} becomes
\begin{align*}
& - \psid{r-e}{s+1} \chaind{r-e+1}{s+2} \clapdots \chaind{r-1}{s+e} \psid{r+e}{r+2} (y_{r+1}\psi_{r+1}(s,s)) \psid{r}{s+e+1} \psid{r+e+1}{s+e+2} \chaind{r+e+2}{s+e+3} \clapdots \chaind{je+r-s-1}{je} z_\la\\
= & - \psid{r-e}{s+1} \chaind{r-e+1}{s+2} \clapdots \chaind{r-1}{s+e} \psid{r+e}{r+2} (\psi_{r+1}y_{r+2} - 1) \psid{r}{s+e+1} \psid{r+e+1}{s+e+2} \chaind{r+e+2}{s+e+3} \clapdots \chaind{je+r-s-1}{je} z_\la\\
= & - \psid{r-e}{s+1} \chaind{r-e+1}{s+2} \clapdots \chaind{r-1}{s+e} \psid{r+e}{s+e+1} \psid{r+e+1}{r+3} y_{r+2} v(1, \dots, s{+}e{+}1, r{+}3, r{+}e{+}3, \dots, je{+}r{-}s)\\
& + \psid{r+e}{r+3} \psid{r-e}{s+1} \chaind{r-e+1}{s+2} \clapdots \chaind{r}{s+e+1} \psid{r+e+1}{r+4} \cancel{\psi_{r+2}\psi_{r+3}\psi_{r+2}} \psid{r+1}{s+e+2} \psid{r+e+2}{s+e+3} \chaind{r+e+3}{s+e+4} \clapdots \chaind{je+r-s-1}{je} z_\la\\
= & \; 0 + \psid{r+e}{r+3} \chaind{r+e+1}{r+4} \psid{r-e}{s+1} \chaind{r-e+1}{s+2} \clapdots \chaind{r+1}{s+e+2} \psid{r+e+2}{s+e+3} \chaind{r+e+3}{s+e+4} \clapdots \chaind{je+r-s-1}{je} z_\la\\
= & \; \psid{r+e}{r+3}\psid{r+e+1}{r+5} \psid{r-e}{s+1} \chaind{r-e+1}{s+2} \clapdots \chaind{r+1}{s+e+2} \psid{r+e+2}{r+6} \cancel{\psi_{r+4} \psi_{r+5} \psi_{r+4}} \psid{r+3}{s+e+3} \psid{r+e+3}{s+e+4} \chaind{r+e+4}{s+e+5} \clapdots \chaind{je+r-s-1}{je} z_\la\\
= & \; \psid{r+e}{r+4} \chaind{r+e+1}{r+5} \chaind{r+e+2}{r+6} \psid{r-e}{s+1}\chaind{r-e+1}{s+2}\chaind{r-e+2}{s+3}\\
& \cdot \psi_{r+3} v(1, \dots, s{+}3, r{-}e{+}4, r{-}e{+}5, \dots, r{+}2, r{+}4, r{+}e{+}4, \dots, je{+}r{-}s)\\
= & \; 0 \text{ by induction if $e=3$.}
\end{align*}
If $e>3$, then we have
\begin{align*}
& \psid{r+e}{r+4} \chaind{r+e+1}{r+5} \chaind{r+e+2}{r+6} \psid{r-e}{s+1}\chaind{r-e+1}{s+2} \clapdots \chaind{r+1}{s+e+2} \cancel{(\psi_{r+3}^2(s+2,s))} \psid{r+2}{s+e+3} \psid{r+e+3}{s+e+4} \chaind{r+e+4}{s+e+5} \clapdots \chaind{je+r-s-1}{je} z_\la\\
= & \; \psid{r+e}{r+4} \chaind{r+e+1}{r+5} \psid{r+e+2}{r+7} \chaind{r+e+3}{r+8} \psid{r-e}{s+1}\chaind{r-e+1}{s+2} \clapdots \chaind{r+2}{s+e+3} \cancel{\psi_{r+6}\psi_{r+7} \psi_{r+6}} \psid{r+5}{s+e+4} \psid{r+e+4}{s+e+5} \chaind{r+e+5}{s+e+6} \clapdots \chaind{je+r-s-1}{je} z_\la\\
= & \; \psid{r+e}{r+4} \psid{r+e+1}{r+6} \chaind{r+e+2}{r+7} \chaind{r+e+3}{r+8} \psid{r-e}{s+1}\chaind{r-e+1}{s+2}\clapdots \chaind{r+2}{s+e+3} \cancel{(\psi_{r+5}^2(s+3,s+1))} \psid{r+4}{s+e+4} \psid{r+e+4}{s+e+5} \chaind{r+e+5}{s+e+6} \clapdots \chaind{je+r-s-1}{je} z_\la\\
= & \; \psid{r+e}{r+5} \chaind{r+e+1}{r+6} \psid{r+e+2}{r+7} \chaind{r+e+3}{r+8} \psid{r-e}{s+1} \chaind{r-e+1}{s+2} \chaind{r-e+2}{s+3} \chaind{r-e+3}{s+4}\\
& \cdot \psi_{r+4} v(1, \dots, s{+}4, r{-}e{+}5, r{-}e{+}6, \dots, r{+}3, r{+}5, r{+}e{+}5, r{+}e{+}6, \dots, je{+}r{-}s)\\
= & \; 0 \text{ by induction if $e=4$.}
\end{align*}
We continue for $e>4$ in a similar manner until we reach
\begin{align*}
& \psi_{r+e} \psid{r-e}{s+1} \chaind{r-e+1}{s+2} \clapdots \chaind{r+e-2}{s+2e-1} \psid{r+e}{s+2e} \psid{r+2e}{s+2e+1} \chaind{r+2e+1}{s+2e+2} \clapdots \chaind{je+r-s-1}{je} z_\la\\
= & \; \psi_{r+e} v(1, \dots, s, r{-}e{+}1, r{-}e{+}2, \dots, r{+}e{-}1, r{+}e{+}1, r{+}2e{+}1, r{+}2e{+}2, \dots, je{+}r{-}s)\\
= & \; 0 \text{ by induction.}\qedhere
\end{align*}
\end{enumerate}
\end{proof}

\begin{proof}[Proof of \cref{prop:cancellation}]
\begin{enumerate}[label=(\roman*)]
\item This in fact follows just like the proof of \cref{smallphivt}, with indices shifted by the corresponding multiples of $e$.
In fact, that proof gives that
\[
\psi_{(r+1)e-1} \psi_{(r+1)e-2} \dots \psi_{re} \Psi_r v = -2 \psi_{(r+1)e-1} \psi_{(r+1)e-2} \dots \psi_{re} v,
\]
since there we allow each component to be an arbitrary (small) hook, not just the trivial partition $(e)$.
If we follow the proof, setting $b=0$, it may be considerably shortened and in fact the prefix of generators $\psi_{(r+1)e-1} \psi_{(r+1)e-2} \dots \psi_{re+1}$ is not needed at all -- the special case $b=0$ of that proof ends on page~\pageref{b0pf}.

\item Without loss of generality, we will assume that $\Psi_{r+1} \Psi_r v$ is reduced, i.e.~$v$ is a linear combination of basis vectors indexed by standard tableaux that have brick $r$ in the second component and bricks $r+1$ and $r+2$ in the first.
One can show that if brick $r$ is in the first component, the calculation of $\Psi_r v_\ttt$ reduces to applying part~\ref{prop:cancellation1} of the proposition to basis vectors of the assumed form.
If brick $r$ is in the second component but bricks $r+1$ and $r+2$ are not both in the first component, then the calculation of $\Psi_{r+1} \Psi_r v_\ttt$ reduces to applying part~\ref{prop:cancellation3} of the proposition to basis vectors of the assumed form.

By repeatedly applying~\cref{cor:psiaction}\ref{cor:psiaction1}, we have
\begin{align*}
& \psi_{re} \Psi_{r+1} \Psi_r v\\
= \; & \psi_{re}
\left(\psid{re+e}{re+1} \chaind{re+e+1}{re+2} \clapdots \chaind{re+2e-1}{re+e}\right)
\left(\psid{re}{re-e+1} \chaind{re+1}{re-e+2} \clapdots \chaind{re+e-1}{re}\right) v\\
= \; & \psi_{re} \psid{re+e}{re-e+1} \chaind{re+e+1}{re-e+2} \clapdots \chaind{re+2e-1}{re} v\\
= \; & \psid{re+e}{re+2} \cancel{\psi_{re} \psi_{re+1} \psi_{re}} \psid{re-1}{re-e+1} \psid{re+e+1}{re-e+2} \chaind{re+e+2}{re-e+3} \clapdots \chaind{re+2e-1}{re} v\\
= \; & \psid{re+e}{re+3} \psid{re-1}{re-e+1} \psid{re+e+1}{re+4} \cancel{\psi_{re+2} \psi_{re+3} \psi_{re+2}}
\psid{re+1}{re-e+2} \psid{re+e+2}{re-e+3} \chaind{re+e+3}{re-e+4} \clapdots \chaind{re+2e-1}{re} v\\
\vdots \; &\\
= \; & \left(\psid{re+e}{re+3} \psid{re-1}{re-e+1}\right)
\left(\psid{re+e+1}{re+5} \psid{re+1}{re-e+2}\right)
\left(\psid{re+e+2}{re+7} \psid{re+3}{re-e+3}\right) \dots
\left(\psid{re+2e-4}{re+2e-5} \psid{re+2e-9}{re-3}\right) \cdot\\
\; & \left(\psi_{re+2e-3} \psid{re+2e-7}{re-2}\right)
\left(\psid{re+2e-5}{re-1}\right)
\cancel{\psi_{re+2e-2} \psi_{re+2e-1} \psi_{re+2e-2}} \psid{re+2e-3}{re} v\\
= \; & \left(\psid{re+e}{re+3} \psid{re+e+1}{re+5} \psid{re+e+2}{re+7}\dots \psid{re+2e-4}{re+2e-5} \psi_{re+2e-3}\right) \psid{re-1}{re-e+1} \psid{re+1}{re-e+2}\dots \psid{re+2e-5}{re-1} \psid{re+2e-3}{re} v.
\end{align*}
If $e=3$, this becomes $\psi_{3r+3} \psid{3r-1}{3r-2} \psid{3r+1}{3r-1} \psid{3r+3}{3r}v$.
However, if $e>3$, then by applying \cref{lem:psipsi}, we have
\begin{align*}
& \; \left(\psid{re+e}{re+3} \psid{re+e+1}{re+5} \psid{re+e+2}{re+7} \dots \psid{re+2e-4}{re+2e-5} \right)
\psid{re-1}{re-e+1} \psid{re+1}{re-e+2} \dots \psid{re+2e-5}{re-1}
\cancel{\psi_{re+2e-3}^2}
\left(\psid{re+2e-4}{re} v \right)\\
= & \; \left(\psid{re+e}{re+3} \psid{re+e+1}{re+5} \dots \psid{re+2e-5}{re+2e-7} \right)
\psid{re-1}{re-e+1} \psid{re+1}{re-e+2} \dots
\\%
&\quad \dots
\psid{re+2e-7}{re-2}
\cancel{\psid{re+2e-4}{re+2e-5}\psiu{re+2e-5}{re+2e-4}}
\left(\psid{re+2e-6}{re-1} \chaind{re+2e-5}{re} \right) v\\
= & \; \left(\psid{re+e}{re+3} \psid{re+e+1}{re+5} \dots \psid{re+2e-6}{re+2e-9} \right)
\psid{re-1}{re-e+1} \psid{re+1}{re-e+2} \dots\\
&\quad \dots 
\psid{re+2e-9}{re-3}
\cancel{\psid{re+2e-5}{re+2e-7} \psiu{re+2e-7}{re+2e-5}}
\left(\psid{re+2e-8}{re-2} \chaind{re+2e-7}{re-1} \chaind{re+2e-6}{re} \right) v\\
\vdots \; &\\
= & \; \psi_{re+e} \psid{re-1}{re-e+1}\psid{re+1}{re-e+2}
\cancel{\psid{re+e-1}{re+3}\psiu{re+3}{re+e-1}}
\left(\psid{re+2}{re-e+3} \chaind{re+3}{re-e+4} \clapdots \chaind{re+e-2}{re-1}\right)
\psid{re+e}{re} v.
\end{align*}
For $e\geq 3$, the last terms become
\begin{align*}
& \; \psid{re-1}{re-e+1} \psid{re+1}{re-e+2} \chaind{re+2}{re-e+3} \clapdots \chaind{re+e-2}{re-1}
(\psi_{re+e}^2(-1,0))
\psid{re+e-1}{re}v\\
= & \; \psid{re-1}{re-e+1} \psid{re+1}{re-e+2} \chaind{re+2}{re-e+3} \clapdots \chaind{re+e-2}{re-1}
(y_{re+e} - y_{re+e+1})
\psid{re+e-1}{re}v.
\end{align*}
We know from \cref{lem:ys} that the second term becomes zero, whilst the first term is
\begin{align*}
& \; \psid{re-1}{re-e+1} \psid{re+1}{re-e+2} \chaind{re+2}{re-e+3} \clapdots \chaind{re+e-2}{re-1}
(y_{re+e}\psi_{re+e-1}(-1,-1))
\psid{re+e-2}{re} v\\
= & \; \psid{re-1}{re-e+1} \psid{re+1}{re-e+2} \chaind{re+2}{re-e+3} \clapdots \chaind{re+e-2}{re-1}
(\psi_{re+e-1}y_{re+e-1}+1)
\psid{re+e-2}{re} v.
\end{align*}
Now the first term is
\begin{align*}
&\psid{re-1}{re-e+1} \psid{re+1}{re-e+2}
\chaind{re+2}{re-e+3} \clapdots \chaind{re+e-2}{re-1} \psi_{re+e-1} y_{re+e-1} v(1, \dots, re-1, re+e-1, a_{re+1}, \dots, a_{re})\\
&= 0 \text{ by \cref{lem:newyaction}\ref{lem:yaction1} since $re+e-1 \not\equiv 0,1 \pmod{e}$.}
\end{align*}
If $e=3$, then applying \cref{cor:psiaction}\ref{cor:psiaction2} to the second term yields
\[
\psid{3r-1}{3r-2} \psid{3r+1}{3r-1} \psid{3r+1}{3r} v
= \psid{3r-1}{3r-2}
\cancel{\psi_{3r+1} \psi_{3r} \psi_{3r+1}}
\psi_{3r-1} \psi_{3r} v
= \cancel{\psi_{3r-1} \psi_{3r-2} \psi_{3r-1}} \psi_{3r} v
= \psi_{3r} v.
\]
If $e>3$, then repeatedly applying \cref{cor:psidownaction} to the second term yields
\begin{align*}
& \; \psid{re-1}{re-e+1} \psid{re+1}{re-e+2} \chaind{re+2}{re-e+3} \clapdots \chaind{re+e-3}{re-2} \cancel{\psi_{re+e-2} \psiu{re+e-3}{re+e-2}}
 \psid{re+e-4}{re-1} \chaind{re+e-3}{re} v\\
= & \; \psid{re-1}{re-e+1} \psid{re+1}{re-e+2} \chaind{re+2}{re-e+3} \clapdots 
\chaind{re+e-4}{re-3}
\cancel{\psid{re+e-3}{re+e-4} \psiu{re+e-5}{re+e-3}}
 \psid{re+e-6}{re-2} \chaind{re+e-5}{re-1} \chaind{re+e-4}{re} v\\
\vdots \; &\\
= & \; \psid{re-1}{re-e+1} 
\cancel{\psid{re+1}{re-e+4}\psiu{re-e+3}{re+1}}
\psiu{re-e+2}{re} v \\
= & \; \cancel{\psid{re-1}{re-e+2} \psiu{re-e+1}{re-1}} \psi_{re} v\\
= & \; \psi_{re} v.
\end{align*}

\item
The proof proceeds analogously to part \ref{prop:cancellation2}, and is omitted for the sake of brevity.\qedhere
\end{enumerate}
\end{proof}

\bibliographystyle{amsalpha}  
\phantomsection
\addcontentsline{toc}{section}{\refname}
\bibliography{master}

\end{document}